\documentclass[english,reqno,11pt]{amsart}
\usepackage{amsmath, amsthm, amsfonts}
\usepackage{hyperref} 
\usepackage{hyperref}
\hypersetup{
	colorlinks=true,
		citecolor=blue!60!black,
		linkcolor=red!60!black,
		urlcolor=green!40!black,
		filecolor=yellow!50!black,
	breaklinks=true,
	pdfpagemode=UseNone,
	bookmarksopen=false,
}
\usepackage{tcolorbox}
\usepackage{amsmath,amsthm,amssymb}
\usepackage{amscd,indentfirst,epsfig}
\usepackage{latexsym}
\usepackage{times}
\usepackage{enumerate}
\usepackage{mathrsfs}
\usepackage{stmaryrd}
\usepackage{amsopn}
\usepackage{amsmath}
\usepackage{amssymb,dsfont,mathtools}
\usepackage{amsfonts,bm}
\usepackage{amsbsy,amsmath}
\usepackage{amscd}
\usepackage{xcolor}
\usepackage{mathtools}

\DeclarePairedDelimiter\floor{\lfloor}{\rfloor}
\usepackage[mono=false]{libertine}
\usepackage[T1]{fontenc}
\usepackage{amsthm}
\usepackage[normalem]{ulem} 
\usepackage[cal=euler, scr=boondoxo]{}
\usepackage{microtype}

\usepackage{numprint}

\newcommand{\verti}[1]{{\left\vert\kern-0.25ex\left\vert\kern-0.25ex\left\vert #1 
    \right\vert\kern-0.25ex\right\vert\kern-0.25ex\right\vert}}
\newtheorem{theo}{Theorem}[section]

\newtheorem{lemme}[theo]{Lemma}
\newtheorem{propo}[theo]{Proposition}
\newtheorem{cor}[theo]{Corollary}
\newtheorem{hyp}[theo]{Assumption}

\newtheorem{nb}[theo]{Remark}

\def \bq {\begin{equation}}
\def \eq {\end{equation}}
\def \leq {\leqslant}
\def \geq {\geqslant}
\def \v {v}
\def \N {\mathbb{N}}
\def \ind {\mathbf{1}}
\def \Z {\mathbb{Z}}
\def \K {\mathcal{K}}
\def \S {\mathbb{S}}
\numberwithin{equation}{section}
\def \A {\mathcal{A}}

\def \Ms {\mathsf{M}}
\def \Gs {\mathsf{G}} 

\def \d {\mathrm{d}}
\def \D {\mathscr{D}}

\def \C {\mathbb{C}}
\def \Rs {\mathcal{R}}

\def \R {\mathbb{R}}

\def \X {\mathbb{X}}

\def \l+ {L^1_+}
\def \l- {L^1_-}

\renewcommand{\epsilon}{\varepsilon}

\def \ds {\displaystyle}
\def \l {\lambda}
\def \T {\mathbb{T}}
\def \B {\mathscr{B}}
\def \e {\varepsilon}

\def \H {\mathsf{H}}
\def \O {\mathcal{D}}

 \def \Ms {\mathsf{M}}

\def \O {\T^{d}}

\def \v {{v}}

\def \w {{w}}

\begin{document}
\title[Collisional kinetic equation]{Convergence rate to equilibrium for conservative scattering models on the torus: a new tauberian approach}

 \author{B. Lods}

 \address{Universit\`{a} degli
 Studi di Torino \& Collegio Carlo Alberto, Department of Economics and Statistics, Corso Unione Sovietica, 218/bis, 10134 Torino, Italy.}\email{bertrand.lods@unito.it}

 \author{M. Mokhtar-Kharroubi}

 \address{Universit\'e de Bourgogne-Franche-Comt\'e, Equipe de Math\'ematiques, CNRS UMR 6623, 16, route de Gray, 25030 Besan\c con Cedex, France
}
\email{mustapha.mokhtar-kharroubi@univ-fcomte.fr}

\maketitle
\begin{abstract}
The object of this paper is to provide a new and systematic tauberian approach to quantitative long time behaviour of $C_{0}$-semigroups $\left(\mathcal{V}(t)\right)_{t \geq0}$  in $L^{1}(\T^{d}\times \R^{d})$ governing conservative linear kinetic equations on the torus with general scattering kernel $\bm{k}(v,v')$ and degenerate (i.e. not bounded away from zero) collision frequency $\sigma(v)=\int_{\R^{d}}\bm{k}(v',v)\bm{m}(\d v')$, (with $\bm{m}(\d v)$ being absolutely continuous with respect to the Lebesgue measure).  We show in particular that
if $N_{0}$ is the maximal integer $s \geq0$ such that
$$\frac{1}{\sigma(\cdot)}\int_{\R^{d}}\bm{k}(\cdot,v)\sigma^{-s}(v)\bm{m}(\d v) \in L^{\infty}(\R^{d})$$
then, for initial datum $f$ such that $\ds\int_{\T^{d}\times \R^{d}}|f(x,v)|\sigma^{-N_{0}}(v)\d x\bm{m}(\d v) <\infty$ it holds
$$\left\|\mathcal{V}(t)f-\varrho_{f}\Psi\right\|_{L^{1}(\T^{d}\times \R^{d})}=\dfrac{{\epsilon}_{f}(t)}{(1+t)^{N_{0}-1}}, \qquad \varrho_{f}:= \int_{\R^{d}}f(x,v)\d x\bm{m}(\d v)$$
where $\Psi$ is the unique invariant density of  $\left(\mathcal{V}(t)\right)_{t \geq0}$ and $\lim_{t\to\infty}{\epsilon}_{f}(t)=0$. {We in particular provide a new criteria of the existence of invariant density.} The proof relies on the explicit computation of the time decay of each term of the Dyson-Phillips expansion of $\left(\mathcal{V}(t)\right)_{t \geq0}$ and on suitable smoothness and integrability properties of the trace on the imaginary axis of Laplace transform of remainders of large order of this Dyson-Phillips expansion. Our construction resorts also on  collective compactness arguments and provides various technical results of independent interest. Finally, as a by-product of our analysis, we derive essentially sharp ``subgeometric'' convergence rate for Markov semigroups associated to general transition kernels.


\end{abstract}

\tableofcontents
\section{Introduction}

The objet of this paper is to provide $L^{1}$--rates of convergence to equilibrium for
conservative linear kinetic equations of the form
\begin{equation}\label{1a}
\partial_{t}f(x,v,t) + v \cdot \nabla_{x}f(x,v,t)+\sigma(x,v)f(x,v,t)=\int_{V}k(x,v,v')f(t,x,v')\bm{m}(\d v'),\end{equation} 
for $(x,v) \in \T^{d} \times V,$ and $t \geq 0$ where
$$\sigma(x,v)=\int_{V}k(x,v',v)\bm{m}(\d v'), \qquad (x,v) \in \T^{d} \times V.$$
Here $V \subset \R^{d}$ is the support of a nonnegative Borel measure $\bm{m}$  while $\T^{d}$ is the $d$-dimensional torus
$$\T^{d}=\R^{d}\big/\mathbb{Z}^{d}.$$
For simplicity, we will assume that the Lebesgue measure on the torus is normalized i.e. $|\T^{d}|=1$. 
 
 \subsection{Assumptions and main result}
This class of equation was dealt with in \cite{mkjfa,mmktore} for a general class of velocity measures $\bm{m}(\d v)$.  A key result in \cite{mkjfa} is that the semigroup governing \eqref{1a} has a spectral gap if and only if
$$\lim_{t\to\infty}\inf_{(x,v) \in \T^{d} \times V}\int_{0}^{t}\sigma(x+tv,v)\d t >0.$$
In this case, there exists automatically an invariant density and the latter is exponentially stable (i.e. the semigroup converges exponentially, in operator norm, to the spectral projection associated to the invariant density). The existence and the stability of an invariant density in the degenerate case
\begin{equation}\label{eq:dege}
\lim_{t\to\infty}\inf_{(x,v) \in \T^{d} \times V}\int_{0}^{t}\sigma(x+tv,v)\d t=0
\end{equation}
are dealt with systematically in \cite{mmktore}. The stability of the invariant density (i.e. the strong convergence of the semigroup to its ergodic projection) is \emph{not quantified} and is obtained either by means of general results on partially integral semigroups \cite{rudnicki} or by means of a $0-1$ law for semigroups \cite{ergodic}. We provide also in Remark \ref{nb:ingham}  below a third approach via Ingham tauberian theorem.\medskip

Our object here, in a continuation of \cite{mmktore}, is to provide rates of convergence to equilibrium in the spirit of our recent construction on collisionless kinetic semigroups with boundary operators \cite{MKL-JFA}. To this end, we restrict ourselves to space homogeneous scattering kernel
$$k(x,v,v')=\bm{k}(v,v')$$
(and consequently $\sigma(x,v)=\sigma(v)$) where the degeneracy condition \eqref{eq:dege} amounts to
\begin{equation}\label{eq:dege1}
\inf_{v\in V}\sigma(v)=0.\end{equation}
The non homogeneous case is left open even if we suspect that a similar, albeit much more technical, construction is possible in that case. We also assume that
\begin{equation}\label{eq:bounded}
\sigma \in L^{\infty}(V) \qquad \text{ and } \qquad \sigma(v)=\int_{V}\bm{k}(v',v)\bm{m}(\d v'), \qquad v \in V.
\end{equation}
We will see in Assumption \ref{hypK} that we willrestrict ourselves to the case in which $\bm{m}$ is \emph{absolutely continuous with respect to the Lebesgue measure}
$$\bm{m}(\d v) =\bm{m}(v)\d v$$
for some nonnegative weight function $\bm{m}\::\:V \to \R^{+}$ satisfying some technical regularity assumption (see \eqref{eq:logM} for details). The specific  nature of the Lebesgue measure is used only once (in the proof of Lemma \ref{lem:lem37}) and could be avoided at the cost of more technical calculations. We have not tried to elaborate on this point here because the whole construction given in the paper is already quite involved. We however insist on the fact that the choice of the Lebesgue measure seems to be only technical.
   We denote by 
$$\X_{0}:=L^{1}(\T^{d} \times V\,,\,\d x\otimes \bm{m}(\d v))$$
endowed with its usual norm $\|\cdot\|_{\X_{0}}.$ More generally, for any $s \in \R$, we set
$$\X_{s}:=L^{1}(\T^{d} \times V\,,\,\max(1,\sigma(v)^{-s})\d x \otimes \bm{m}(\d v))$$
with norm $\|\cdot\|_{\X_{s}}.$
Notice that the absorption semigroup $\left(U_{0}(t)\right)_{t\geq 0}$ given by
\begin{equation}\label{eq:U0t}
U_{0}(t)f(x,v)=\exp\left(-\sigma(v)\,t\right)f(x-tv,v), \qquad t\geq0, f \in \X_{0}\end{equation}
has zero type in $\X_{0}$:
$$\omega(U_{0})=0$$
under the degeneracy condition \eqref{eq:dege1}.
The generator of $\left(U_{0}(t)\right)_{t\geq0}$ is given by
$$\A f(x,v)=-v \cdot \nabla_{x} f(x,v) - \sigma(v)f(x,v), \qquad f \in \D(\A)=\left\{f \in \X_{0}\;;\;v \cdot \nabla_{x} f \in \X_{0}\right\}.$$
We introduce the operator, acting in the $v$-variable only,
\begin{equation}\label{eq:Hhkernel}
\mathcal{K}f(v)=\int_{V}\bm{k}(v,v')f(v')\bm{m}(\d v'), \qquad f \in L^{1}(V)=L^{1}(V,\d\bm{m}).\end{equation}
Due to \eqref{eq:bounded}, one sees first that
$$\mathcal{K} \in \mathscr{B}(\X_{0}), \qquad \qquad \|\K\|_{\mathscr{B}(\X_{0})} \leq \|\sigma\|_{\infty}$$
and also that
$$ \K \in \mathscr{B}(\X_{-1},\X_{0}), \qquad \qquad \|\K\|_{\mathscr{B}(\X_{-1},\X_{0})}=1.$$
 The fact that $\K$ is a bounded operator in $\X_{0}$ implies that $\A + \K$ is the generator of a $C_{0}$-semigroup $\left(\mathcal{V}(t)\right)_{t\geq0}$ in $\X_{0}$ given by
$$\mathcal{V}(t)=\sum_{n=0}^{\infty}U_{n}(t), \qquad t \geq 0$$
where
$$U_{n+1}(t)=\int_{0}^{t}U_{n}(t-s)\K U_{0}(s)\d s = \int_{0}^{t}U_{0}(t-s)\K U_{n}(t-s)\d s, \qquad n \in \N,\qquad t \geq0.$$
Introduce the following notation, for any $s \in \R$
\begin{equation}\label{eq:varthetas}
\vartheta_{s}(w):=\frac{1}{\sigma(w)}\int_{V}\sigma^{-s}(v)\bm{k}(v,w)\bm{m}(\d v), \qquad \qquad w \in V.\end{equation}

{The results of the present paper are based upon several sets of Assumptions. The first \emph{fundamental assumptions} are the following which are at the basis of the underlying method:}
\begin{hyp}\label{hypH}
Assume that $\K\::\:L^{1}(V) \to L^{1}(V)$ is a \emph{weakly compact}  operator of the form \eqref{eq:Hhkernel} which satisfies the following
\begin{enumerate}[(1)]
\item For any $v \in V$
\begin{equation}\label{eq:conservative}
\sigma(v)=\int_{V}\bm{k}(w,v)\bm{m}(\d w),\end{equation}
with $\sigma \in L^{\infty}(V)$ and
\begin{equation}\label{spahomo}
\inf_{v}\sigma(v)=0.
\end{equation}
%Typically, if $V$ is unbounded, we may need to replace \eqref{spahomo} with 
%$$\liminf_{|v|\to\infty}\sigma(v)=0.$$
\item There exists some (maximal) integer $N_{0} \geq 1$ such that
\begin{equation}\label{eq:varthetaN0}
\vartheta_{N_{0}} \in L^{\infty}(V).\end{equation}
\item Introducing, for any $\delta >0$ the set
$$\Sigma_{\delta}=\{v \in V\;;\;\sigma(v) \leq \delta\},$$
we assume that
\begin{equation}\label{eq:mue}
\lim_{\delta \to 0^{+}}\sup_{w \in V}\frac{1}{\sigma(w)}\int_{\Sigma_{\delta}}\bm{k}(v,w)\bm{m}(\d v)=0\,\end{equation}
%and
%\begin{equation}\label{eq:muemmk}
%\lim_{\delta \to 0^{+}}\sup_{w \in V}\int_{\Sigma_{\delta}}\bm{k}(v,w)\frac{\bm{m}(\d v)}{\sigma(v)}=0\,.\end{equation}
\end{enumerate}
\end{hyp}

{These assumptions provide actually a new practical criteria ensuring the existence (and uniqueness) of an invariant density:}
\begin{theo}\label{theo:main-invar} Assume that $\K$ is an irreducible operator {satisfying Assumptions \ref{hypH} and  the measure $\bm{m}$ is such that that there exists $\alpha >0$ such that, for any bounded set $S \subset V$, there is $c(S) >0$ such that
\begin{equation}\label{eq:hypm}
\sup_{\nu \in \S^{d-1}}\bm{m}\otimes \bm{m}\big(\left\{(v,w) \in S\times S\;;\;|(v-w)\cdot \nu| < \e\right\}\big) \leq c(S)\e^{\alpha}, \qquad \forall \e >0.\end{equation}}
Then, there exists a unique ${\Psi} \in \D(\A)$ \emph{spatially homogeneous} with 
$$\Psi(v) >0, \qquad \int_{\O\times V}\Psi(x,v)\d x\otimes \bm{m}(\d v)=1$$
such that
$$\left(\A+\K\right)\Psi=0=\K\Psi.$$
Moreover, $\Psi \in \X_{N_{0}-1}$.\end{theo}
\begin{nb} We recall that $\K$ is irreducible if there exists  no non trivial $\Omega \subset \T^{d}  \times V$  such that $\K$ leaves invariant $L^{1}(\Omega)$ which is identified to the closed subspace of $\X_{0}$ of functions vanishing a.e. outside $\Omega$. Practical criterion ensuring the irreducibility of $\K$ is given in \cite[Proposition 7]{mmktore}. In particular, $\K$ is irreducible if $\bm{k}(v,w) >0$ for $\bm{m}\otimes\bm{m}$-a.e. $(v,w) \in V\times V$. Notice that the existence and uniqueness of a steady solution has been obtained, under \emph{different assumptions} in  \cite{mmktore}. The approach followed here is  technically different from \cite{mmktore} and, as said, resort to different assumptions (see Proposition \ref{prop:varphi0} for details).\end{nb} 
\begin{nb} Notice that, if $\bm{k}(\cdot,\cdot)$ satisfies a \emph{detailed balance} condition, i.e. there exists a positive spatially homogeneous density $\mathcal{M}=\mathcal{M}(v)$, $\mathcal{M} \in L^{1}(V)$ such that
$$\bm{k}(v,w)\mathcal{M}(w)=\mathcal{M}(v)\bm{k}(v,w), \qquad \forall v,w \in V$$
then, up to a normalisation factor, $\Psi=\mathcal{M}$ is an invariant density and assumption \eqref{eq:mue} is not needed for our analysis. {Of course, assumption \eqref{eq:hypm} is satisfied if $\bm{m}$ is absolutely continuous with respect to the Lebesgue measure over $\R^{d}$ which is the framework we will further adopt in the paper.}
\end{nb}

A second set of Assumptions, most of technical nature, is the following
\begin{hyp}\label{hypK}
The measure $\bm{m}(\d v)$ is absolutely continuous with respect to the Lebesgue measure 
$$\bm{m}(\d v)=\bm{m}(v)\d v$$
for some  weight function $\bm{m}$ such that 
\begin{equation}\label{eq:logM}
\sup_{v\in V}\left|v \cdot \nabla_{v}\log \bm{m}(v)\right| < \infty.\end{equation}
Moreover, the kernel $\bm{k}(v,v')$ is such that there exist two positive constants $C_{1},C_{2} >0$ such that
\begin{equation}\label{eq:K2}
\int_{V}\left|w\cdot \nabla_{w}\bm{k}(v,w)\right|\max\left(1,\sigma^{-1}(v)\right)\bm{m}(v)\d v \leq C_{1}\sigma(w) \qquad \forall w \in V
\end{equation}
and
\begin{equation}\label{eq:K3}
\int_{V}\left|v \cdot \nabla_{v}\bm{k}(v,w)\right| \bm{m}(v)\d v \leq C_{2}\sigma(\w) \qquad \forall \w\in V.
\end{equation}
\end{hyp}

\begin{nb} We will comment in Subsection \ref{sec:ass} below on this set of assumptions as well as to the subsequent Assumption \ref{hypH}. We only mention here that property \eqref{eq:logM} is satisfied for instance for weight functions of the form
$$\bm{m}(v)=\left(1+|v|^{2}\right)^{\frac{s}{2}}, \qquad s \geq0.$$
\end{nb}
Our main result is the following

\begin{theo}\label{theo:maindec}
Under Assumptions \ref{hypH} and \ref{hypK}, if $\left(\mathcal{V}(t)\right)_{t\geq0}$ is an irreducible semigroup then for any $f \in \X_{N_{0}}$ there exist a constant $C_{f} > 0$ and 
$$\mathsf{\Theta}_{f} \in \mathscr{C}_{0}(\R,\X_{0}) \cap L^{1}(\R,\X_{0})$$
such that
\begin{equation}\label{eq:rat}
\left\|\mathcal{V}(t)f-\varrho_{f}\Psi\right\|_{\X_{0}} \leq \frac{C_{f}}{(1+t)^{N_{0}-1}} \bm{\epsilon}(t) \qquad \forall t\geq 0,
\end{equation}
where $\varrho_{f}:=\ds\int_{\T^{d}\times V}f(x,v)\d x\bm{m}(\d v)$ and
\begin{equation}\label{eq:defepst}
\bm{\epsilon}(t)=\frac{1}{1+t}+\left\|\int_{-\infty}^{\infty}\exp\left(i\eta\,t\right)\mathsf{\Theta}_{f}(\eta)\d \eta\right\|_{\X_{0}}\,,\, \qquad \lim_{t\to \infty}\bm{\epsilon}(t)=0.\end{equation}
%is such that $\lim_{t\to \infty}\bm{\epsilon}(t)=0.$ 
Moreover, for any $\mathsf{p} >4$, 
%If we assume moreover that there is $C(\mathsf{p}) > 0$ and $\beta >0$ such that
%\begin{equation}\label{eq:decay-power}
%\int_{|\eta| >R}\left\|\Ms_{i\eta}^{\mathsf{p}}\right\|_{\mathscr{B}(\X_{0})} \d \eta \leq \frac{C(\mathsf{p})}{R^{\beta}}, \qquad \forall R >0\end{equation}
%then,  
there is some positive constant $K=K(\mathsf{p}) >0$  such that 
\begin{equation}\label{eq:estMod}
\left\|\int_{-\infty}^{\infty}\exp\left(i\eta\,t\right)\mathsf{\Theta}_{f}(\eta)\d \eta\right\|_{\X_{0}}\leq K\,\left(\omega_{f}\left(\frac{\pi}{t}\right)\right)^{\frac{\mathsf{p}-4}{\mathsf{p}}} \qquad \forall t \geq 1\end{equation}
where $\omega_{f}\::\:\R^{+} \to \R^{+}$ denotes the \emph{minimal modulus  of uniform continuity} of the  mapping $\mathsf{\Theta}_{f}$.\end{theo}

\begin{nb} Notice that $(\mathcal{V}(t))_{t\geq0}$ is irreducible if  there is no invariant subspace $L^{1}(\Omega)$ of $\X_{0}$ left invariant by $\mathcal{V}(t)$ for any $t\geq0.$  One can prove that if $\K$ is irreducible then so is $\left(\mathcal{V}(t)\right)_{t\geq0}$ (see \cite[Proposition 7]{mmktore}).\end{nb}

\begin{nb} Recall that, in Assumption \ref{hypH}, we assumed $N_{0} \geq 1$ to be an integer. Without such an assumption, i.e if $N_{0}=\floor{N_{0}}+\alpha$, $\alpha \in (0,1)$, our main decay rate will then read
\begin{equation*}
\left\|\mathcal{V}(t)f-\varrho_{f}\Psi\right\|_{\X_{0}} \leq \frac{C_{f}}{(1+t)^{\floor{N_{0}}-1}} \bm{\epsilon}(t) \qquad \quad f \in \X_{N_{0}}.
\end{equation*}
In that case,  we believe that, for concrete examples of collision kernel $\bm{k}(v,v')$, it should be possible to explicit $\bm{\epsilon}(t)$ through an identification of the modulus of continuity of $\mathsf{\Theta}_{f}$ in terms of the non-integer part $\alpha=N_{0}-\floor{N_{0}]}.$
\end{nb}

Here above and in all the sequel, for any Banach space $(X,\|\cdot\|_{X})$ and any $k \in \N$, we set
\begin{multline*}
\mathscr{C}_{0}^k(\R,X)=\left\{h\::\:\R \to X\;;\;\text{of class $\mathscr{C}^{k}$  over $\R$ }\right.\\
\left.\text{and such that } \lim_{|\eta|\to\infty}\left\|\frac{\d^j}{\d\eta^j}h(\eta)\right\|_{X}=0 \qquad \forall j \leq k\right\}\end{multline*}
and we endow $\mathscr{C}_{0}^k(\R,X)$ with the norm 
$$\|h\|_{\mathscr{C}_{0}^k(\R,X)}:=\max_{0\leq j \leq k}\sup_{\eta}\left\|\frac{\d^j}{\d\eta^j}h(\eta)\right\|_{X}$$ which makes it a Banach space. We of course adopt the notation $\mathscr{C}_{0}(\R,X)=\mathscr{C}_{0}^{0}(\R,X)$.\medskip

{The above main result of the paper provides an explicit decay of the solution to \eqref{1a}. We strongly believe however that the interest of the present paper goes far beyond the mere convergence rate but it paves the way to a general abstract tauberian approach to the convergence rate of perturbed stochastic semigroup \cite{progr}. Moreover, because of the use of several fine collective compactness results and decay of Dyson-Phillips iterates, the paper contains several intermediate results of fundamental interest.}

\medskip

We can already mention that the function $\mathsf{\Theta}_{f}(\eta)$ appearing in \eqref{eq:estMod} is related to suitable derivatives of the trace along the imaginary axis $\l=i\eta$ $(\eta \in\R)$ of  the Laplace transform of suitable (large order) reminder of the Dyson-Phillips series defining the semigroup $\left(\mathcal{V}(t)\right)_{t\geq0}$. See Theorem \ref{theo:MainLaplace} and its proof in Section \ref{sec:Main} for a more precise statement.

It is important to mention that, as a direct by-product of our construction, our analysis covers also the important case in which the initial datum $f_{0}$ is independent of $x$. In that case, the general solution $f(x,v,t)=f(v,t)$ is also independent of $x$ and satisfies the \emph{spatially homogenous} version of \eqref{1a} which can be rewritten as
\begin{equation}\label{eq:1aSH}
\partial_{t}f(v,t)=\int_{V}\left[\bm{k}(v,v')f(t,v')-\bm{k}(v',v)f(t,v)\right]\bm{m}(\d v')\,.\end{equation} 
Models governed by eq. \eqref{eq:1aSH} are ubiquitous in the study of Markov processes and  results like our main Theorem \ref{theo:maindec} provide in this case an estimate of the decay rate for the transition probability semigroup of a continuous time Markov jump process. For such jump processes, roughly speaking, positive lower bound on the transition kernel $\bm{k}(v,v')$ induces exponential convergence of the stochastic processes (this case is refer to as the ``geometric case'' in the study of Markov processes) whereas our degeneracy condition \eqref{spahomo} prevents such an exponential convergence. Our result provides therefore a (seemingly sharp) ``subgeometric'' convergence rate for these kinds of processes. Our construction is quite involved and relies on various technical results of independent interest. As far as we know, most of our results are new and appear here for the first time.

\subsection{About Assumptions \ref{hypH} and \ref{hypK}} \label{sec:ass} Let us comment a bit about Assumptions \ref{hypH}--\ref{hypK}, referring the reader to the subsequent Section \ref{sec:Main} for more details on that matter. We first observe that the conservative assumption \eqref{eq:conservative} is natural to deal with Markov semigroups whereas, as said already, the degeneracy condition \eqref{spahomo} is the one which prevents the existence of a spectral gap. 

As illustrated by the above Theorem \ref{theo:maindec}, the decay rate   is \emph{prescribed} by the maximal gain of integrability that the boundary operator is able to provide through the function $\vartheta_{N_{0}}$ where we notice that
\begin{equation}\label{eq:KXN}
\mathcal{K} \in \mathscr{B}(\X_{-1},\X_{N_{0}}), \qquad \|\K\|_{\mathscr{B}(\X_{-1},\X_{N_{0}})} \leq \|\vartheta_{N_{0}}\|_{\infty}.\end{equation}
This illustrates the fundamental role of \eqref{eq:varthetaN0} in Assumption \ref{hypH}. The fact that we assume here $N_{0}$ to be an integer is an artefact of the approach we follow since, as established in Theorem \ref{theo:MainLaplace}, $N_{0}-1$ is also the maximal regularity of the trace function $\Upsilon_{n}(\eta)f$ we can derive for $f \in \X_{N_{0}}$. 

We already pointed out that a consequence of Assumptions \ref{hypH} concerns the existence of an invariant density $\Psi$ in Theorem \ref{theo:main-invar} since it allows to apply the results from \cite{mmktore}. Moreover, \eqref{eq:mue}   is the cornerstone hypothesis to establish that there exists $\mathsf{q} >0$ such that 
\begin{equation}\label{eq:coll}
\left\{\left[\K\Rs(\l,\A)\right]^{\mathsf{q}}\;;\;0 \leq \mathrm{Re}\l \leq 1\right\} \subset \B(\X_{0}) \text{ is collectively compact} \end{equation} in Theorem \ref{theo:main-assum}.
{All these consequences of Assumptions \ref{hypH} are the fundamental brick on which we build our theory. The role of Assumptions \ref{hypK}, on the contrary, is more of technical nature. Indeed, Assumptions \ref{hypK}  are technical requirements} which allow to deduce the decay rate of iterates of  $\K\Rs(\l,\A)$ on the imaginary axis with respect to $|\mathrm{Im}\l\,|$. We refer to Theorem \ref{theo:main-assum} and especially \eqref{eq:decayP}  for a precise statement. In particular, under such an assumption, we point out (see \eqref{eq:power}) that
\begin{equation*} 
\int_{|\eta|>1}\left\|\left[\Rs(i\eta,\A)\K\right]^{\mathsf{p}}\right\|_{\mathscr{B}(\X_{0})} \d \eta < \infty\end{equation*}
for any $\mathsf{p} >4$ which is crucial for the estimate \eqref{eq:estMod}.

We wish to insist here on the fact that the assumption \eqref{eq:mue} is the one ensuring the above collective compactness \eqref{eq:coll} whereas \eqref{eq:K2}--\eqref{eq:logM} (together with the fact that $\bm{m}(\d v)$ is absolutely continuous w. r. t. the Lebesgue measure) are those assumptions which provide \emph{quantitative estimates} on the behaviour of iterates of $\K\Rs(\l,\A)$. In particular, all the results of the paper which resort  only on some \emph{qualitative} properties of $\K\Rs(\l,\A)$ remain valid \emph{without the assumptions \eqref{eq:K2}--\eqref{eq:logM}}. See for instance Theorem \ref{theo:P1conv} and Remark \ref{nb:ingham} for an example of such qualitative results.

%Notice that \eqref{eq:mue} implies that $\K \in \B(\X_{0})$ is a \emph{weakly compact} operator. 

\subsection{Related literature}  A very precise exposition of modern tools developed for the convergence to equilibrium and stability of  Markov processes is the monograph \cite{tweedie}. The bibliography about exponential convergence (geometric case) for such processes is too vast and, since our main purpose in the present work is rather the study of kinetic equation like \eqref{1a}, we refer the reader to \cite{jose-mischler} for a nice introduction to the field. We only mention here that such geometric convergence results usually resort to hypocoercivity results and Harris-type results (see for instance  \cite{canizo20} for an application of Harris-type techniques to the study of fragmentation equation) or to a careful spectral analysis of the associated semigroup (see \cite{banasiak,mmkfrag} for very recent application to fragmentation models). For subgeometric convergence to equilibrium, a somehow concurrent approach (well-adapted to nonlinear models) is the so-called entropy method which consists in quantifying the entropy dissipation properties of the collisional operato to deduce establish the algebraic rate of convergence of the entropy (this, in turn, provides a decay in the usual $L^{1}$-norm thanks to Csiszar-Kullback inequality). For linear model, this method has been applied in \cite{amit} for the linear Boltzmann equation or its relaxation model caricature \cite{caceres}.

For subgeometric convergence to equilibrium of Markov processes, Harris-type tools have been develop in the probabilistic community (see for instance \cite{douc, bernou}) and we refer again the reader to \cite{jose-mischler} for a thorough description of such results. Subgeometric convergence rates have also been studied for classical models as the Fokker-Planck
equation \cite{kavian} and, for spatially homogeneous linear Boltzmann equation in a previous contribution by the authors \cite{spatial}.

As far as spatially inhomogeneous kinetic equations are concerned, the question of estimating the speed of approach to equilibrium for a non-homogeneous, linear transport equation with degenerated total scattering cross-section has been considered mainly in the context of the linearized Boltzmann or Landau equation (see for instance \cite{caflish, kleber} to mention just a few relevant results). Spectral gap estimates via the so-called hypocoercivity method have been derived in a $L^{2}$-setting in \cite{mouhot,duan} while algebraic rate of convergence towards the equilibrium, still in the $L^{2}$ setting, has been established in \cite{desv} norm is established for a case in which the cross-section $\sigma$ is depending on $x$ and vanishes in some portion of the space. We also mention the recent contribution \cite{bouin} in which a decay similar to the one in Theorem \ref{theo:maindec} is obtained in a $L^{2}$ framework. 

For a purely $L^{1}$--approach, the literature on the field is scarcer. We mention the contribution \cite{mkjfa} and \cite{evans} which prove the existence of a spectral gap if $\sigma$ is bounded for below by means of spectral analysis and Harris-type results respectively.  Harris-type of results are actually providing subgeometric rate of convergence for linear Boltzmann equation with weak confining potentials in \cite{evans}. 

For ``subgeometric'' convergence to equilibrium for the degenerate linear kinetic equations \eqref{1a} on the torus, as far as we know, the only previous work providing results similar to those obtained in the present paper is  \cite{komo} in which $(V;\bm{m}(\d v))$ is a probability space (we change slightly the quick presentation of this work in order to compare it to ours).  The main assumption in \cite{komo} is then
\begin{equation}\label{eq:komo-a}\begin{cases}
\bm{k}(\cdot,\cdot) \in L^{\infty}(V \times V,\bm{m}\otimes\bm{m})\,,\\ 
\qquad \ds\int_{V}\bm{k}(v,v')\bm{m}(\d v')=\ds\int_{V}\bm{k}(v',v)\bm{m}(\d v'), \quad \bm{k}(v,v') \leq C\,\sigma(v)\sigma(v')\\
\ds \int_{V}\sigma^{-a}(v)\bm{m}(\d v) < \infty\end{cases}\end{equation}
for some (maximal) $a >0$. The analysis of \cite{komo}  is carried out in the space
$$\bm{W}_{a}:=\left\{f \in \X_{0}\;;\;\|f\|_{\bm{W}_{a}}:=\sum_{p\in \Z^{d}}\left\|\widehat{f}(p)\right\|_{L^{1}(\mu_{a})} < \infty\right\}$$
where, for any Fourier mode $p \in \Z^{d}$, $\widehat{f}(p)$ is the Fourier transform of $f$ in the $x$-variable
$$\widehat{f}(p,\cdot)=\int_{\T^{d}}f(x,\cdot)\exp\left(i\,x\cdot p\right)\d x, \qquad p \in \Z^{d}$$
and
$$\|\widehat{f}(p)\|_{L^{1}(\mu_{a})}=\int_{V}\left|\widehat{f}(p,v)\right|\sigma^{-a}(v)\bm{m}(\d v)\,,\,\quad p \in \Z^{d}.$$
Under assumption \eqref{eq:komo-a}, the main result in \cite{komo} is a convergence rate towards equilibrium of the type
\begin{equation}\label{eq:ratekomo}
\|\mathcal{V}(t)f-\varrho_{f}\ind_{\T^{d}\times V}\|_{\X_{0}} \leq \frac{c}{(1+t)^{a}}\|f\|_{\bm{W}_{a}}, \qquad t\geq0.\end{equation}
Notice that $\|f\|_{\bm{W}_{a}}$ is a combination of the Wiener algebra norm in space variable $x$ and the weighted norm with weight $\sigma^{-a}$ in velocity. Moreover, under assumption \eqref{eq:komo-a}, the unique steady equilibrium state is $\Psi=\ind_{\T^{d}\times V}$. The strategy of \cite{komo} is completely different from ours and is based upon some Fourier diagonalization of the transport operator $v\cdot \nabla_{x}$,  the inverse Laplace transform for each Fourier mode and the use of
the theory of Fredholm determinants to derive \eqref{eq:ratekomo}. However, the result obtained is comparable to ours. Note that the boundedness of $\bm{k}(\cdot,\cdot)$ and the finiteness of the measure $\bm{m}(\d v)$ implies that $K\::\:L^{1}(V  ) \to L^{1}(V )$ is \emph{weakly compact} whereas, under assumption \eqref{eq:komo-a}, one can check without difficulty that 
$$\vartheta_{a+1}(v)=\frac{1}{\sigma(v)}\int_{V}\sigma^{-a-1}(v')\bm{k}(v,v')\bm{m}(\d v') \leq C\int_{V}\sigma^{-a}(v')\bm{m}(\d v') < \infty$$
i.e. \eqref{eq:varthetaN0} holds true with $N_{0}=\floor{a}+1$. Our main result  gives then a decay rate like $(1+t)^{-N_{0}+1}=(1+t)^{-\floor{a}}$. Note that the second assumption in \eqref{eq:komo-a} implies \eqref{eq:mue} which means that \eqref{eq:komo-a} implies all our Assumptions \ref{hypH}. However, the construction in \cite{komo} is independent of our set of assumptions \ref{hypK}. Even though the class of measure $\bm{m}(\d v)$ is much general in \cite{komo} than in our presentation, our main result Theorem \ref{theo:maindec} provides a sharper rate of convergence rate if $a \in \N$: using $\mathrm{O-o}$ Landau's notation, our result improves the $\mathrm{O}\left((1+t)^{-a}\right)=\mathrm{O}\left(1+t)^{-N_{0}+1}\right)$ rate in \eqref{eq:ratekomo} into a 
$$\mathrm{o}\left((1+t)^{-N_{0}+1}\right)$$
rate. Moreover, our result applies to a broader class of functions since $\bm{W}_{a}$ is a proper subspace of $\X_{N_{0}-1}=\X_{a}$. We also point out that the diagonilisation Fourier procedure makes the approach in \cite{komo} difficult to adapt to the case in which $\bm{k}$ (and thus $\sigma$) depends on $x$ whereas our approach appears robust enough to allow to tackle this case.

We finally mention, on a different but related topic, the recent works \cite{bernou1,fournier} and our contribution \cite{MKL-JFA}
 consider linear transport equation for collisionless gas in which  the scattering occurs only on the boundary of a bounded region and the return to equilibrium is induced by the boundary conditions. More precisely, in \cite{bernou1,fournier}, some \emph{ad hoc} Harris-type results are tailored to treat such collisionless model whereas, in our contribution \cite{MKL-JFA}, a tauberian approach similar to the one we consider in the present work is devised. The results in \cite{MKL-JFA}
 served as an inspiration for the techniques developed in the present paper.

\subsection{Strategy}  The general strategy to prove Theorem \ref{theo:maindec} is explained in full details in Section \ref{sec:Main} in which the main steps are described. In a nutshell, we just mention here that our approach is Tauberian in essence since we will deduce the decay of the semigroup $\left(\mathcal{V}(t)\right)_{t\geq0}$ for some fine properties of its Laplace transform along the imaginary axis $\l=i\eta$. This approach combines in a robust and efficient way the so-called \emph{semigroup and resolvent approaches} for the study of the long-time behaviour of transport equation (see \cite{MMK} for a comprehensive description of the semigroup approach and \cite{mmk-mech} for a first account of the resolvent one). As said, inspired by our results in \cite{MKL-JFA}, we device here a method which combines the two approaches. In particular, using that the semigroup $\mathcal{V}(t)$ is given by a Dyson-Phillips series 
\begin{equation}\label{eq:DysPhil}
\mathcal{V}(t)=\sum_{n=0}^{\infty}U_{n}(t)\end{equation}
where
\begin{equation}\label{eq:UnK}
U_{n+1}(t)=\int_{0}^{t}U_{n}(t-s)\K\,U_{0}(s)\d s, \qquad t \geq0, \qquad n \in \N\,,\end{equation}
and 
we first establish, for $f \in \X_{N_{0}}$ a \emph{universal} decay of each of the iterated $U_{n}(t)f$ as $t \to \infty.$ Second, we carefully study the behaviour of some reminder of the above Dyson-Phillips expansion 
$$\bm{S}_{n+1}(t)=\sum_{k=n}^{\infty}U_{k}(t)$$
and shows that its Laplace transform
$$\mathscr{S}_{n+1}(\l)f=\int_{0}^{\infty}\exp\left(-\l t\right)\bm{S}_{n+1}(t)f\d t, \qquad \mathrm{Re}\l >0$$
 can be extended, for suitable class of functions $f$, up to the imaginary axis $\l=i\eta$ with moreover a nice decay of the Laplace transform as $|\eta| \to\infty.$ The existence of such a trace of $\mathscr{S}_{n+1}(\l)f$ is based upon some important collective compactness arguments, as introduced in \cite{anselone}, and the consequences of those compactness argument to the spectral theory of some of the operators defining $\mathscr{S}_{n+1}(\l)$ (see Section \ref{sec:Assum} for details). Using then the inverse Laplace transform, this allows to deduce our main decay estimate in Theorem \ref{theo:maindec}. As said already, this very rough description is expanded in the next Section \ref{sec:Main} which provides for a detailed description of the main technical difficulty as well as  the organization of the paper (see Section \ref{sec:orga}). 
 
 We point out here that, even though our approach is inspired by our previous contribution  \cite{MKL-JFA}, it differs from it in several technical and conceptual aspects. In particular, a crucial role in our analysis will be played by the family of operators
$$\Ms_{\l}:=\K\Rs(\l,\A), \qquad \mathrm{Re}\l >0$$
where $\Rs(\l,\A)$ is the resolvent of $\A$ which can be easily extended in a family of stochastic operators along the imaginary axis $\mathrm{Re}\l=0$ (see Section \ref{sec:reguTr} for details). One can easily get convinced here by direct computations that 
$$\lim_{\e\to0}\left\|\Ms_{\e+\eta}f-\Ms_{i\eta}f\right\|_{\X_{0}}=0 \qquad \forall f \in \X_{0}$$
but
$$\sup_{\e>0}\left\|\Ms_{\e+i\eta}-\Ms_{i\eta}\right\|_{\B(\X_{0})}>0.$$
Therefore, the mapping $\e \geq 0 \mapsto \Ms_{\e+i\eta}$ is strongly continuous but is not continuous in the operator norm topology whereas, in our construction in \cite{MKL-JFA}, the analogue of $\Ms_{\e+i\eta}$  was played by a family of boundary operators was continuous in the operator norm up to $\e=0.$ To be able to deduce spectral properties of $\Ms_{\l}$ from the sole strong continuity, we have to resort to several \emph{collective compactness results} proving in particular that
there exists $\mathsf{q} >0$ such that $$\left\{ \Ms_{\l}^{\mathsf{q}}\;;\;0 \leq \mathrm{Re}\l \leq 1\right\} \subset \B(\X_{0})$$
is collectively compact. The fact that we are dealing only with some \emph{iterate} of $\Ms_{\l}$ and not with $\Ms_{\l}$ itself prevents us to use \emph{directly} known functional analysis results linking the collective compactness and the strong convergence and forces us to tailor some specific extension of the results of \cite{anselone} here. This use of collective compactness is one of main difference between the the approach followed in \cite{MKL-JFA} and the one cooked up for the present contribution.

\section{General strategy and main results}\label{sec:Main}

Our general strategy to investigate the time-decay of $\mathcal{V}(t)f$ is based upon a general and robust approach which fully exploits semigroup and resolvent interplays.

\subsection{Preliminary facts} We recall that $(\A+\K,\D(\A))$ generates a $C_{0}$-semigroup $\left(\mathcal{V}(t)\right)_{t\geq0}$ 
given by \eqref{eq:DysPhil} and \eqref{eq:UnK} 
and the key basic observation is the decay of $U_{0}(t)$ on the hierarchy of spaces $\X_{k}$, namely
\begin{lemme}\label{lem:decayU0} %Assume that $\bm{m}(\d v)=\varpi(|v|)\d v$ with $\varpi(\cdot)$ radially symmetric and nonnegative. 
Given $k \geq0$, one has
\begin{equation}\label{eq:uo}
\left\|U_{0}(t)f\right\|_{\X_{0}} \leq \left(\frac{k}{e\,t}\right)^{k}\|f\|_{\X_{k}}, \qquad \forall t >0, \qquad f \in \X_{k}.\end{equation} Moreover, for any $f \in \X_{k+1}$, it holds
\begin{equation}\label{eq:uoK}
\int_{0}^{\infty}\left\|U_{0}(t)f\right\|_{\X_{k}}\d t \leq \|f\|_{\X_{k+1}} \qquad \text{ and } \qquad \int_{0}^{\infty}t^{k}\left\|U_{0}(t)f\right\|_{\X_{0}}\d t \leq \Gamma(k+1)\|f\|_{\X_{k+1}}\end{equation}
where $\Gamma(\cdot)$ is the usual Gamma function.
\end{lemme}
\begin{proof}[Proof of Lemma \ref{lem:decayU0}] 
Let $f \in \X_{k}$ and $t >0$ be fixed. For simplicity, we introduce $g(x,v)=\sigma(v)^{-k}|f(x,v)|$, $(x,v) \in \T^{d}\times V$. One has then
\begin{equation}\begin{split}\label{eq:normuo}
\|U_{0}(t)f\|_{\X_{0}}&=\int_{\T^{d}\times V}\sigma(v)^{k}e^{-t\sigma(v)}|g(x-tv,v)|\d x \bm{m}(\d v)\\
&=\int_{\T^{d}\times V}\sigma(v)^{k}e^{-t\sigma(v)}|g(y,v)|\d y\,\bm{m}(\d v)\end{split}\end{equation}
where we performed the change of variable $y=x-tv$, $\d y=\d x.$ Using the elementary inequality
\begin{equation}\label{eq:uke}
u^{k}e^{-u} \leq \left(\frac{k}{e}\right)^{k} \qquad \forall u \geq 0\end{equation}
and applying it with $u=t\sigma(v)$, one has
$$\|U_{0}(t)f\|_{\X_{0}} \leq \left(\frac{k}{t\,e}\right)^{k}\int_{\T^{d}\times V}|g(y,v)|\d y\,\bm{m}(\d v)=\left(\frac{k}{t\,e}\right)^{k}\int_{\T^{d}\times V}\sigma(v)^{-k}|f(y,v)|\d y\,\bm{m}(\d v)$$
which gives \eqref{eq:uo}. Now, one deduces also from \eqref{eq:normuo}   and Fubini's theorem that
\begin{multline*}
\int_{0}^{\infty}\|U_{0}(t)f\|_{\X_{k}}\d t\leq \int_{0}^{\infty}\|U_{0}(t)g\|_{\X_{0}}=\int_{0}^{\infty}\d t\int_{\T^{d}\times V}e^{-t\sigma(v)} |g(y,v)|\d y\,\bm{m}(\d v)\\
=\int_{\T^{d}\times V}|g(y,v)|\d y\,\bm{m}(\d v)\int_{0}^{\infty}e^{-t\sigma(v)}\d t=\int_{\T^{d}\times V}\sigma(v)^{-1}(v)|g(y,v)|\d y\,\bm{m}(\d v) \leq \|f\|_{\X_{k+1}}
\end{multline*}
whereas
\begin{equation*}\begin{split}
\int_{0}^{\infty}t^{k}\|U_{0}(t)f\|_{\X_{0}}\d t&=\int_{\T^{d}\times V}|f(y,v)|\d y\,\bm{m}(\d v)\int_{0}^{\infty}t^{k}e^{-t\sigma(v)}\d t\\
&=\int_{\T^{d}\times V}\sigma^{-k}(v)|f(y,v)|\d y\,\bm{m}(\d v)\int_{0}^{\infty}\left(t\sigma(v)\right)^{k}e^{-t\sigma(v)}\d t\\
&=\int_{\T^{d}\times V}\sigma^{-k-1}(v)|f(y,v)|\d y\,\bm{m}(\d v)\int_{0}^{\infty}\tau^{k}e^{-\tau}\d \tau
\end{split}\end{equation*}
where we performed the change of variable $\tau=t\,\sigma(v)$ in the last step. This gives the last estimate.
\end{proof}

 \subsection{Decay of the Dyson-Phillips iterated}\label{sec:DP-2} We extend the decay of the semigroup $\left(U_{0}(t))\right)_{t\geq0}$ obtained in Lemma \ref{lem:decayU0} to the iterates 
$\left(U_{k}(t)\right)_{t\geq0}$ for any $k \geq1$. To do so, we  first observe that for any $t \geq 0$, $U_{0}(t)$ commute with any multiplication operator depending on the velocity i.e.
$$U_{0}(t)(\varpi\,f)=\varpi\,U_{0}(t)f$$
for any $\varpi=\varpi(v) \in L^{\infty}(V)$. Moreover, $U_{0}(t)$ has an exponential decay on any region in which $\sigma$ is bounded away from zero. Namely, for any $\delta >0$, introduce
$$\Lambda_{\delta}:=\{v \in V\;;\;\sigma(v) \geq \delta\}, \qquad \Sigma_{\delta}=V \setminus \Lambda_{\delta}$$
one has
$$\|U_{0}(t)\ind_{\Lambda_{\delta}}f\|_{\X_{0}} \leq e^{-t\delta}\|\ind_{\Lambda_{\delta}}f\|_{\X_{0}}\leq e^{-t\delta}\|f\|_{\X_{0}} \qquad \forall f \in \X_{0},\qquad t\geq 0.$$ 
Introducing for any $\delta >0$, the operator $\K^{(\delta)} \in \mathscr{B}(\X_{0})$ given by
\begin{equation}\label{defi:Hepsi}
\K^{(\delta)}f(x,v)=\ind_{\Lambda_{\delta}}\K f(x,v) \qquad \forall f \in \X_{0}, \quad (x,v) \in \T^{d}\times V\end{equation}
as well as
\begin{equation}\label{def:Kdel}
\overline{\K}^{(\delta)}=\K-\K^{(\delta)}\,,\end{equation}
one can check the following 
\begin{lemme}\label{lem:sizeDelta}
For any $n \in \{1,\ldots,N_{0}\}$,
\begin{equation}\label{eq:overKdelta}
\|\K-{\K}^{(\delta)}\|_{\mathscr{B}(\X_{-1},\X_{0})} \leq \delta^{n}\|\vartheta_{n}\|_{\infty}, \qquad \qquad \forall 0 \leq n \leq N_{0}.\end{equation}
\end{lemme}
Having such a property in mind, we can deduce the decay of $U_{k}(t)$ as $t \to \infty$ for any $k \in \N$ resulting in
\begin{propo}\label{prop:Snt} Assume that
$$f \in \X_{N_{0}}$$
then, for any $n \geq 1$, there exists $\bm{C}_{n} >0$ such that
\begin{equation}\label{eq:UntNH}
\left\|\sum_{k=0}^{n}U_{k}(t)f\right\|_{\X_{0}} \leq \bm{C}_{n}\,\left(\frac{\log t}{t}\right)^{N_{0}}\left\|f\right\|_{\X_{N_{0}}} \qquad \forall t >0.\end{equation}
\end{propo}
The proof is based on a splitting of each term $U_{k}(t)$ as 
$$U_{k}(t)=U_{k}^{(\delta)}(t)+\overline{U}_{k}^{(\delta)}(t)$$
where $U_{k}^{(\delta)}(t)$ is constructed as a Dyson-Phillips iterated involving \emph{only} the operator $\K^{(\delta)}$ whereas the reminder terms make appear \emph{at least once} the difference $\K-\K^{(\delta)}$ and, as such, can be made small with respect to $\delta$ by virtue of Lemma \ref{lem:sizeDelta}. For the part involving only $\K^{(\delta)}$, we are dealing with a Dyson-Phillips iterate associated with a collision frequency which is \emph{bounded away from zero}  we can deduce a full exponential decay of $U_{k}^{(\delta)}(t)$ and, optimizing the parameter $\delta$, we deduce the algebraic decay of $U_{k}(t)$.  Details are given in Appendix \ref{appen:DYSON}. 

\subsection{Representation formulae for remainder terms}

On the basis of Proposition \ref{prop:Snt}, one sees that, to capture a decay of $\mathcal{V}(t)f$, it is enough to focus on the decay of the reminders
\begin{equation}\label{eq:Sn+1}
\bm{S}_{n+1}(t):=\mathcal{V}(t)-\sum_{k=0}^{n}U_{k}(t), \qquad n \geq 0,\;\,t >0.\end{equation}
This is where the resolvent approach enters in the game since it will be convenient to observe that such a reminder admits a useful representation formula as an inverse Laplace transform. 
We recall here that, for $\mathrm{Re}\l >0$, the resolvent of $\A+\K$ exists and is given by
\begin{equation}\label{eq:resAK}
\Rs(\l,\A+\K)=\sum_{n=0}^{\infty}\Rs(\l,\A)\left[\K\Rs(\l,\A)\right]^{n}
\end{equation}
where the series converge in operator norm. 

Before such a representation, let us explain here the main important consequences of Assumptions \ref{hypH}--\ref{hypK} we refer to Section \ref{sec:Assum} for a complete proof.
\begin{theo}\label{theo:main-assum}
If $\K$ satisfies Assumptions \ref{hypH}, then there exists $\mathsf{q} \in \N$ such that
$$\left\{\left[\K\Rs(\l,\A)\right]^{\mathsf{q}}\;;\;0 \leq \mathrm{Re}\l \leq 1\right\} \subset \B(\X_{0})$$
is collectively compact. {Moreover, if $\K$ satisfies also Assumptions \ref{hypK},} there exists $C_{0} >0$ such that
\begin{equation}\label{eq:decayP}
\left\|\left[\Rs(\l,\A)\K\right]^{2}\right\|_{\B(\X_{0})}+ \left\|\left[\K\Rs(\l,\A)\right]^{2}\right\|_{\B(\X_{0})} \leq \frac{C_{0}}{\sqrt{|\l|}}, \qquad \forall \l \in \overline{\C}_{+} \setminus \{0\}.\end{equation}
In particular,  for any $\mathsf{p}>4$, there is $C_{\mathsf{p}} >0$ such that
\begin{equation}
\label{eq:power}
\sup_{\varepsilon \geq 0}\int_{|\eta|>1}\left\|\left[\Rs(\varepsilon+i\eta,\A)\K\right]^{\mathsf{p}}\right\|_{\mathscr{B}(\X_{0})} \d \eta \leq C_{\mathsf{p}}< \infty.\end{equation}
\end{theo}
\begin{nb} We point here that the collective compactness properties is a consequence of Assumption \eqref{eq:mue} and is crucial in particular for the existence of an invariant density $\Psi$ of $\left(\mathcal{V}(t)\right)_{t\geq0}$ in Theorem \ref{theo:main-invar} (we refer to \cite{mmktore} for details on that matter).\end{nb} 

It is well-established that the reminder $\bm{S}_{n+1}(t)$ of the Dyson-Phillips series here above admits the following Laplace transform:
\begin{equation}\label{eq:Snl}\begin{split}
\mathscr{S}_{n+1}(\l)f&=\int_{0}^{\infty}\exp\left(\l t\right)\bm{S}_{n+1}(t)f\d t\\
&=\sum_{k=n}^{\infty}\Rs(\l,\A)\left[\K\Rs(\l,\A)\right]^{k+1}  f, \qquad \forall \mathrm{Re}\l >0, f \in \X_{0}.
\end{split}\end{equation}
This allows to express in a natural way $\bm{S}_{n+1}(t)f$ as the inverse Laplace transform of $\mathscr{S}_{n+1}(\e+i\eta)$ $(\eta \in \R,\e>0)$ (see Proposition \ref{propo:Sn+1} in Appendix \ref{appen:DYSON}). 

The crucial point in our analysis is then to extend the representation formula \eqref{eq:Snl} \emph{up to the boundary $\e=0$}. Of course, since $0 \in \mathfrak{S}(\A)$ we cannot expect to define $\Rs(\l,\A+\K)$ for $\mathrm{Re}\l=0$ using \eqref{eq:resAK}. However,  under assumption \eqref{eq:conservative}, it is possible to extend the definition of $\K\Rs(\l,\A)$ to $\l=0$ by observing that
$$\lim_{\l \to 0}\K\Rs(\l,\A)\varphi=\Ms_0\varphi, \qquad \varphi \in \X_0$$
exists with, for almost every $(x,v) \in \T^{d}\times V$,
$$
\Ms_{0}\varphi(x,v)=\int_{V}\bm{k}(v,w)\bm{m}(\d w)\int_{0}^{\infty}\exp\left(-t\sigma(w)\right)\varphi(x-tw,w)\d t\,.$$
Notice indeed that, for $\varphi \in \X_{0}$, $\varphi \geq0$,
\begin{equation}\begin{split}
\label{eq:M0}
\|\Ms_{0}\varphi\|_{0} &= \int_{\T^{d}\times V}\d x\,\bm{m}(\d v)\int_{V}\bm{k}(v,w)\bm{m}(\d w)\int_{0}^{\infty}\exp\left(-t\sigma(w)\right) \varphi(x-tw,w) \d t\\
&=\int_{\T^{d}\times V}\varphi(y,w) \d y\,\bm{m}(\d w)\int_{0}^{\infty}\exp\left(-t\sigma(w)\right)\d t\int_{V}\bm{k}(v,w)\bm{m}(\d v)\\
&=\int_{\T^{d}\times V} \varphi(y,w)\d y \frac{1}{\sigma(w)}\bm{m}(\d w)\int_{V}\bm{k}(v,w)\bm{m}(\d v)=\int_{\T^{d}\times V} \varphi(y,w)\d y\,\bm{m}(\d w)\end{split}\end{equation}
where we used the change of variable $y=x-tw$ in the second identity and \eqref{eq:conservative} for the last one. Thus, $\Ms_{0}$ is stochastic, i.e. mass-preserving on the positive cone of $\X_{0}$:
$$\|\Ms_{0}\varphi\|_{\X_{0}}=\|\varphi\|_{\X_{0}} \qquad \forall \varphi \in \X_{0},\;\varphi \geq0.$$
In particular, 
$$\|\Ms_{0}\|_{\mathscr{B}(\X_{0})} = 1.$$
More generally, it is possible to define, for any $\eta \in \R$
$$\Ms_{i\eta}f:=\lim_{\e\to0^{+}}\K\Rs(\e+i\eta,\A)f, \qquad f \in \X_{0}.$$
The properties of the operator $\Ms_{i\eta}$ are the cornerstone of our analysis which, somehow, culminates with the  following result :
\begin{theo}\label{theo:MainLaplace} {Let Assumptions \ref{hypH} and \ref{hypK} be in force}, Let $f \in \X_{N_{0}}$ be such that
\begin{equation}\label{eq:0mean-main}
\varrho_{f}=\int_{\Omega\times V}f(x,v)\d x \otimes \bm{m}(\d v)=0.\end{equation}
Then, the following holds:
\begin{enumerate}
\item For any $n\geq0$ the limit
$$\lim_{\e\to0^{+}}\mathscr{S}_{n}(\e+i\eta)f,$$
exists in $\mathscr{C}_{0}^{N_{0}-1}(\R,\X_{0})$. Its limit is denoted ${\Upsilon}_{n}(\eta)f$. 
\item For any $n \geq 5 \cdot 2^{N_{0}-1}$, the trace function
$$\eta \in \R \longmapsto \Upsilon_{n}(\eta)f \in \X_{0}$$
and its derivatives of order $k \in \{0,\ldots,N_{0}-1\}$ are integrable, i.e.
$$\int_{\R}\left\|\dfrac{\d^{k}}{\d\eta^{k}}\Upsilon_{n}(\eta)f\right\|_{\X_{0}} \d \eta < \infty\qquad \forall k \in \{0,\ldots,N_{0}-1\}.$$
\end{enumerate}
Consequently, for $n \geq 5\cdot 2^{N_{0}-1}-1$ and $f \in \X_{N_{0}}$ satisfying \eqref{eq:0mean-main}, 
 one has
\begin{equation}\label{eq:SntInt-main} 
\bm{S}_{n+1}(t)f=\lim_{\ell\to\infty}\frac{1}{2\pi} \int_{-\ell}^{\ell}\exp\left(i\eta t\right)\Upsilon_{n+1}(\eta)f\d\eta=\frac{1}{2\pi}\int_{-\infty}^{\infty}\exp\left(i\eta t\right)\Upsilon_{n+1}(\eta)f\d\eta, \qquad \forall t >0 \end{equation}
where the convergence holds in $\X_{0}$ and
\begin{equation}\label{eq:SnDeri-main}
\bm{S}_{n+1}(t)f=\left(-\frac{i}{t}\right)^{N_{0}-1}\frac{1}{2\pi}\int_{-\infty}^{\infty}\exp\left(i\eta t\right)\dfrac{\d^{N_{0}-1}}{\d\eta^{N_{0}-1}}\Upsilon_{n+1}(\eta)f\,\d\eta\end{equation}
holds true for any $t \geq0$ where the convergence of the integral holds in $\X_{0}$.
\end{theo}
We admit this result for a little while, the whole rest of the paper being devoted to a complete proof of this fundamental Theorem. Let us illustrate right away how to deduce our main result from Theorem \ref{theo:MainLaplace}:

\begin{proof}[Proof of Theorem \ref{theo:maindec}] Let us fix $f \in \X_{N_{0}}$. To prove the result, we can assume without loss of generality that $\varrho_{f}=0$.  Of course, the term $\mathsf{\Theta}_{f}(\cdot)$ is given by
$$
\mathsf{\Theta}_{f}(\eta)=\dfrac{\d^{N_{0}-1}}{\d \eta^{N_{0}-1}}\Upsilon_{n+1}(\eta)f \in  \X_{0}, \qquad \eta \in \R$$
for some suitable choice of $n\in \N$. Recall first that, for any $n \in \N$ and any $t\geq 0$
$$\mathcal{V}(t)f=\sum_{k=0}^{\infty}{U}_{k}(t)f=\sum_{k=0}^{n}{U}_{k}(t)+\bm{S}_{n+1}(t)f$$
where, according to Proposition \ref{prop:Snt},
$$\left\|\sum_{k=0}^{n}{U}_{k}(t)f\right\|_{\X_{0}} \leq C_{n}\left(\frac{\log}{1+t}\right)^{-N_0}, \qquad \forall t\geq 0$$
for some positive constant $C_{n}$ depending on $n$ and $f$ (but not on $t$). Choosing now $n\geq 2^{N_{0}-1}\mathsf{p}$ and using \eqref{eq:SnDeri-main}, one obtains
$$\left\|\mathcal{V}(t)f\right\|_{\X_{0}} \leq  C_{n}(1+t)^{-N_{0}-1}+t^{-N_{0}-1}\mathcal{F}_{n}(t)$$
where
$$\mathcal{F}_{n}(t)=\left\|\frac{1}{2\pi}\int_{-\infty}^{\infty}\exp\left(i\eta t\right)\dfrac{\d^{N_{0}-1}}{\d\eta^{N_{0}-1}}\Upsilon_{n+1}(\eta)f\d\eta\right\|_{\X_{0}}$$
is such that $\lim_{t\to \infty} \mathcal{F}_{n}(t)=0$ according to Riemann-Lebesgue Theorem (recall the mapping $\eta \mapsto \dfrac{\d^{N_{0}-1}}{\d\eta^{N_{0}-1}}\Upsilon_{n+1}(\eta)f \in \X_{0}$ is integrable over $\R$ according to \eqref{eq:SnDeri-main}). This proves the first part of the result. 

Let us now prove the second part of it. According to \eqref{eq:decayP}, one deduces very easily that, for any $\mathsf{p} >4$, there is $C_{\mathsf{p}} >0$ such that
$$\|\Ms_{i\eta}^{\mathsf{p}}\|_{\B(\X_{0})} \leq C_{\mathsf{p}}\,|\eta|^{-\frac{\mathsf{p}}{4}}, \qquad \forall |\eta| >1$$
from which, for $R >1$, 
\begin{equation}\label{eq:decay-power}
\int_{|\eta| >R}\left\| \mathsf{M}_{i\eta} ^{\mathsf{p}}\right\|_{\mathscr{B}(\X_{0})} \d \eta \leq \frac{8C_{\mathsf{p}}}{\mathsf{p}-4}R^{-\frac{\mathsf{p}-4}{4}}=:\tilde{C}(\mathsf{p})R^{-\beta}, \qquad \forall R >1\end{equation}
with $\beta=\frac{\mathsf{p}-4}{4}.$ Since the mapping
$$\mathsf{\Theta}_{f}\::\:\eta \in \R \longmapsto  \dfrac{\d^{N_{0}-1}}{\d\eta^{N_{0}-1}}\Upsilon_{n+1}(\eta)f \in \X_{0}$$
belongs to $\mathscr{C}_{0}(\R,\X_{0})$, it is  uniformly continuous. This allows to define a (minimal) modulus of continuity 
$$\omega_{f}(s):=\sup\left\{\left\|\mathsf{\Theta}_{f}(\eta_{1})-\mathsf{\Theta}_{f}(\eta_{2})\right\|_{\X_{0}}\,;\,\eta_{1},\eta_{2}\in \R,\,|\eta_{1}-\eta_{2}|\leq s\right\},\quad  s\geq 0.$$
The estimate then comes from some standard reasoning about Fourier transform. Namely, introducing the Fourier transform  of the (Bochner integrable) function $\mathsf{\Theta}_{f}$ as 
$$\widehat{\mathsf{\Theta}}_{f}(t)=\int_{\R}\exp(i\eta t)\mathsf{\Theta}_{f}(\eta)\d\eta \in  \X_{0}, \qquad t \geq 0$$
one has then, since $e^{i\pi}=-1=\exp(i\pi t/t),$ $t >0$,
$$\widehat{\mathsf{\Theta}}_{f}(t)=-\int_{\R}\exp\left(i\eta t + i \frac{\pi}{t} t\right)\mathsf{\Theta}_{f}(\eta)\d \eta=-\int_{\R}\exp\left(i y t\right)\mathsf{\Theta}_{f}\left(y-\frac{\pi}{t}\right)\d y$$
which gives, taking the mean of both the expressions of $\widehat{\mathsf{\Theta}}_{f}(t)$,
$$\widehat{\mathsf{\Theta}}_{f}(t)=\frac{1}{2}\int_{\R}\exp\left(i\eta t\right)\left(\mathsf{\Theta}_{f}(\eta)-\mathsf{\Theta}_{f}\left(\eta-\frac{\pi}{t}\right)\right)\d\eta.$$
Consequently, if one assumes that $R > 2\pi$, 
\begin{equation*}
\left\|\widehat{\mathsf{\Theta}}_{f}(t)\right\|_{\X_{0}} \leq \frac{1}{2}\int_{|\eta|\leq R}\left\|\mathsf{\Theta}_{f}(\eta)-\mathsf{\Theta}_{f}\left(\eta-\frac{\pi}{t}\right)\right\|_{\X_{0}}\d\eta 
+\int_{|\eta| >\frac{R}{2}}\left\|\mathsf{\Theta}_{f}(\eta)\right\|_{ \X_{0}}\d\eta\end{equation*}
where we used that $\{\eta \in \R\;;\;|\eta+\frac{\pi}{t}| > R\} \subset \{\eta \in \R\;;\;|\eta| > R-\pi\} \subset \{\eta \in \R\;;\;|\eta| >\frac{R}{2}\}$ since $t\geq1,$ $R-\pi >\frac{R}{2}$. Therefore, using the modulus of continuity $\omega_{f}$ and   \eqref{eq:decay-power}, we deduce that
\begin{equation}\begin{split}\label{eq:Hol}
\left\|\widehat{\mathsf{\Theta}}_{f}(t)\right\|_{\X_{0}}&\leq  R\omega_{f}\left(\frac{\pi}{t}\right)  + \int_{|\eta| >\frac{R}{2}}\left\|\mathsf{\Theta}_{f}(\eta)\right\|_{\X_{0}}\d\eta\\
&\leq  R\omega_{f}\left(\frac{\pi}{t}\right) + 2^{\beta}\tilde{C}(\mathsf{p})R^{-\beta}\|f\|_{\X_{N_{0}}}, \qquad \forall R > 2\pi, \qquad t \geq 1.\end{split}
\end{equation}
Optimising then the parameter $R$, i.e. choosing
$$R=\left(\frac{2^{\beta }\beta\,\tilde{C}(\mathsf{p})\|f\|_{\X_{N_{0}}} }{\omega_{f}\left(\frac{\pi}{t}\right)}\right)^{\frac{1}{\beta+1}}$$
(up to work with $t \geq t_{0} \geq 1$ to ensure that $R >2\pi$), we obtain the desired estimate since $\frac{\beta}{\beta+1}=\frac{\mathsf{p-4}}{\mathsf{p}}$.\end{proof} 

%5The rest of the paper is devoted to the proof of Theorem \ref{theo:MainLaplace}. 

\subsection{Organization of the paper}\label{sec:orga} 
%The general strategy to prove Theorem \ref{theo:maindec} is explained in full details in Section \ref{sec:Main} in which the main steps are detailed. 
%The full proof of Theorem \ref{theo:maindec} can then be deduced from our main technical result in Theorem \ref{theo:MainLaplace} and the rest of the paper is devoted to the full proof of this result. 
The rest of the paper is devoted to the proof of Theorem \ref{theo:MainLaplace}. More precisely, in Section \ref{sec:Assum} we deduce from Assumptions \ref{hypH} the main collective compactness properties and decay estimates which will play a fundamental role for the proof of Theorem \ref{theo:MainLaplace}. In Section \ref{sec:reguTr}, we established the preliminary results about the regularity of some of the terms appearing in the resolvent $\Rs(\l,\A+\K)$ and establish the existence of some of their limit along the imaginary axis $\l=i\eta$, $\eta \in \R$. In Section \ref{sec:spec}, we especially focus on the spectral properties of $\Ms_{\l}$ and, in a particular way, on its spectral behaviour around $\l=0$. We use, in a crucial way in this part, Anselone's collective compactness theory and the results established in Section \ref{sec:Assum}. Section \ref{sec:trace} establish the existence of the extension of $\Rs(\l,\A+\K)$ to the imaginary axis. In Section \ref{sec:near0}, we finally provide the full proof of Theorem \ref{theo:MainLaplace}. The paper ends with three Appendices containing several of the main technical aspects of the proofs. Namely, Appendix \ref{appen:DYSON} is devoted to the proof of Proposition \ref{prop:Snt}, Appendix \ref{appen:Ml} establishes some of the technical properties of $\Ms_{\l}$ used in Section \ref{sec:spec} and Appendix \ref{sec:functional}
 recalls the main aspects of Anselone's collective compactness theory we use in the paper.

\section{Consequences of Assumptions \ref{hypH} and \ref{hypK}}\label{sec:Assum}

We comment here on the main consequences of our  set of Assumptions on the operator $\K$ given in Assumptions \ref{hypH}.  Namely, the following two Theorems \ref{theo:collective} and \ref{theo:decay-KRA} provides a complete proof of Theorem \ref{theo:main-assum}. 

\subsection{A criterion for collective compactness of some power}
We begin with by showing how  Assumption \ref{hypH}  is the key argument to deduce the collective compactness of some power of $\K\Rs(\l,\A)$ in Theorem \ref{theo:main-assum}.
Our scope is to prove the following
\begin{theo}\label{theo:collective}
Let us assume that $\K\::\:L^{1}(V) \to L^{1}(V)$ is a \emph{weakly compact}   operator and its kernel
satisfies \eqref{eq:conservative} together with \eqref{eq:mue}. {We also assume the measure $\bm{m}$ to satisfy \eqref{eq:hypm}.}
Then, there exists $\mathsf{q} \in \N$ such that the family of operators
$$\left\{\Ms_{\l}^{\mathsf{q}}\;;\;0 \leq \mathrm{Re}\l \leq 1\right\} \subset \B(\X_{0})$$
is collectively compact.
\end{theo}
\begin{nb} %The above Theorem is valid  in the case in which $\bm{m}(\d v)=\bm{m}(v)\d v$ is absolutely continuous with respect to the Lebesgue measure. 
%However, it is straightforward to extend it, thanks to the use of \cite[Theorem 8]{mmktore} and \cite[Theorem 13]{mkjfa}, to the case of a general measure  $\bm{m}$ satisfying the following: there exists $\alpha >0$ such that, for any bounded set $S \subset V$, there is $c(S) >0$ such that
%\begin{equation}\label{eq:hypm}
%\sup_{\nu \in \S^{d-1}}\bm{m}\otimes \bm{m}\big(\left\{(v,w) \in S\times S\;;\;|(v-w)\cdot \nu| < \e\right\}\big) \leq c(S)\e^{\alpha}, \qquad \forall \e >0.\end{equation} 
We insist here on the fact that the collective compactness provided by Theorem \ref{theo:collective} is the crucial argument in the proof of an invariant density $\Psi$ in Theorem \ref{theo:main-invar} as obtained in \cite{mmktore}.\end{nb}
The proof of Theorem \ref{theo:collective} resorts on similar results in \cite{mmktore} and on a suitable approximation argument. We use here notations of Section \ref{sec:DP-2}. Namely, for any $\delta >0$, we recall from \eqref{defi:Hepsi}
$$\K^{(\delta)}\::\:\varphi \in \X_{0} \mapsto \K^{(\delta)}\varphi(x,v)=\ind_{\Lambda_{\delta}}(v)\int_{V}\bm{k}(v,w)\varphi(x,w)\bm{m}(\d w) \in \X_{0}.$$
It is clear that $\K^{(\delta)}\in \B(\X_{0})$ with $\K^{(\delta)}\varphi(x,v) \leq \K\varphi(x,v)$ for any $\varphi \in \X_{0}$, $\varphi \geq0.$
One has the following:
\begin{lemme}\label{lem:approx}
 For any $n \in \N$,
$$\sup_{\mathrm{Re}\l \geq0}\left\|\bigg[\K\Rs(\l,\A)\bigg]^{n}-\left[\K^{(\delta)}\Rs(\l,\A)\right]^{n}\right\|_{\B(\X_{0})} \leq n\mu_{\delta}$$
where
$$\mu_{\delta}=\sup_{w\in V}\frac{1}{\sigma(w)}\int_{\Sigma_{\delta}}\bm{k}(v,w)\bm{m}(\d v), \qquad \quad \Sigma_{\delta}=V \setminus \Lambda_{\delta}=\{v \in V\;;\;\sigma(v) \leq \delta\}.$$
\end{lemme}
\begin{proof} The proof is made by induction over $n \in \N$. Observing that $\K-\K^{(\delta)}$ is an integral operator of the form
$$(\K-\K^{(\delta)})\varphi(x,v)=\ind_{\Sigma_{\delta}}(v)\int_{V}\bm{k}(v,w)\varphi(x,w)\d\bm{m}(\d w)$$
one already saw that 
$$\|\K\Rs(\l,\A)-\K^{(\delta)}\Rs(\l,\A)\|_{\B(\X_{0})}=\|(\K-\K^{(\delta)})\Rs(\l,\A)\|_{\B(\X_{0})} \leq \|\K-\K^{(\delta)}\|_{\B(\X_{-1},\X_{0})}.$$
Since
$$\|\K-\K^{(\delta)}\|_{\B(\X_{-1},\X_{0})}=\sup_{w \in V}\frac{1}{\sigma(w)}\int_{V}\ind_{\Sigma_{\delta}}(v)\bm{k}(v,w)\bm{m}(\d v)=\mu_{\delta}$$
this proves the result for $n=1$. Let us assume the result to be true for some $n\geq1.$ One has
\begin{multline*}
\bigg[\K\Rs(\l,\A)\bigg]^{n+1}-\left[\K^{(\delta)}\Rs(\l,\A)\right]^{n+1}=\left(\K-\K^{(\delta)}\right)\Rs(\l,\A)\left[\K\Rs(\l,\A)\right]^{n} \\
+  \K^{(\delta)}\Rs(\l,\A)\left(\left[\K\Rs(\l,\A)\right]^{n}-\left[\K^{(\delta)}\Rs(\l,\A)\right]^{n}\right).\end{multline*}
Observing that 
$$\left\| \left[\K^{(\delta)}\Rs(\l,\A)\right]^{n}\right\|_{\B(\X_{0})} \leq \left\| \left[\K\Rs(\l,\A)\right]^{n}\right\|_{\B(\X_{0})} \leq \|\Ms_{0}^{n}\|_{\B(\X_{0})} \leq 1,$$
for any $\mathrm{Re}\l \geq 0,$ $n \in\N$, one deduces that
\begin{multline*}
\left\|\left[\K\Rs(\l,\A)\right]^{n+1}-\left[\K^{(\delta)}\Rs(\l,\A)\right]^{n+1}\right\|_{\B(\X_{0})} \leq \left\|\left(\K-\K^{(\delta)}\right)\Rs(\l,\A)\right\|_{\B(\X_{0})} \\
+ \left\|\bigg[\K\Rs(\l,\A)\bigg]^{n}-\left[\K^{(\delta)}\Rs(\l,\A)\right]^{n}\right\|_{\B(\X_{0})}.\end{multline*}
Using the induction hypothesis, one sees that
$$\left\|\bigg[\K\Rs(\l,\A)\bigg]^{n+1}-\left[\K^{(\delta)}\Rs(\l,\A)\right]^{n+1}\right\|_{\B(\X_{0})} \leq \mu_{\delta}+ n\mu_{\delta}=(n+1)\mu_{\delta}$$
which proves the result.
\end{proof}
We now recall some result which is somehow proven in \cite[Theorem 18]{mmktore} (but not stated as below):
\begin{lemme}\label{lem:KNCol} {Let us assume that $\K\::\:L^{1}(V) \to L^{1}(V)$ is a \emph{weakly compact}   operator and its kernel
satisfies \eqref{eq:conservative} together with \eqref{eq:mue} while the measure $\bm{m}$ satisfies \eqref{eq:hypm}.} There exists $N$ large enough (depending only on $\bm{m}$) such that, for any $\delta >0$,
$$\left\{\left[\Rs(\l,\A)\K^{(\delta)}\right]^{N}\;;\;\mathrm{Re}\l \in [0,1]\right\} \subset \B(\X_{0})$$
is collectively compact
\end{lemme}
\begin{proof} {Recall the definition of the Dyson-Phillips \eqref{eq:DysPhil}--\eqref{eq:UnK}. Since the measure $\bm{m}(\d v)=\bm{m}(v)\d v$ satisfies \eqref{eq:hypm}, one can apply \cite[Theorem 13]{mkjfa} to assert that there exists $N_{0} \in \N$  such that $U_{j}(t)$ is compact for any $j \geq N_{0}$ and any $t \geq 0$.
%%a reminder 
%$\bm{S}_{k+1}(t)$ is compact for any $t\geq0,$ $k \geq N_{0}$ where 
%$$\bm{S}_{k+1}(t)=\sum_{j=k}^{\infty}U_{j}(t).$$
%Notice that the compactness of $\bm{S}_{k+1}^{(\delta)}(t)$ for $k \geq N_{0}$ implies that of $U_{j}(t)$ for $j \geq N_{0}$ (see \cite[Theorem 2.6, p. 16]{MMK}).
Let us now consider $\delta >0$ and recall that $\Lambda_{\delta}=\{v \in V\;;\sigma(v) \geq \delta\}$ and  $\K^{(\delta)}=\ind_{\Lambda_{\delta}}\K$. One observes that 
$\Rs(\l,\A)\ind_{\Lambda_{\delta}}=\Rs(\l,\A_{\delta})$
where $\A_{\delta}=\A\ind_{\Lambda_{\delta}}$ is identified with the advection operator $\A$ on $\Lambda_{\delta}$. We also define the sequence of Dyson-Phillips iterated $\left(U_{j}^{(\delta)}(t)\right)_{j}$ associated to $\A_{\delta}=\A\ind_{\Lambda_{\delta}}$ and $\K^{(\delta)}$, i.e.
$$U_{j}^{(\delta)}(t)=\int_{0}^{t}U_{0}^{(\delta)}(t-s)\K^{(\delta)}U_{j-1}^{(\delta)}(s)\d s, \qquad j \geq1$$
and $\left(U_{0}^{(\delta)}(t)\right)_{t\geq0}$ is the $C_{0}$-semigroup in $\X_{0}$ generated by $\A_{\delta}=\A\ind_{\Lambda_{\delta}}$. One observes that
$$U_{j}^{(\delta)}(t) \leq U_{j}(t), \qquad \forall t\geq0, \quad j \in \N$$
Therefore, by a domination argument, $U_{j}^{(\delta)}(t)$ is weakly-compact for any $j \geq N_{0}$ and, consequently, $U_{j}^{(\delta)}(t)$ is compact for any $j \geq N_{1}=2N_{0}+1$ (see \cite[Theorem 2.6 and Corollary 2.1, p. 16]{MMK}).  It follows then that the mapping 
$$t \geq0 \mapsto U_{j}^{(\delta)}(t) \in \B(\X_{0})$$
is continuous (in operator norm) for $j \geq N_{2}=N_{1}+1$ (see \cite[Corollary 2.2, p. 10]{MMK}).}

Now, observe that, on $\Lambda_{\delta}$ the collision frequency is bounded from below, one has 
$$s(\A_{\delta}) <0.$$  Then, since 
$$\int_{0}^{\infty}\exp\left(-\l t\right)U_{j}^{(\delta)}(t)\d t=\left[\Rs(\l,\A_{\delta})\K^{(\delta)}\right]^{j}\Rs(\l,\A_{\delta})$$
where the integral is converging in operator norm for any $\mathrm{Re} \l \geq0$, we deduce that
$$\left[\Rs(\l,\A_{\delta})\K^{(\delta)}\right]^{N_{2}+1} \quad \text{ is compact for any } \mathrm{Re}\l >s(\A_{\delta}).$$
Since, for any $\mathrm{Re}\l > s(\A_{\delta})$, the mapping
$$t \geq 0 \mapsto \exp\left(-i\,t\,\mathrm{Im}\l\right)\left[\exp\left(-t\,\mathrm{Re}\l \right)U_{j}^{(\delta)}(t)\right] \d t \in \mathscr{B}(\X_{0})$$
is \emph{Bochner integrable}, one has
$$\left[\Rs(\l,\A_{\delta})\K^{(\delta)}\right]^{j}\Rs(\l,\A_{\delta})=\int_{0}^{\infty}\exp\left(-\l t\right)U_{j}^{(\delta)}(t)\d t$$
is compact for $j \geq N_{2}$ and $\mathrm{Re}\l > s(\A_{\delta})$ and, by Riemann-Lebesgue Theorem, 
$$\lim_{|\mathrm{Im}\l | \to \infty}\left\|\int_{\R}\exp\left(-i\,t\,\mathrm{Im}\l \right)\left[\exp\left(-t\mathrm{Re}\l \right)U_{j}^{(\delta)}(t)\right] \d t\right\|_{\B(\X_{0})}=0$$
according to Riemann-Lebesgue Theorem. It follows that
$$\lim_{|\mathrm{Im}\l | \to \infty}\sup_{\mathrm{Re}\l \in [0,1]}\left\|\int_{\R}\exp\left(-i\,t\,\mathrm{Im}\l \right)\left[\exp\left(-t\mathrm{Re}\l \right)U_{j}^{(\delta)}(t)\right] \d t\right\|_{\B(\X_{0})}=0.$$
Indeed, $$\left\{\exp\left(-\bullet\mathrm{Re}\l \right)U_{j}^{(\delta)}(\bullet)\;;\;\mathrm{Re}\l \in [0,1]\right\}$$
is a compact subset of $L^{1}(\R^{+}\,,\,\B(\X_{0}))$ since the mapping %$t \geq 0 \mapsto \exp\left(-t\,\mathrm{Re}\l \right)U_{j}^{(\delta)}(t) \in L^{1}(\R^{+},\B(\X_{0}))$ depends continuously on $\mathrm{Re}\l \in [0,1]$.
$$\mathrm{Re}\l \in [0,1] \mapsto \exp\left(-\bullet \mathrm{Re}\l \right)U_{j}^{(\delta)}(\bullet) \in L^{1}(\R^{+},\B(\X_{0}))$$ is continuous. Thus, 
$$\lim_{|\mathrm{Im}\l | \to \infty}\sup_{\mathrm{Re}\l \in [0,1]}\left\|\left[\Rs(\l,\A_{\delta})\K^{(\delta)}\right]^{N_{2}+1}\right\|_{\B(\X_{0})}=0, \qquad \forall \delta >0$$
and finally 
$$\left\{\left[\Rs(\l,\A_{\delta})\K^{(\delta)}\right]^{N_{2}+1}, \quad 0 \leq \mathrm{Re}\l \leq 1\right\}$$
is collectively compact.  One gets the conclusion since $\Rs(\l,\A)\K^{(\delta)}=\Rs(\l,\A_{\delta})\K^{(\delta)}$.
\end{proof}
\begin{proof}[Proof of Theorem \ref{theo:collective}] The previous Lemma \ref{lem:approx} shows that, for any $n \in \N$, $\left[\K\Rs(\l,\A)\right]^{n}$ can be approximated in the operator norm by $\left[\K^{(\delta)}\Rs(\l,\A)\right]^{n}$ as $\delta \to 0$ \emph{uniformly with respect to} $\l$ in the set
$$\{\l \in \C\;;\;\mathrm{Re}\l \in [0,1]\}.$$ Therefore, to prove Theorem \ref{theo:collective}, it is enough to prove that there exists $m \in \N$ such that, for \textit{any fixed} $\delta \in [0,1]$, the set of operators
$$\left\{\left[\K^{(\delta)}\Rs(\l,\A)\right]^{m}\;;\;0 \leq \mathrm{Re}\l \leq1\right\} \subset \B(\X_{0})$$
is collectively compact. One notices now that, for any $m \in\N$, $m \geq 2$, 
\begin{equation}\label{eq:KKRR}
\left[\K^{(\delta)}\Rs(\l,\A)\right]^{m}=\K^{(\delta)}\left[\Rs(\l,\A)\K^{(\delta)}\right]^{m-2}\Rs(\l,\A)\K^{(\delta)}\Rs(\l,\A), \qquad \l \in \overline{\C}_{+}.\end{equation}
We deduce from the previous Lemma that we can find $m$ large enough (\emph{independent of $\delta$}) so that 
$$\left\{\left[\Rs(\l,\A )\K^{(\delta)}\right]^{m-2}\;;\;\mathrm{Re}\l \in [0,1]\right\} \subset \B(\X_{0})$$
is collectively compact.  
%and so is
%\begin{equation}\label{eq:CollRK}
%\left\{\left[\Rs(\l,\A_{\delta})\K^{(\delta)}\right]^{m-2}\;;\;;\;0 \leq \mathrm{Re}\l \leq 1\right\} \subset \B(\X_{0}).\end{equation}
Since
$$\left|\Rs(\l,\A)\K^{(\delta)}\Rs(\l,\A)\varphi\right| \leq \Rs(0,\A_{\delta})\Ms_{0}|\varphi|$$
for any $\varphi \in \X_{0}$ with $\Rs(0,\A_{\delta})$ and $\Ms_{0}$ both bounded. One deduces that
$$\left\{\Rs(\l,\A)\K^{(\delta)}\Rs(\l,\A)\varphi\;;\;\|\varphi\|_{\X_{0}} \leq 1\;\;;\;\;0 \leq \mathrm{Re}\l \leq 1\right\}$$
is a bounded subset of $\X_{0}.$ From the collective compactness of $\left\{\left[\Rs(\l,\A_{\delta})\K^{(\delta)}\right]^{m-2}\right\}_{\mathrm{Re}\l \in [0,1]}$, one deduces that
$$\left\{\left[\Rs(\l,\A)\K^{(\delta)}\right]^{m-2}\Rs(\l,\A)\K^{(\delta)}\Rs(\l,\A)\varphi\;\;;\;\|\varphi\|_{\X_{0}}\leq 1\;;\;0 \leq \mathrm{Re}\l \leq 1\right\}$$
is included in some compact of $\X_{0}$.  This ends the proof since $\K^{(\delta)} \in \B(\X_{0})$.  
 \end{proof}

\subsection{Decay of $\K\Rs(\l,\A)\K$} The scope here is to prove the following result
\begin{theo}\label{theo:decay-KRA}
{If $\K$ satisfies Assumption \ref{hypK}}, there exists $C >0$ such that
\begin{equation}\label{eq:decayKRA}
\left\|\left[\Rs(\l,\A)\K\right]^{2}\right\|_{\B(\X_{0})}+\left\|\left[\K\Rs(\l,\A)\right]^{2}\right\|_{\B(\X_{0})} \leq \frac{C}{\sqrt{|\l|}}, \qquad \qquad \forall \l \in \overline{\C}_{+}, \quad |\l|\geq 1.\end{equation}
\end{theo}
We will divide the proof in several lemmas which contain  the crucial estimates of our analysis. We begin with the following easy consequences of Assumption \ref{hypK}
\begin{lemme}\label{lem:K2kk} {Let $\K$ satisfy Assumption \ref{hypK} and introduce} 
$$\bm{K}_{2}(v,w',w)=\bm{k}(v,w)\bm{k}(w,w')\bm{m}(w), \qquad \forall v,w,w'\in V^{3}.$$
Then,  there exists $C >0$ such that
\begin{multline*}\int_{\R^{d}}\sigma^{-1}(w) \d w\int_{\R^{d}}\left|w\cdot \nabla_{w}\bm{K}_{2}\left(v,w',w\right)\right|\max\left(1,\sigma^{-1}(v)\right)\bm{m}(\d v)\\
 \leq  C\,\sigma(w')\,,\qquad \forall w'\in V,\end{multline*}
Moreover,
$$\left|w \cdot \nabla_{w} \sigma(w)\right| \leq C_{1}\sigma(w), \qquad \forall w \in V.$$\end{lemme}
\begin{proof} Noticing that
\begin{multline*}
w\cdot \nabla_{w}\bm{K}_{2}\left(v,w',w\right)=\left[\bm{k}(v,w)\left(w\cdot \nabla_{w}\bm{k}(w,w')\right)+\bm{k}(w,w')\left(w\cdot \nabla_{w}\bm{k}(v,w)\right)\right]\bm{m}(w)\\
+\bm{k}(v,w)\bm{k}(w,w')w\cdot \nabla_{w}\bm{m}(w)\,,\end{multline*}
it holds, for $s=0,1$,
\begin{multline*}
\int_{\R^{d}}\sigma^{-1}(w)\d w\int_{\R^{d}}\left|w\cdot \nabla_{w}\bm{K}_{2}\left(v,w',w\right)\right|\sigma^{-s}(v)\bm{m}(\d v) \\
\leq \int_{\R^{d}}\sigma^{-1}(w)\bm{k}(w,w')\bm{m}(\d w)\int_{\R^{d}}\left|w\cdot \nabla_{w}\bm{k}(v,w)\right|\sigma^{-s}(v)\bm{m}(\d v) \\+ 
\int_{\R^{d}}\sigma^{-1}(w) \left(\sigma(w)\left|w \cdot \nabla_{w}\bm{k}(w,w')\right|\right)\vartheta_{s}(w)\bm{m}(\d w)\\
+ \int_{\R^{d}}\vartheta_{s}(w)\bm{k}(w,w')\left|w \cdot \nabla_{w}\log \bm{m}(w)\right|\bm{m}(\d w)
\end{multline*}
which results in 
$$\int_{\R^{d}}\sigma^{-1}(w)\d w\int_{\R^{d}}\left|w\cdot \nabla_{w}\bm{K}_{2}\left(v,w',w\right)\right|\sigma^{-s}(v)\bm{m}(\d v) \leq C \sigma(w')$$
with $C=C_{1}+\|\vartheta_{s}\|_{\infty}\left(C_{2}+\sup_{w}|w\cdot \nabla_{w}\log \bm{m}(w)|\right)<\infty$  where we used \eqref{eq:varthetas}, \eqref{eq:K2}--\eqref{eq:logM} and the conservative assumption \eqref{eq:conservative}. For the estimate of $w \cdot \nabla \sigma(w)$, one simply observes that, due to the conservative assumption
$$w \cdot \nabla\sigma(w)=\int_{\R^{d}}w \cdot \nabla_{w}\bm{k}(v,w)\bm{m}(\d v)$$
so that
$$\left|w \cdot \nabla_{w}\sigma(w)\right| \leq \int_{\R^{d}}|w \cdot \nabla_{w}\bm{k}(v,w)|\bm{m}(\d v) \leq C_{1}\sigma(w)$$
according to \eqref{eq:K2}.
\end{proof}
 Let $\l=\e+i\eta$,  $\e\geq0,$ $\eta \in \R$. The proof of Theorem \ref{theo:decay-KRA} consists in actually estimating 
 $$\|\K\Rs(\l,\A)\K\|_{\B(\X_{-1},\X_{0})}.$$
One checks  easily that
\begin{multline*}
\K\Rs(\l,\A)\K\,f(x,v)= \int_{V}\bm{k}(v,w)\bm{m}(\d w)\\
\int_{0}^{\infty}\exp\left(-\left(\l+\sigma(w)\right)t\right)\d t\int_{V}\bm{k}(w,w')f(x-tw,w')\bm{m}(\d w').\end{multline*}
Therefore
\begin{multline*}
\K\Rs(\l,\A)\K\,f(x,v)=\int_{V}\bm{m}(\d w')\int_{0}^{\infty}\exp\left(-\left(\l +\sigma(w)\right)t\right)\d t\\
\int_{V}\bm{k}(v,w)\bm{k}(w,w')f(x-tw,w')\bm{m}(\d w)\,.\end{multline*}
We split the integral over time into $\int_{\delta}^{\infty} + \int_{0}^{\delta}$ and write accordingly
$$\K\Rs(\l,\A)\K=\mathcal{U}_{\delta}(\l)+\mathcal{U}_{\delta}'(\l)$$
where, in $\mathcal{U}_{\delta}(\l)$, we restrict the time integral to the set $[\delta,\infty)$ and $\mathcal{U}_{\delta}'(\l)$ is defined with the time integral over $(0,\delta)$. The estimate of $\|\mathcal{U}_{\delta}'(\l)\|_{\B(\X_{-1},\X_{0})}$ is the easiest one.
\begin{lemme}\label{lem:Udelta'} For any $\delta >0$ it holds
$$\|\mathcal{U}_{\delta}'(\l)f\|_{\X_{s}}\leq \delta\,\|\vartheta_{s}\|_{\infty}\|\sigma\|_{\infty}\|f\|_{\X_{-1}}$$
for any $f \in \X_{-1}$ and any $s \leq N_{0}$ (where we recall that $\vartheta_{0}\equiv 1$). \end{lemme}
\begin{proof} By definition
$$\mathcal{U}_{\delta}'(\l)f(x,v)=\int_{V}\bm{m}(\d w')\int_{0}^{\delta}\exp\left(-\left(\l +\sigma(w)\right)t\right)\d t\\
\int_{V}\bm{k}(v,w)\bm{k}(w,w')f(x-tw,w')\bm{m}(\d w)$$
for any $f \in \X_{-1},$ $(x,v) \in \O\times V,$ $\l \in \overline{\C}_{+}$. Thus,
\begin{multline*}\|\mathcal{U}_{\delta}'(\l)f\|_{\X_{s}} \leq \int_{\T^{d}}\d x \int_{V} \sigma^{-s}(v)\bm{m}(\d v)\\
\int_0^{\delta}\d t \int_{V}\bm{m}(\d w)\int_{V}\bm{k}(v,w)\bm{k}(w,w')|f(x-tw,w')|\bm{m}(\d w').\end{multline*}
One easily sees, with the change of variables $x \mapsto y=x-tw$ that
\begin{equation*}\begin{split}\|\mathcal{U}_{\delta}'(\l)f\|_{\X_{s}}& \leq \int_{\T^{d}\times V}|f(y,w')|\d y \bm{m}(\d w') \\
&\phantom{++++++}\int_{0}^{\delta}\d t
\int_{V\times V}\sigma^{-s}(v)\bm{k}(v,w)\bm{k}(w,w')\bm{m}(\d w)\bm{m}(\d v)\\
&\leq \delta \|\sigma\|_{\infty}\|\vartheta_{s}\|_{\infty}\int_{\T^{d}\times V}\sigma(w')|f(y,w')|\d y\,\bm{m}(\d w')
\end{split}\end{equation*}
where we used that $\mathrm{Re}\l \geq 0,$ $\sigma \geq0$ and
\begin{equation*}%\begin{split}
\int_{V\times V}\bm{k}(v,w)\bm{k}(w,w')\sigma^{-s}(v)\bm{m}(\d v)\bm{m}(\d w')=\int_{V}\sigma(w)\vartheta_{s}(w)\bm{k}(w,w')\bm{m}(\d w)\,. 
%\\&\leq \|\sigma\|_{\infty}\|\vartheta_{s}\|_{\infty}\sigma(w')\,.\end{split}
\end{equation*}
This gives the desired estimate.
\end{proof}
It remains to estimate $\|\mathcal{U}_{\delta}(\l)\|_{\B(\X_{-1},\X_{0})}.$ One begins with the following representation of $\mathcal{U}_{\delta}(\l)$
\begin{lemme}\label{lem:lem37}
For any $f \in \X_{-1}$, $\l \in \overline{\C}_{+}$ and any $(x,v) \in \O\times V$ one has
\begin{equation}\label{eq:Udel} 
\mathcal{U}_{\delta}(\l)f(x,v) =\int_{[0,1]^{d}\times V}\widetilde{H}_{\l}(x-y,v,w')f(y,w')\d y \bm{m}(\d w')
\end{equation}
where
\begin{multline*}
\widetilde{H}_{\l}(z,v,w')=\sum_{k\in \Z^{d}}\int_{\delta}^{\infty}t^{-d}\exp\left(-\l t\right)\,\exp\left(-t\sigma\left(\frac{z-k}{t}\right)\right)\times\\
 \times \bm{k}\left(v,\frac{z-k}{t}\right)\bm{k}\left(\frac{z-k}{t},w'\right)\bm{m}\left(\frac{z-k}{t}\right)\d t, \qquad z \in \R^{d}, v,w \in V.\end{multline*}
\end{lemme}
\begin{proof} By definition,
\begin{multline*}
\mathcal{U}_{\delta}(\l)f(x,v)=\int_{V}\bm{m}(\d w')\int_{\delta}^{\infty}\exp\left(-\left(\l +\sigma(w)\right)t\right)\d t\\
\int_{V}\bm{k}(v,w)\bm{k}(w,w')f(x-tw,w')\bm{m}(\d w).\end{multline*}
We recall that functions over $\T^{d}$ are identified with functions defined over $\R^{d}$ which are $[0,1]^{d}$-periodic. Given $t \geq \delta$, one performs the change of variable 
\begin{equation}\label{eq:change}
y:=x-tw, \qquad \d y=t^{d}\d w,\end{equation}
and   
\begin{multline*}
\mathcal{U}_{\delta}(\l)f(x,v)=\int_{\R^{d}\times V}f(y,w')\d y\,\bm{m}(\d w')\int_{\delta}^{\infty}t^{-d}\exp\left(-\l t\right)\\
\ind_{\frac{x-y}{t} \in V}(t)
\exp\left(-t\sigma\left(\frac{x-y}{t}\right)\right)\,\bm{k}\left(v,\frac{x-y}{t}\right)\bm{k}\left(\frac{x-y}{t},w'\right)\bm{m}\left(\frac{x-y}{t}\right)\d t\,.
\end{multline*}
Given $z \in \R^{d}, w' \in V$, we introduce the function 
$$H_{\l}(z,v,w')=\int_{\delta}^{\infty}t^{-d}\exp\left(-\l t\right)\ind_{\frac{z}{t} \in V}(t)
\exp\left(-t\sigma\left(\frac{z}{t}\right)\right)\,\bm{k}\left(v,\frac{z}{t}\right)\bm{k}\left(\frac{z}{t},w'\right)\bm{m}\left(\frac{z}{t}\right)\d t$$
so that
$$\mathcal{U}_{\delta}(\l)f(x,v)=\int_{\R^{d}\times V}H_{\l}(x-y,v,w')f(y,w')\d y\,\bm{m}(\d w').$$
Notice that $\R^{d}$ admits the following partition representation
$$\R^{d}=\bigcup_{k \in \Z^{d}}\left(k+[0,1]^{d}\right)$$
so that, using the $\Z^{d}$-periodicity of $f(\cdot,w')$, we have
\begin{equation*}\label{eq:Udel1}\begin{split}
\mathcal{U}_{\delta}(\l)f(x,v)&=\sum_{k \in \Z^{d}}\int_{k+[0,1]^{d}}\d y\int_{V}H_{\l}(x-y,v,w')f(y,w')\bm{m}(\d w')\\
&=\sum_{k \in \Z^{d}}\int_{[0,1]^{d}\times V}H_{\l}(x-y-k,v,w')f(y,w')\d y \bm{m}(\d w')\end{split}
\end{equation*}
which results in \eqref{eq:Udel}.
\end{proof}
With this representation, the key estimate is provided by the following
\begin{lemme}\label{lem:Udelta}There exists $C_{3} >0$ (depending only on $d,C_{1},C_{2},\|\sigma\|_{\infty}$, $\|v\cdot\nabla_{v}\log \bm{m}\|_{\infty}$ and $\|\vartheta_{s}\|_{\infty}$) such that
$$\int_{\R^{d}}\max\left(1,\sigma^{-1}(v)\right)\bm{m}(\d v)\int_{[0,1]^{d}}\left|\widetilde{H}_{\l}(x-y,v,w')\right|\d x\leq \frac{{C}_{3}}{|\l|}\left(\frac{1}{\delta}+1\right)\sigma(w')$$
for any $y, w' \in \O \times V$, $\delta >0,$ $\l \in \C_{+}\setminus\{0\}.$ 
\end{lemme}
\begin{proof} We begin with the following observation, valid for any $u \in \R^{d},$ and which is obtained thanks to a simple use of integration by parts, 
\begin{multline*}
\int_{\delta}^{\infty} t^{-d}\exp\left(-t\sigma\left(\frac{u}{t}\right)\right)\exp\left(-\l t\right)
\bm{k}\left(v,\frac{u}{t}\right)\bm{k}\left(\frac{u}{t},w'\right)\bm{m}\left(\frac{u}{t}\right)\d t \\
=\frac{1}{\lambda}\int_{\delta}^{\infty}\frac{\d}{\d t}\left[
t^{-d}\exp\left(-t\sigma\left(\frac{u}{t}\right)\right)\bm{k}\left(v,\frac{u}{t}\right)\bm{k}\left(\frac{u}{t},w'\right)\bm{m}\left(\frac{u}{t}\right)\right] \exp\left(-\lambda t\right)\d t \\
-\frac{1}{\lambda} t^{-d}\exp\left(-t\sigma\left(\frac{u}{t}\right)\right)\bm{k}\left(v,\frac{u}{t}\right)\bm{k}\left(\frac{u}{t},w'\right)\bm{m}\left(\frac{u}{t}\right)\bigg\vert_{t=\delta}^{\infty} \\
=\frac{1}{\lambda}\int_{\delta}^{\infty}\frac{\d}{\d t}\left[t^{-d}\exp\left(-t\sigma\left(\frac{u}{t}\right)\right)\bm{k}\left(v,\frac{u}{t}\right)\bm{k}\left(\frac{u}{t},w'\right)\bm{m}\left(\frac{u}{t}\right)\right] \exp\left(-\lambda t\right)\d t\\
+\frac{1}{\lambda}\delta^{-d}\exp\left(-\delta\sigma\left(\frac{u}{\delta}\right)\right)\bm{k}\left(v,\frac{u}{\delta}\right)\bm{k}\left(\frac{u}{\delta},w'\right)\bm{m}\left(\frac{u}{\delta}\right)\exp\left(-\lambda \delta\right)
\end{multline*}
 Recalling that we introduced $\bm{K}_{2}(v,w',w)=\bm{k}(v,w)\bm{k}(w,w')\bm{m}(w),$ for any $v,w,w' \in V^{3},$ one sees that, for $v,u,w' \in V^{3}$ and $t >\delta$, one has
alors
\begin{multline*}
\frac{\d}{\d t}\left[t^{-d}\exp\left(-t\sigma\left(\frac{u}{t}\right)\right)\bm{k}\left(v,\frac{u}{t}\right)\bm{k}\left(\frac{u}{t},w'\right)\bm{m}\left(\frac{u}{t}\right)\right]\\
= t^{-d}\exp\left(-t\sigma\left(\frac{u}{t}\right)\right)\,\left[-\frac{d}{t}-\sigma\left(\frac{u}{t}\right)+\frac{u}{t}\cdot \nabla \sigma\left(\frac{u}{t}\right)\right]\bm{K}_{2}\left(v,w',\frac{u}{t}\right)\\
+t^{-d}\exp\left(-t\sigma\left(\frac{u}{t}\right)\right)\,\left(-\frac{u}{t^{2}}\cdot \nabla_{3}\bm{K}_{2}\left(v,w',\frac{u}{t}\right)\right)
\end{multline*}
where $\nabla_{3}\bm{K}_{2}(v,w',w)$ denotes the gradient with respect to the third variable $w\in \R^{3}.$ In particular
$$\frac{1}{t}\nabla_{3}\bm{K}_{2}\left(v,w',\frac{u}{t}\right)=\nabla_{u}\bm{K}_{2}\left(v,w',\frac{u}{t}\right).$$
From this, one has
\begin{multline*}
\int_{\delta}^{\infty} t^{-d}\exp\left(-t\sigma\left(\frac{u}{t}\right)\right)\exp\left(-\l t\right)
\bm{k}\left(v,\frac{u}{t}\right)\bm{k}\left(\frac{u}{t},w'\right)\bm{m}\left(\frac{u}{t}\right)\d t \\
=\frac{1}{\l}\int_{\delta}^{\infty}t^{-d}\exp\left(-t\sigma\left(\frac{u}{t}\right)\right)\,\left[-\frac{d}{t}-\sigma\left(\frac{u}{t}\right)+\frac{u}{t}\cdot \nabla \sigma\left(\frac{u}{t}\right)\right]\bm{K}_{2}\left(v,w',\frac{u}{t}\right)\exp\left(-\l t\right)\d t\\
-\frac{1}{\l}\int_{\delta}^{\infty}t^{-d}\exp\left(-t\sigma\left(\frac{u}{t}\right)\right)\frac{u}{t}\cdot \nabla_{u}\bm{K}_{2}\left(v,w',\frac{u}{t}\right)\exp\left(-\l t\right)\d t\\
+\frac{1}{\lambda}\delta^{-d}\exp\left(-\delta\sigma\left(\frac{u}{\delta}\right)\right)\bm{K}_{2}\left(v,w',\frac{u}{\delta}\right)\exp\left(-\lambda \delta\right)\,.
\end{multline*}
Consequently, recalling that
$$\widetilde{H}_{\l}(z,v,w')=\sum_{k\in \Z^{d}}\int_{\delta}^{\infty} t^{-d}\exp\left(-t\sigma\left(\frac{z-k}{t}\right)\right)\exp\left(-\l t\right)
\bm{K}_{2}\left(v,w',\frac{z-k}{t}\right)\d t$$
one deduces that
\begin{multline*}
\left|\widetilde{H}_{\l}(z,v,w')\right|\leq \frac{1}{|\l|}\int_{\delta}^{\infty}t^{-d}\exp\left(-t\sigma\left(\frac{z-k}{t}\right)\right)\bm{K}_{2}\left(v,w',\frac{z-k}{t}\right)\\
\left|-\frac{d}{t}-\sigma\left(\frac{z-k}{t}\right)+\frac{z-k}{t}\cdot \nabla \sigma\left(\frac{z-k}{t}\right)\right| \d t\\
+\frac{1}{|\l|}\sum_{k\in \Z^{d}}\int_{\delta}^{\infty}t^{-d}\exp\left(-t\sigma\left(\frac{z-k}{t}\right)\right)\left|\frac{z-k}{t}\cdot \nabla_{z}\bm{K}_{2}\left(v,w',\frac{z-k}{t}\right)\right|\d t\\
+\frac{1}{|\lambda|}\delta^{-d}\sum_{k\in Z^{d}}\exp\left(-\delta\sigma\left(\frac{z-k}{\delta}\right)\right) \bm{K}_{2}\left(v,w',\frac{z-k}{\delta}\right).
\end{multline*}
Therefore, using periodicity once again,
\begin{multline*}
\int_{[0,1]^{d}}\left|\widetilde{H}_{\l}(x-y,v,w')\right|\d x\leq \frac{1}{|\l|}\int_{\delta}^{\infty}t^{-d}\d t\int_{\R^{d}}\exp\left(-t\sigma\left(\frac{x-y}{t}\right)\right)\\
\left(\frac{d}{\delta}+\sigma\left(\frac{x-y}{t}\right)+\left|\frac{x-y}{t}\cdot \nabla \sigma\left(\frac{x-y}{t}\right)\right|\right)\bm{K}_{2}\left(v,w',\frac{x-y}{t}\right) \d x\\
+\frac{1}{|\l|}\int_{\delta}^{\infty}t^{-d}\d t\int_{\R^{d}}\exp\left(-t\sigma\left(\frac{x-y}{t}\right)\right)\left|\frac{x-y}{t}\cdot \nabla_{x}\bm{K}_{2}\left(v,w',\frac{x-y}{t}\right)\right|\d x\\
+\frac{1}{|\lambda|}\delta^{-d}\int_{\R^{d}}\exp\left(-\delta\sigma\left(\frac{x-y}{\delta}\right)\right) \bm{K}_{2}\left(v,w',\frac{x-y}{\delta}\right)\d x
\end{multline*}
Performing the backward change of variables with respect to \eqref{eq:change}, i.e. setting $x \mapsto w=\frac{x-y}{t}$ for the first two integrals  and using that
$$\int_{\delta}^{\infty}\exp\left(-t\sigma(w)\right)\d t \leq \frac{1}{\sigma(w)}, \qquad \forall w \in \R^{d},$$ 
one deduces that
\begin{multline*}
\int_{[0,1]^{d}}\left|\widetilde{H}_{\l}(x-y,v,w')\right|\d x\leq \frac{1}{|\l|}\int_{\R^{d}} 
\left(\frac{d}{\delta}+\sigma\left(w\right)+\left|w\cdot \nabla \sigma\left(w\right)\right|\right)\bm{K}_{2}\left(v,w',w\right)\sigma^{-1}(w) \d w\\
+\frac{1}{|\l|} \int_{\R^{d}} \left|w\cdot \nabla_{w}\bm{K}_{2}\left(v,w',w\right)\right|\sigma^{-1}(w) \d w\\
+\frac{1}{|\lambda|} \int_{\R^{d}}\exp\left(-\delta\sigma\left(w\right)\right) \bm{K}_{2}\left(v,w',w\right)\d w
\end{multline*}
where we performed the change of variable $x \mapsto w=\frac{x-y}{\delta}$ for the last integral. Integrating now with respect to $v\in \R^{d}$, we deduce that, for $s=0,1$,
\begin{multline*}
\int_{\R^{d}}\sigma^{-s}(v)\bm{m}(\d v)\int_{[0,1]^{d}}\left|\widetilde{H}_{\l}(x-y,v,w')\right|\d x\\
\leq \frac{1}{|\l|}\int_{\R^{d}\times\R^{d}} 
\left(\frac{d}{\delta}+\sigma\left(w\right)+\left|w\cdot \nabla \sigma\left(w\right)\right|\right)\bm{K}_{2}\left(v,w',w\right)\sigma^{-1}(w) \d w \sigma^{-s}(v)\bm{m}(\d v)\\
+\frac{1}{|\l|} \int_{\R^{d}\times\R^{d}} \left|w\cdot \nabla_{w}\bm{K}_{2}\left(v,w',w\right)\right|\sigma^{-1}(w) \d w \sigma^{-s}(v)\bm{m}(\d v)\\
+\frac{1}{|\lambda|} \int_{\R^{d}\times \R^{d}}\exp\left(-\delta\sigma\left(w\right)\right) \bm{K}_{2}\left(v,w',w\right) \d w \sigma^{-s}(v)\bm{m}(\d v)\,.
\end{multline*}
Noticing that
$$\int_{\R^{d}}\bm{K}_{2}\left(v,w',w\right)\sigma^{-s}(v)\bm{m}(\d v)=\sigma(w)\vartheta_{s}(w)\bm{k}(w,w')\bm{m}(w)$$
and using Lemma \ref{lem:K2kk}, we  obtain that, or $s=0,1$
\begin{multline*}
\int_{\R^{d}}\sigma^{-s}(v)\bm{m}(\d v)\int_{[0,1]^{d}}\left|\widetilde{H}_{\l}(x-y,v,w')\right|\d x\\
\leq \frac{1}{|\l|}\left(\frac{d}{\delta}+\|\sigma\|_{\infty} +C_{1}\|\sigma\|_{\infty}\right)\|\vartheta_{s}\|_{\infty}\int_{\R^{d}} \bm{k}\left(w,w'\right) \bm{m}(\d w)\\
+\frac{C}{|\l|}\sigma(w')+\frac{\|\sigma\|_{\infty}\|\vartheta_{s}\|_{\infty}}{|\lambda|} \int_{\R^{d} } \bm{k} \left(w,w'\right)\bm{m}(\d w)\,
\end{multline*}
from which we easily derive the conclusion.
\end{proof}
\begin{nb} We insist on the fact that the change of variables \eqref{eq:change} (and its backward counterpart) is the only part of our analysis in which we use the fact that $\bm{m}(\d v)$ is absolutely continuous with respect to the Lebesgue measure. We strongly believe that, at a price of additional technical calculations, our analysis can be extended to a larger class of measures $\bm{m}(\d v)$ satisfying \eqref{eq:hypm}.\end{nb}
We have all at hands to prove the result
\begin{proof}[Proof of Theorem \ref{theo:decay-KRA}] Recalling that, for any $\delta >0$ and any $f \in \X_{-1}$
$$\K\Rs(\l,\A)\K=\mathcal{U}_{\delta}(\l)+\mathcal{U}_{\delta}'(\l)$$
we deduce from \eqref{eq:Udel} that, for $s=0,1$,
\begin{multline*}
\|\K\Rs(\l,\A)\K f\|_{\X_{s}} \leq \|\mathcal{U}_{\delta}'(\l)f\|_{\X_{s}} \\+ \int_{\O\times V}\left|f(y,w')\right|\d y\,\bm{m}(\d w')\int_{V}\sigma^{-s}(v)\bm{m}(\d v)\int_{[0,1]^{d}}\left|\widetilde{H}_{\l}(x-y,v,w')\right|\d x\\
\leq \delta\|\sigma\|_{\infty}\|\vartheta_{s}\|_{\infty}\|f\|_{\X_{-1}} + \frac{C_{3}}{|\l|}\left(\frac{1}{\delta}+1\right)\|f\|_{\X_{-1}}
\end{multline*}
where we used Lemmas \ref{lem:Udelta'} and \ref{lem:Udelta} in the second estimate. Optimizing now the parameter $\delta$, we deduce that
$$\|\K\Rs(\l,\A)\K\|_{\B(\X_{-1},\X_{s})} \leq 2\sqrt{\frac{C_{3}\|\sigma\|_{\infty}}{|\l|}}+\frac{C_{3}}{|\l|}, \qquad \forall \l \in \C_{+}\setminus\{0\}.$$
For $s=0$, we deduce from this that there exists $C_{4} >0$ such that, for $|\l| >1$,
$$\left\|\left[\K\Rs(\l,\A)\right]^{2}\right\|_{\B(\X_{0})} \leq \|\K\Rs(\l,\A)\K\|_{\B(\X_{-1},\X_{0})}\|\Rs(\l,\A)\|_{\B(\X_{0},\X_{-1})} \leq \frac{C_{4}}{\sqrt{|\l|}}$$
since $\sup_{\l \in \overline{\C}_{+}}\|\Rs(\l,\A)\|_{\B(\X_{0},\X_{-1})} \leq 1.$ Now, with $s=1$, we deduce in the same way that there is $C_{5} >0$ such that
$$\left\|\left[\Rs(\l,\A)\K\right]^{2}\right\|_{\B(\X_{0})} \leq \|\Rs(\l,\A)\|_{\B(\X_{1},\X_{0})}\|\K\Rs(\l,\A)\K\|_{\B(\X_{0},\X_{1})}\leq \frac{C_{5}}{\sqrt{|\l|}}$$
since $\|\K\Rs(\l,\A)\K\|_{\B(\X_{0},\X_{1})} \leq \|\sigma\|_{\infty}\|\K\Rs(\l,\A)\K\|_{\B(\X_{-1},\X_{1})}$. 
This proves the result.  \end{proof}
%\begin{nb} We deduce from \eqref{eq:decayKRA} that
%$$\left\|\left[\K\Rs(\l,\A)\right]^{2n}\right\|_{\B(\X_{0})} \leq \frac{C^{n}}{|\l|^{\frac{n}{2}}}, \qquad \forall \l \in \C_{+}, |\l| \geq 1, \qquad n \geq 1$$
%and in particular \eqref{eq:power} holds true for any $\mathsf{p} >4$. 
%\end{nb}

\section{Regularity and extension results}\label{sec:reguTr}

{From now, we will always assume that Assumptions \ref{hypH} and \ref{hypK} are in force.}

\subsection{About the resolvent of $\A$}

We recall that under Assumption \eqref{spahomo} one has
$$0 \in \mathfrak{S}(\A).$$
More precisely, since $\A$ generates a $C_{0}$-group in $\X_{0}$, one can prove (see \cite{mkposi}) that there exists $\lambda_{\star} >0$ such that
\begin{equation}\label{eq:SpecA}
\mathfrak{S}(\A)=\{\l \in \C\;;\;-\lambda_{\star} \leq \mathrm{Re}\l \leq 0\}.\end{equation}
This shows in particular that 
$$i\eta \in \mathfrak{S}(\A) \qquad \qquad \forall \eta \in \R.$$ 
 Regarding the behaviour of $\Rs(\e+i\eta,\A)f$, one first observes that, by virtue of \eqref{eq:SpecA}
\begin{equation}\label{eq:epsi-s-A}
\limsup_{\e\to 0^{+}}\left\|\Rs(\e+i\eta,\A)\right\|_{\mathscr{B}(\X_{0})}=\infty.\end{equation}
However, studying the behaviour of $\Rs(\e+i\eta,\A)f$ on the hierarchy of spaces $\X_{k}$, $k \geq 1$ yields better estimates:
\begin{propo}\phantomsection\label{prop:convRsT0}
For any $f \in \X_{0}$ and $\e >0$, the mapping 
$$\eta \in \R \longmapsto \Rs(\e+i\eta,\A)f \in \X_{0}$$
belongs to $\mathscr{C}_{0}^{k}(\R,\X_{0})$ for any $k \in \N.$ Moreover, given $s \in \R$, for any $f \in \X_{s+1}$, 
$$\lim_{\e\to0^{+}}\Rs(\e+i\eta,\A)f$$ exists in $\mathscr{C}_{0}(\R,\X_{s})$. This limit is denoted $\Rs(i\eta,\A)f,$ i.e.
\begin{equation}\label{eq:RetaXs}
\lim_{\e\to0^{+}}\sup_{\eta\in\R}\left\|\Rs(\e+i\eta,\A)f-\Rs(i\eta,\A)f\right\|_{\X_{s}}=0 \qquad \forall f \in \X_{s+1}.\end{equation}
\end{propo}
\begin{proof} We begin with the first part of the Proposition.
Observing that
$$\Rs(\e+i\eta,\A)f=\int_{0}^{\infty}e^{-i\eta t}e^{-\e t}U_{0}(t)f\d t$$
where the mapping $t \in \R \mapsto e^{-\e t}U_{0}(t)f \in \X_{0}$ is Bochner integrable, one deduces from Riemann-Lebesgue Theorem that the mapping $\eta \in \R \mapsto \Rs(\e+i\eta,\A)f \in \X_0$ belongs to $\mathscr{C}_{0}(\R,\X_0)$, in particular
\begin{equation}\label{eq:Rsieta}
\lim_{|\eta|\to\infty}\left\|\Rs(\e+i\eta,\A)f\right\|_{\X_{0}}=0.\end{equation}
 Given $k \in \N$, because 
$$\dfrac{\d^{k}}{\d\eta^{k}}\Rs(\e+i\eta,\A)f=(-i)^{k}\int_{0}^{\infty}t^{k}e^{-i\eta t} e^{-\e t}U_{0}(t)f\,\d t$$
the exact same argument shows that 
$$\lim_{|\eta|\to\infty}\left\|\frac{\d^{k}}{\d\eta^{k}}\Rs(\e+i\eta,\A)f\right\|_{\X_{0}}=0$$
which proves that $\eta \in \R \mapsto \Rs(\e+i\eta,\A)f$ belongs to $\mathscr{C}_{0}^{k}(\R,\X_{0}).$ Let us focus now on the limit as $\e \to 0^{+}$. 
Let $f \in \X_{1}$ be given, we deduce from Lemma \ref{lem:decayU0} and the dominated convergence theorem that
$$\lim_{\e\to0^{+}}\Rs(\e+i\eta,\A)f=\lim_{\e\to0^{+}}\int_{0}^{\infty}e^{-i\eta t}e^{-\e t}U_{0}(t)f \d t=\int_{0}^{\infty}e^{-i\eta t}U_{0}(t)f\d t$$
exists in $\X_{0}$. The limit is of course denoted $\Rs(i\eta,\A)f$ and one has
$$
\left\|\Rs(\e+i\eta,\A)f-\Rs(i\eta,\A)f\right\|_{\X_{0}} \leq \int_{0}^{\infty}\left|e^{-\e t}-1\right|\,\|U_{0}(t)f\|_{\X_{0}}\d t\, \qquad \forall \eta \in\R,\,\e >0.$$
Thus
\begin{equation}\label{eq:RsT0ieta}
\lim_{\e\to0^{+}}\sup_{\eta\in \R}\left\|\Rs(\e+i\eta,\A)f-\Rs(i\eta,\A)f\right\|_{\X_{0}}=0\end{equation}
still using the fact that $t \mapsto \|U_{0}(t)f\|_{\X_{0}}$ is integrable over $[0,\infty)$ and the dominated convergence theorem. Now, given $s \in \R$ and $f \in \X_{s+1}$, setting 
$$g(x,v)=\sigma^{s}(v)f(x,v)$$
one sees that $g \in \X_{1}$ and $U_{0}(f)g(x,v)=\sigma^{s}(v)U_{0}(t)f(x,v)$. Applying \eqref{eq:RsT0ieta} to $g$, we deduce that
$$\lim_{\e\to0^{+}}\sup_{\eta\in \R}\left\|\Rs(\e+i\eta,\A)f-\Rs(i\eta,\A)f\right\|_{\X_{s}}=0$$
and the result follows.
\end{proof}
\begin{nb}\label{nb:Ck} One deduces from the above Proposition and  Banach-Steinhaus Theorem  \cite[Theorem 2.2, p. 32]{brezis}    that 
\begin{equation}\label{eq:Ck}
C_{s}:=\sup\left\{\left\|\Rs(\e+i\eta,\A)\right\|_{\mathscr{B}(\X_{s+1},\X_{s})}\;;\;\e \in (0,1]\;;\;\eta \in \R\right\} <\infty \qquad \forall s \in \R.\end{equation}
\end{nb}
We deduce then the following
\begin{cor}\label{cor:double} Given $s  \in \R$ and $I \subset \R$ be a given compact interval. If 
$$g\::\:\l \in \C_{+} \longmapsto g(\l) \in \X_{s+1}$$
is a continuous mapping such that the limit
$$\widetilde{g}(\eta):=\lim_{\e\to0^{+}}g(\e+i\eta)$$
exists in $\X_{s+1}$ uniformly with respect to $\eta \in I$. Then
$$\lim_{\e\to 0^{+}}\Rs(\e+i\eta,\A)g(\e+i\eta)=\Rs(i\eta,\A)\widetilde{g}(\eta)$$
in $\X_{s}$ where the convergence is uniform with respect to $\eta \in I$.\end{cor}
\begin{proof} The proof is a simple adaptation of the one in \cite[Corollary 4.14]{MKL-JFA}. Details are omitted.\end{proof}
\begin{nb}\label{nb:I=R} If $I=\R$, the above result still holds true under the additional assumption that
$$\lim_{|\eta|\to\infty}\|g(\e+i\eta)\|_{\X_{s+1}}=0 \qquad \forall \e >0$$
and, in such a case, the convergence of $\Rs(\e+i\eta,\A)g(\e+i\eta)$ towards $\Rs(i\eta,\A)\widetilde{g}(\eta)$ holds in $\mathscr{C}_{0}(\R,\X_{s})$.
\end{nb}
The convergence established in Prop. \ref{prop:convRsT0} extends to derivatives of $\Rs(\e+i\eta,\A)f$
\begin{lemme}\label{lem:convDerRsT0} Given $k \in \N$ and $f \in \X_{k+1}$, it holds
\begin{equation*}
\lim_{\e\to0^{+}}\sup_{\eta \in \R}\left\|\dfrac{\d^{k}}{\d\eta^{k}}\Rs(\e+i\eta,\A)f-\dfrac{\d^{k}}{\d\eta^{k}}\Rs(i\eta,\A)f\right\|_{\X_{0}}=0.
\end{equation*}
Consequently, the mapping
$$\eta \in \R \longmapsto \Rs(i\eta,\A)f \in \X_{0}$$
belongs to $\mathscr{C}_{0}^{k}(\R,\X_{0}).$
\end{lemme}
\begin{proof} As already established
$$\dfrac{\d^{k}}{\d\eta^{k}}\Rs(\e+i\eta,\A)f=(-i)^{k}\int_{0}^{\infty}e^{-i\eta t}t^{k}e^{-\e t}U_{0}(t)f\d t, \qquad \e >0$$
and, since 
$$\Rs(i\eta,\A)f=\int_{0}^{\infty}e^{-i\eta t}U_{0}(t)f\d t$$
one sees easily that, if $f \in \X_{k+1}$, 
$$\frac{\d^{k}}{\d\eta^{k}}\Rs(i\eta,\A)f=(-i)^{k}\int_{0}^{\infty}e^{-i\eta t}t^{k}U_{0}(t)f\d t$$
is well-defined in $\X_{0}$ thanks to Lemma \ref{lem:decayU0}. One concludes then exactly as in Prop. \ref{prop:convRsT0}. 
\end{proof}
\begin{nb} The above formula extends trivially to derivative of $\Rs(\l,\A)$ with respect to $\l$ and one has
\begin{equation}\label{eq:dlkRslA}
\sup_{\l \in \overline{\C}_{+}}\left\|\frac{\d^{k}}{\d \l^{k}}\Rs(\l,A)f\right\|_{\X_{0}} \leq \int_{0}^{\infty}t^{k}\left\|U_{0}(t)f\right\|_{\X_{0}}\d t \leq \Gamma(k+1)\|f\|_{\X_{k+1}}
\end{equation}
thanks to \eqref{eq:uoK}.
\end{nb}
\subsection{Definition and properties of $\Ms_{\l}$}

The above results allow us to  define a \emph{bounded} linear operator
$$\Ms_{i\eta} \in \mathscr{B}(\X_{0})$$
as the \emph{strong limit} of $\K\Rs(\e+i\eta,\A)$ as $\e\to0^{+}$, i.e.
$$\Ms_{i\eta}\varphi:=\lim_{\e\to0^{+}}\K\Rs(\e+i\eta,\A)\varphi, \qquad \forall \varphi \in \X_{0}$$
where the limit is meant in $\X_{0}.$ Indeed, recall that, for any $f \in \X_{0}$,
$$\lim_{\e\to0^{+}}\sup_{\eta \in \R}\left\|\Rs(\e+i\eta,\A)f-\Rs(i\eta,\A)f\right\|_{\X_{-1}}=0$$
and, since $\K \in \mathscr{B}(\X_{-1},\X_{0})$, one deduces that
\begin{equation}\label{eq:stronMie}
\lim_{\e\to0}\sup_{\eta \in\R}\left\|\K\Rs(\e+i\eta,\A)f-\Ms_{i\eta}f\right\|_{\X_{0}}=0,\end{equation}
where
$$\Ms_{i\eta} =\K\Rs(i\eta,\A) \in \mathscr{B}(\X_{0}).$$
\begin{nb} It is easy to check that $\Ms_{i\eta}$ given by,
$$\Ms_{i\eta}f(x,v):=\int_{V}\bm{k}(v,w)\bm{m}(\d w)\int_{0}^{\infty}\exp\left(-t\left(i\eta+\sigma(w)\right)\right)f(x-tw,w)\d t, \qquad f \in \X_{0}.$$
In particular, one sees that
\begin{equation*}\begin{split}
\|\Ms_{i\eta}f\|_{\X_{0}} &\leq \int_{\T^{d}\times V}\d x\,\bm{m}(\d v)\int_{V}\bm{k}(v,w)\bm{m}(\d w)\int_{0}^{\infty}\exp\left(-t\sigma(w)\right)|f(x-tw,w)\d t\\
&\leq \int_{\T^{d}\times V}\sigma^{-1}(w)|f(y,w)|\d y\,\bm{m}(\d w)\int_{V}\bm{k}(v,w)\bm{m}(\d v)
\end{split}\end{equation*}
where we used the change of variable $x \mapsto y=x-tw$ to compute the integral over $\T^{d}$. Using then assumption \eqref{eq:conservative}, 
we obtain
$$\left\|\Ms_{i\eta}\right\|_{\mathscr{B}(\X_{0})} \leq 1.$$
Notice that, with straightforward computations, one has
\begin{multline}\label{eq:KReieta}
\left[\K\Rs(\e+i\eta,\A)f-\Ms_{i\eta}f\right](x,v)\\
=\int_{V}\bm{k}(v,w)\bm{m}(\d w)
\int_{0}^{\infty}\left[\exp\left(-\e t\right)-1\right]\exp\left(-\left(i\eta+\sigma(w)\right)t\right)f(x-tw,w)\d t\end{multline}
for almost every $(x,v) \in \T^{d}\times V$ and any $f\in \X_{0}$. In particular,
\begin{equation}\label{eq:dif}
\sup_{\eta \in \R}\|\K\Rs(\e+i\eta,\A)f-\Ms_{i\eta}f\|_{\X_{0}} \leq
\int_{\T^{d}\times V}\left|f(y,w)\right|\,\left[1-\frac{\sigma(w)}{\e+\sigma(w)}\right]\d y\,\bm{m}(\d w)\,.\end{equation}
This allows to recover \eqref{eq:stronMie}
 by a simple use of the dominated convergence theorem.\end{nb}
%\begin{nb} Notice that, choosing $\varphi(x,v)=\Phi(v) \geq0$, $\|\Phi\|_{L^{1}(V)}=1$ shows that 
%$$\left\|\Ms_{0}\right\|_{\mathscr{B}(\X_{0})}=1.$$
%Moreover, one notices that, for any $\varphi \in \X_{0}$, the limit
%$$\lim_{t\to\infty}\K\int_{0}^{t}U_{0}(\tau)\varphi\d \tau=:\K\int_{0}^{\infty}U_{0}(\tau)\varphi\d \tau$$
%exists in $\X_{0}$ 
%and one has
%$$\Ms_{0}\varphi=\K\int_{0}^{\infty}U_{0}(\tau)\varphi\d \tau$$
%(see \cite{ergodic} for details). \end{nb}

From the above definition, and with a slight abuse of notations, we set
$$\Ms_{\l}:=\K\Rs(\l,\A), \qquad \forall \mathrm{Re}\l \geq0$$
so that $\Ms_{i\eta}$ is the strong limit of $\Ms_{\e+i\eta}$ in $\X_{0}$.  One actually has the following regularizing properties of $\Ms_{\e+i\eta}$:
\begin{lemme} For any $k \in \{0,\ldots,N_{0}\}$ one has
\begin{equation}\label{eq:MsSUP}
\sup\left\{\left\|\Ms_{\l}\right\|_{\mathscr{B}(\X_{0},\X_{k})}\;,\;\mathrm{Re}\l \geq 0\right\} \leq \|\vartheta_{k}\|_{\infty} < \infty.\end{equation}
\end{lemme}
\begin{proof} Given $f \in \X_{0}$, $\mathrm{Re}\l \geq0$ and $k \geq0$, one has
$$\|\Ms_{\l}f\|_{\X_{k}} \leq \|\Ms_{0}f\|_{\X_{k}}$$ 
and it suffices to prove the result for $\Ms_{0}f$. One checks easily that
$$\|\Ms_{0}f\|_{\X_{k}} \leq \int_{\T^{d}\times V}\left|f(y,w)\right|\d y\,\bm{m}(\d w)\int_{V}\bm{k}(v,w)\sigma^{-k}(v)\bm{m}(\d v)\int_{0}^{\infty}\exp\left(-\sigma(w)t\right)\d t$$
i.e.
$$\|\Ms_{0}f\|_{\X_{1}}\leq \int_{\T^{d}\times V}\left|f(y,w)\right|\d y\frac{\bm{m}(\d w)}{\sigma(w)}\int_{V}\bm{k}(v,w)\frac{\bm{m}(\d v)}{\sigma^{k}(v)}=\int_{\T^{d}\times V}\vartheta_{k}(w)\left|f(y,w)\right|\d y\,\bm{m}(\d w)$$
which proves the result since $\vartheta_{k}\in L^{\infty}(V,\bm{m})$ as soon as $k \leq N_{0}$.
\end{proof}
\begin{nb}\label{nb:MlRel} The same proof shows that, actually, 
$$\|\Ms_{\l}f\|_{\X_{0}}\leq \int_{\T^{d}\times V}\frac{\sigma(w)}{\mathrm{Re}\l+\sigma(w)}\,\left|f(y,w)\right|\d y\,\bm{m}(\d w)$$
for any $f \in \X_{0},$ $\l \in \overline{\C}_{+}$. Thefore,
\begin{equation}\label{eq:KK}
\|\Ms_{\l}\|_{\mathscr{B}(\X_{0})} \leq \sup_{w}\frac{\sigma(w)}{\mathrm{Re}\l+\sigma(w)}=\frac{\|\sigma\|_{\infty}}{\mathrm{Re}\l+\|\sigma\|_{\infty}}\end{equation}
since the mapping $x \mapsto \frac{x}{\mathrm{Re}\l+x}$ is increasing for $\mathrm{Re}\l >0.$ \end{nb}

 With this, one can prove that the above convergence  in \eqref{eq:stronMie}  extends to the stronger norm $\X_{k}$ for $k \leq N_{0}$ and to  iterations of $\K\Rs(\e+i\eta,\A)$:
\begin{lemme}\label{lem:MeMi} 
Let $f \in \X_{0}$. Then
\begin{equation}\label{eq:stronMiek}
\lim_{\e\to0^{+}}\sup_{\eta \in\R}\left\|\K\Rs(\e+i\eta,\A)f-\Ms_{i\eta}f\right\|_{\X_{k}}=0, \qquad \forall k \leq N_{0}.\end{equation}
Moreover, for  any $j \in \N$ 
\begin{equation}\label{cor:uniformpower}
\lim_{\e\to0^{+}}\left\|\left[\K\Rs(\e+i\eta,\A)\right]^{j}f-\Ms_{i\eta}^{j}f\right\|_{\X_{k}}=0\end{equation}
uniformly with respect to $\eta \in \R$ for any $k \leq N_{0}.$
\end{lemme}
\begin{proof} Recalls that \eqref{eq:KReieta} gives the expression of $\K\Rs(\e+i\eta,\A)f-\Ms_{i\eta}f$. In particular, for any $k \in \R$,
\begin{multline*}
\sup_{\eta \in \R}\|\K\Rs(\e+i\eta,\A)f-\Ms_{i\eta}f\|_{\X_{k}} \\
\leq
\int_{\T^{d}\times V}\left|f(y,w)\right|\,\left[1-\frac{\sigma(w)}{\e+\sigma(w)}\right]\left(\frac{1}{\sigma(w)}\int_{V}\sigma^{-k}(v)\bm{k}(v,w)\bm{m}(\d v)\right)\d y\,\bm{m}(\d w)\\
=\int_{\T^{d}\times V}\left|f(y,w)\right|\,\left[1-\frac{\sigma(w)}{\e+\sigma(w)}\right]\vartheta_{k}(w)\d y\,\bm{m}(\d w).
\end{multline*}
Thus, for $k \leq N_{0}$, since $\vartheta_{k} \in L^{\infty}(V)$, we deduce that
\begin{equation}\label{eq:dif}
\sup_{\eta \in \R}\|\K\Rs(\e+i\eta,\A)f-\Ms_{i\eta}f\|_{\X_{k}} \leq
\|\vartheta_{k}\|_{\infty}\int_{\T^{d}\times V}\left|f(y,w)\right|\,\left[1-\frac{\sigma(w)}{\e+\sigma(w)}\right]\d y\,\bm{m}(\d w)\,\end{equation}
and recover \eqref{eq:stronMiek}
 by a simple use of the dominated convergence theorem. Let us prove now \eqref{cor:uniformpower} by induction on $j \in \N.$ For $j=1$, the result is exactly \eqref{eq:stronMie}. Assume the result to be true for $j \geq 1$. Given $f \in \X_{0}$, set 
$$g(\e,\eta)=\K\Rs(\e+i\eta,\A)f.$$
Since $\lim_{\e\to0}g(\e,\eta)=\Ms_{i\eta}f $ in $\mathscr{C}_{0}(\R,\X_{0})$ then
$$\left\{g(\e,\eta)\;;\;\eta \in \R\;;\;\e\in [0,1]\right\}$$ is a compact subset of $\X_{0}$  We notice now that
\begin{multline*}
\left\|\left[\K\Rs(\e+i\eta,\A)\right]^{j+1}f-\Ms_{i\eta}^{j+1}f\right\|_{\X_{k}}
\leq \left\|\left[\K\Rs(\e+i\eta,\A)\right]^{j}g(\e,\eta)-\Ms_{i\eta}^{j}g(\e,\eta)\right\|_{\X_{k}} \\
+\|\vartheta_{k}\|_{\infty}^{j}\left\| \K\Rs(\e+i\eta,\A)f-\Ms_{i\eta}f\right\|_{\mathscr{B}(\X_{0})}.
\end{multline*}
where we used \eqref{eq:MsSUP}. The compactness of the family $\{g(\e,\eta)\;;\;\eta \in \R\;;\;\e\in[0,1]\}\subset \X_0$ together with the induction hypothesis easily gives then
$$\lim_{\e\to0^+}\sup_{\eta \in \R}\left\|\left[\K\Rs(\e+i\eta,\A)\right]^{j}g(\e,\eta)-\Ms_{i\eta}^{j}g(\e,\eta)\right\|_{\X_{k}}=0.$$
This readily implies that
$$\lim_{\e\to0^+}\sup_{\eta \in \R} \left\|\left[\K\Rs(\e+i\eta,\A)\right]^{j+1}f-\Ms_{i\eta}^{j+1}f\right\|_{\X_{k}}=0$$
which achieves the inductive proof.
\end{proof}
\begin{nb} Notice that an easy consequence of \eqref{eq:stronMiek} together with Corollary \ref{cor:double} is that 
\begin{equation}\label{eq:RKRei}
\lim_{\e\to0}\sup_{\eta\in \R}\left\|\Rs(\e+i\eta,\A)\K\Rs(\e+i\eta,\A)f-\Rs(i\eta,\A)\K\Rs(i\eta,\A)\right\|_{\X_{k-1}}=0, \qquad \forall f \in \X_{0}\end{equation}
and any $k \leq N_{0}$.
\end{nb}
\begin{nb}\label{nb:doubleMs} As in Corollary \ref{cor:double}, we deduce in particular that, if $I \subset \R$ be a given compact interval and 
$g\::\:\l \in \C_{+} \longmapsto g(\l) \in \X_{0}$  is a continuous mapping such that the limit
$$\widetilde{g}(\eta):=\lim_{\e\to0^{+}}g(\e+i\eta)$$
exists in $\X_{0}$ uniformly with respect to $\eta \in I$. Then
\begin{equation}\label{eq:strongMiekUn}
\lim_{\e\to 0^{+}}\sup_{\eta \in I}\left\|\Ms_{\e+i\eta}g(\e+i\eta)-\Ms_{i\eta}\widetilde{g}(\eta)\right\|_{\X_{s}}=0
\end{equation}
for any $s \leq N_{0}$. As in Remark \ref{nb:I=R}, if $I=\R$,  the limit \eqref{eq:strongMiekUn} still holds true under the additional assumption that
$$\lim_{|\eta|\to\infty}\|g(\e+i\eta)\|_{\X_{s+1}}=0 \qquad \forall \e >0.$$
In such a case, one has $\Ms_{\e+i\eta}g(\e+i\eta)$ converges to $\Ms_{i\eta}\widetilde{g}(\eta)$ in $\mathscr{C}_{0}(\R,\X_{s})$.\end{nb}

\subsection{About the differentiability of $\Ms_{\l}$}
Let us now focus on the strong differentiability of the mapping $\l \mapsto \Ms_{\l}$ at $\l=0$. We begin with observing that the above computations all deal with limit of operators related to $\Rs(\e+i\eta,\A)$ for a  given $\eta$ (i.e. they concern limits along \emph{horizontal lines}) since this is enough for our analysis. Regarding the limiting behaviour around $\l=0$, we will need to allow a convergence in the usual sense in $\C$. Namely, we can reformulate our results as follows with the exact same proofs. We collect the results we will need later on in the following
\begin{theo}\label{theo:convL0}
For any $s \in \R$ and any $f \in \X_{s+1}$,
\begin{equation}\label{eq:Reta0Xs}
\underset{\l \in \overline{\C}_{+}}{\lim_{\l\to 0}}\left\|\Rs(\l,\A)f-\Rs(0,\A)f\right\|_{\X_{s}}=0.\end{equation}
Let $f \in \X_{0}$. Then
\begin{multline}\label{eq:stronMk0}
\underset{\l \in \overline{\C}_{+}}{\lim_{\l\to 0}}\left\|\K\Rs(\l,\A)f-\Ms_{0}f\right\|_{\X_{k}}=0, \\
\qquad \underset{\l \in \overline{\C}_{+}}{\lim_{\l\to 0}}\left\|\Rs(\l,\A)\K\Rs(\l,\A)f-\Rs(0,\A)\Ms_{0}f\right\|_{\X_{k-1}}=0, \qquad \forall k \leq N_{0},
\end{multline}
Moreover, for  any $j \in \N$ 
\begin{equation}\label{cor:uniformpower0}
\underset{\l \in \overline{\C}_{+}}{\lim_{\l\to 0}}\left\|\left[\K\Rs(\l,\A)\right]^{j}f-\Ms_{0}^{j}f\right\|_{\X_{k}}=0\end{equation}
uniformly with respect to $\eta \in \R$ for any $k \leq N_{0}.$
\end{theo}
\begin{proof} The proof is exactly the same as the ones derived in the previous subsection. For instance, for  $f \in \X_{1}$, one has
$$
\left\|\Rs(\l,\A)f-\Rs(0,\A)f\right\|_{\X_{0}} \leq \int_{0}^{\infty}\left|e^{-\l t}-1\right|\,\|U_{0}(t)f\|_{\X_{0}}\d t\, \qquad \forall \l \in \overline{\C}_{+}.$$
Then, since for any $t \geq 0$, $\lim_{\l \to 0}\left|e^{-\l t}-1\right|=0$, $\sup_{t \geq0}\left|e^{-\l t}-1\right| \leq 2$ for any  $\l \in \overline{\C}_{+}$  and 
since $t \mapsto \|U_{0}(t)f\|_{\X_{0}}$ is integrable over $[0,\infty)$, we deduce from  the dominated convergence theorem that
$$\lim_{\l \to 0}\left\|\Rs(\l,\A)f-\Rs(0,\A)f\right\|_{\X_{0}}=0.$$
We conclude then to \eqref{eq:Reta0Xs} as in the proof of \eqref{eq:RetaXs}. The other results are based upon the same arguments. Details are left to the reader.\end{proof}
For the first derivative of $\Ms_{\l}$ one has the following \footnote{In all the sequel, when considering $\lim_{\l \to 0}$, we will always assume that $\l$ approaches $0$ belonging to $\overline{\C}_{+}$, i.e.
$$\lim_{\l\to 0}\{\ldots\}=\underset{\l \in \overline{\C}_{+}}{\lim_{\l\to 0}}\{\ldots\}.$$
}

\begin{lemme}\label{lem:M'0f}
For any $f \in \X_{1}$, the limit
$${\lim_{\l\to 0}}\dfrac{\d}{\d \l}\Ms_{\l}f $$
exists in $\X_{0}$ and is denoted $\Ms_{0}'f$. Moreover,  one has
$$\left\|\Ms_{0}'f\right\|_{\X_{0}} \leq \|f\|_{\X_{1}}.$$
\end{lemme}
\begin{proof} Given $f \in \X_{1}$, let 
$$\Ms_{0}'f(x,v)=\int_{V}\bm{k}(v,w)\bm{m}(\d w)\int_{0}^{\infty}t\,\exp\left(-t\sigma(w)\right)f(x-tw,w)\d t, \qquad (x,v) \in \T^{d}\times V.$$
One has, with the usual change of variables $x \mapsto y=x-tw$ and since
\begin{multline*}
\|\Ms_{0}'f\|_{\X_{0}}\leq \int_{\T^{d}\times V}\left|f(y,w)\right|\d y\,\bm{m}(\d w)\int_{V}\bm{k}(v,w)\bm{m}(\d v)\int_{0}^{\infty}t\exp\left(-t\sigma(w)\right)\d t\\
=\int_{\T^{d}\times V}\left|f(y,w)\right|\d y\frac{\bm{m}(\d w)}{\sigma^{2}(w)}\int_{V}\bm{k}(v,w)\bm{m}(\d v)=\int_{\T^{d}\times V}\left|f(y,w)\right|\d y\frac{\bm{m}(\d w)}{\sigma(w)}
\end{multline*}
which proves that
$$\|\Ms_{0}'f\|_{\X_{0}} \leq \|f\|_{\X_{1}}.$$ Noticing now that
$$\dfrac{\d}{\d\l}\Ms_{\l}f(x,v)=\int_{V}\bm{k}(v,w)\bm{m}(\d w)\int_{0}^{\infty}t\,\exp\left(-\left(\l+\sigma(w)\right)t\right)f(x-tw,w)\d t$$
one gets that $\lim_{\l\to0}\left\|\frac{\d}{\d \l}\Ms_{\l}f-\Ms_{0}'f\right\|_{\X_{0}}=0$
thanks to the dominated convergence theorem.
\end{proof}
%\begin{nb} It can be shown easily that
%$$\K \int_{0}^{t}\tau U_{0}(\tau)f(x,v)\d \tau=\int_{V}\bm{k}(v,w)\bm{m}(\d w)\int_{0}^{t}\tau\,\exp\left(-\tau\sigma(w)\right)f(x-\tau\,w,w)\d \tau$$
%for any $(x,v)\in \T^{d}\times V$ and therefore, the limit
%$\lim_{t\to\infty}\K\int_{0}^{t}\tau\,U_{0}(\tau)f\d\tau$ also exists in $\X_{0}$ and coincides with $\Ms_{0}'f$. \end{nb}

We can extend this to higher power of $\Ms_{\l}$ under the sole assumption that $f \in \X_{1}$, namely
\begin{propo}\label{lem:diffHl}
For any $n \in \N$, set
$$\mathsf{L}_{n}(\l)=\Ms_{\l}^{n}, \qquad \forall \mathrm{Re}\l \geq 0.$$
Then, %for any $f \in \X_{1}$, 
%$$\underset{\l \in \C_{+}}{\lim_{\l\to0}}\dfrac{\d}{\d \l}\mathsf{L}_{n}(\l)f$$
%exists in $\X_{0}$. We denote this limit $\mathsf{L}_{n}'(0)f$ and there exists $C_{n}>0$ such that
%$$\left\|\mathsf{L}_{n}'(0)f\right\|_{\B(\X_{1},\X_{0})} \leq C_{n} <\infty.$$
%More generally, 
for any $p \in \{1,\ldots,N_{0}\}$, if $f \in \X_{p}$ there exists $\mathsf{L}_{n}^{(p)}(0)f \in \X_{s}$ for any $s \leq N_{0}$ such that
$$\underset{\l\in \C_{+}}{\lim_{\l \to 0}}\left\|\dfrac{\d^{p}}{\d\l^{p}}\mathsf{L}_{n}(\l)f-\mathsf{L}_{n}^{(p)}(0)f\right\|_{\X_{s}}=0$$
and 
\begin{equation}\label{eq:Lnps}
\sup_{\l \in \overline{\C}_{+}}\left\|\mathsf{L}_{n}^{(p)}(\l)f\right\|_{\X_{s}} \leq \left\|\mathsf{L}_{n}^{(p)}(0)f\right\|_{\X_{s}} \leq  \|\vartheta_{s}\|_{\infty}C_{n,p}\|f\|_{\X_{p}},\end{equation}
where
$$C_{n,p}=\sum _{|\bm{r}|=p}{p \choose \bm{r}}\prod_{j=1}^{n-1}\|\vartheta_{r_{j}}\|_{\infty}$$
with $\bm{r}=\left(r_{1},\ldots,r_{n}\right) \in \N^{n}$ is a multi-index such that $\ds |\bm{r}|=\sum_{j=1}^{n}r_{j}=p$ and $\ds {p \choose \bm{r}}=\frac{p!}{r_{1}!\ldots r_{n}!}.$
\end{propo}
The proof of this result is deferred to Appendix \ref{appen:Ml}.

\section{Spectral properties of $\Ms_{\l}$ along the imaginary axis}\label{sec:spec}

Recall that we defined
$$\Ms_{\l}:=\K\Rs(\l,\A), \qquad \qquad \mathrm{Re}\l \geq 0$$ so that $$\Ms_{\l}\varphi(x,v)=\int_{V}\bm{k}(v,w)\bm{m}(\d w)\int_{0}^{\infty}\exp\left(-\left(\l+\sigma(w)\right)t\right)\varphi(x-tw,w)\d t,$$
for any $\varphi \in \X_{0},$ and a.e $(x,v) \in \Omega\times V.$  
We study here more carefully the properties of $\Ms_{i\eta}$ for $\eta \in\R$.  We begin with the following which, as we shall see  allow to prove Theorem \ref{theo:main-invar} in the Introduction.
\begin{propo}\label{prop:varphi0} Assume that $\K$ is an irreducible operator satisfying Assumptions \ref{hypH} and  the measure $\bm{m}$ satisfies \eqref{eq:hypm}. There exists a  positive $\varphi_{0} \in \X_{N_{0}}$ such that
$$\Ms_{0}\,\varphi_{0}=\varphi_{0}, \qquad \int_{\T^{d}\times V}\varphi_{0}(x,v)\,\d x\,\bm{m}(\d v)=1.$$
\end{propo}
\begin{proof} Recall that $\Ms_{0}$ is stochastic, power-compact and irreducible. Consequently, its spectral radius $r_{\sigma}(\Ms_{0})=1$ is an algebraically simple and isolated eigenvalue of $\Ms_{0}$ and there is a normalised and positive eigenfunction $\varphi_{0} \in \X_{0}$ such that
$$\Ms_{0}\,\varphi_{0}=\varphi_{0}, \qquad \int_{\T^{d}\times V}\varphi_{0}\,\d x\,\bm{m}(\d v)=1.$$
In particular, one can define
\begin{equation}\label{eq:Psidef}
\Psi(x,v)=\int_{0}^{\infty}\,\exp\left(-s\sigma(v)\right)\varphi_{0}(x-sv,v)\d s, \qquad \forall (x,v) \in \O  \times V\end{equation}
and checks easily that 
\begin{multline*}
\|\Psi\|_{\X_{-1}}=\int_{V}\sigma(v)\bm{m}(\d v)\int_{\O }\d x\int_{0}^\infty \exp\left(-s\sigma(v)\right)\,\varphi_{0}(x-sv,v)\,\d t\\
%=\int_{\Omega}\d y\int_{V}\left|\varphi_{0}(y,v)\right|\bm{m}(\d v)\int_0^{\infty}\exp\left(-t\sigma(v)\right)\d t\\
=\int_{\O  \times V} \varphi_{0}(y,v)\,\d y\,\bm{m}(\d v)= \|\varphi_{0}\|_{\X_{0}}
\end{multline*}
i.e. $\Psi \in \X_{-1}$. Then, since $\K \in \mathscr{B}(\X_{-1},\X_{N_{0}})$, one has $\K\Psi \in \X_{N_{0}},$ i.e. $\Ms_{0}\varphi_{0} \in \X_{N_{0}}.$ Since 
$\varphi_{0}=\Ms_{0}\varphi_{0},$ we deduce that
$\varphi_{0} \in \X_{N_{0}}$ and this proves the Proposition.\end{proof}
The above Proposition as well as its proof allow to prove the existence and uniqueness of the invariant density
\begin{proof}[Proof of Theorem \ref{theo:main-invar}]
With the notations of the above proof, recall that we define $\Psi$ through \eqref{eq:Psidef} with $\Psi \in \X_{-1}.$ Since we prove that $\varphi_{0} \in \X_{N_{0}}$ one actually has that 
$$\Psi \in \X_{N_{0}-1}.$$
Moreover, a simple computation shows that
\begin{multline*}
U_0(t)\Psi(x,v)=\exp\left(-t\sigma(v)\right)\int_{0}^\infty \exp\left(-s\sigma(v)\right)\varphi_{0}(x-(t+s)v,v)\d s \\
=\int_{t}^{\infty}\exp\left(-s\sigma(v)\right)\varphi_{0}(x-sv,v)\d s \end{multline*}
so that, 
$$t^{-1}\left(U_{0}(t)\Psi(x,v)-\Psi(x,v)\right)=-\frac{1}{t}\int_0^{t}\exp\left(-s\sigma(v)\right)\varphi_{0}(x-sv,v)\d s$$
and one sees that $\Psi \in \D(\A)$ with 
$$\A \Psi=\lim_{t \to 0^+}t^{-1}\left(U_{0}(t)\Psi-\Psi\right)=-\varphi_0$$
where the limit is meant in $\X_0.$ Since $\K \Psi=\Ms_{0}\varphi_0$, the identity $\Ms_{0}\varphi_{0}=\varphi_{0}$ gives
$$\left(\A+\K\right){\Psi}=0$$
i.e. ${\Psi}$ is the invariant density of $\left(\mathcal{V}(t)\right)_{t\geq0}$. The fact that the invariant density with unit norm is unique comes from the irreducibility of $\left(\mathcal{V}(t)\right)_{t\geq0}$ . Notice that, then since the same result can also be proven in the spatially homogeneous case, i.e. considering only the semigroup generated by $\K$, one sees that there also exists a unique invariant density $\widetilde{\Psi}$ to $\K$, i.e. $\widetilde{\Psi}=\widetilde{\Psi}(v) >0$ such that $\K\widetilde{\Psi}=0$. In such case of course, $(\A+\K)\widetilde{\Psi}=0$ and, by uniqueness, $\Psi=\widetilde{\Psi}$ which proves that $\Psi$ turns out to be spatially homogeneous.\end{proof}

\begin{nb} Notice that, in the above proof, the definition we gave of $\Psi$ is somehow the expression we would deduce from  $\Rs(0,\A)\varphi_{0}$ if $0 \notin \mathfrak{S}(\A)$. Of course, what happens here is that, because $\varphi_{0} \in \X_{1}$, the expression still makes sense. This idea will be the cornerstone of our construction of the trace of $\Rs(\l,\A)$ along the imaginary axis in Proposition \ref{prop:convRsT0}.\end{nb}

 \begin{propo}\phantomsection \label{prop:Meps} 
For any $\lambda \in \C \setminus \{0\}$ with $\mathrm{Re}\l \geq 0$, 
$$r_{\sigma}(\Ms_{\l}) < 1.$$
\end{propo}
\begin{proof} We already saw in 
\eqref{eq:KK} that 
$$\|\Ms_{\l}\|_{\mathscr{B}(\X_{0})} \leq \frac{\|\sigma\|_{\infty}}{\mathrm{Re}\l+\|\sigma_{\infty}\|_{\infty}}, \qquad \forall \mathrm{Re}\l >0.$$
In particular, for $\mathrm{Re}\l >0$, $r_{\sigma}(\Ms_{\l}) < 1.$ Let us focus on the case  $\mathrm{Re}\l=0$. 
For $\l=i\eta$, one has
$$\Ms_{i\eta}\in \mathscr{B}(\X_{0}) \qquad \text{ with } \quad \left|\Ms_{i\eta} \right| \leq \Ms_{0}$$
where $\left|\Ms_{i\eta}\right|$ denotes the absolute value operator of $\Ms_{i\eta}$ (see \cite{chacon}). The operator $\Ms_{0}$ being power compact, the same holds for $\left|\Ms_{i\eta}\right|$ by a domination argument  so that
$$r_{\mathrm{ess}}(\left|\Ms_{i\eta}\right|)=0$$
where $r_{\mathrm{ess}}(\cdot)$ denotes the essential spectral radius. We prove that $r_{\sigma}(\left|\Ms_{i\eta} \right|) < 1$ by contradiction: assume, on the contrary, $r_{\sigma}(\left|\Ms_{i\eta}\right|)=1 > r_{\mathrm{ess}}(\left|\Ms_{i\eta} \right|)=0$, then $r_{\sigma}(\left|\Ms_{i\eta} \right|)$ is an isolated eigenvalue of $\left|\Ms_{i\eta} \right|$ with finite algebraic multiplicity and also an eigenvalue of the dual operator, associated to a nonnegative eigenfunction. From the fact that $\left|\Ms_{i\eta}\right| \leq \Ms_{0}$ with $\left|\Ms_{i\eta}\right| \neq \Ms_{0}$, one can invoke \cite[Theorem 4.3]{marek} to get that
$$r_{\sigma}(\left|\Ms_{i\eta}\right|) < r_{\sigma}(\Ms_{0})=1$$
which is a contradiction. Therefore, $r_{\sigma}(\left|\Ms_{i\eta}\right|) < 1$ and, since $r_{\sigma}(\Ms_{i\eta}) \leq r_{\sigma}(\left|\Ms_{i\eta}\right|)$, the conclusion holds true. 
\end{proof}
  
\begin{theo}\phantomsection \label{theo:spectTH} If $\K \in \mathscr{B}(\X_{0})$ satisfies Assumption \ref{hypH} then $i\R \subset \mathfrak{S}(\A+\K).$\end{theo}
\begin{proof} 
From Proposition \ref{prop:Meps} and Banach-Steinhaus Theorem \cite[Theorem 2.2, p. 32]{brezis}, for any $\eta \neq 0$,
\begin{equation}\label{eq:resM1}
\limsup_{\varepsilon \to 0^{+}}\left\|\Rs(1,\Ms_{\varepsilon+i\eta})\right\|_{\mathscr{B}(\X_{0})} < \infty.\end{equation}
Recall that, for $\mathrm{Re}\l >0$,
\begin{equation}\label{eq:lambdaTH}
\Rs(\l,\A+\K)=\Rs(\l,\A)+\sum_{n=1}^{\infty}\Rs(\l,\A)\left[\K\Rs(\l,\A)\right]^{n}=\Rs(\l,\A)+\Rs(\l,\A)\Ms_{\l}\Rs(1,\Ms_{\l})\,.\end{equation}
Now, combining \eqref{eq:MsSUP} and \eqref{eq:Ck} one has 
$$\sup_{\e\in [0,1],\eta \in\R}\left\|\Rs(\e+i\eta,\A)\Ms_{\e+i\eta}\right\|_{\mathscr{B}(\X_{0})} \leq C_{1}\|\vartheta_{1}\|_{\infty}< \infty$$
which, thanks to \eqref{eq:resM1}, yields
$$\limsup_{\e\to 0^{+}}\left\|\Rs(\e+i\eta,\A)\Ms_{\e+i\eta}\Rs(1,\Ms_{\e+i\eta})\right\|_{\mathscr{B}(\X_{0})} < \infty.$$
This, together with \eqref{eq:epsi-s-A} and \eqref{eq:lambdaTH} proves that, for any $\eta \in \R$, $\eta \neq 0$, it holds
$$\limsup_{\varepsilon \to 0^{+}}\left\|\Rs(\varepsilon+i\eta,\A+\K)\right\|_{\mathscr{B}(\X_{0})}=\infty,$$
whence $i\eta\in \mathfrak{S}(\A+\K)$ for any $\eta \neq 0$. Recalling that $0 \in \mathfrak{S}_{p}(\A+\K)$ we get the conclusion.\end{proof}

\subsection{{Spectral properties of $\Ms_{\l}$ in the vicinity of $\l=0$.}}\label{sec:Ml0}
We recall that, being $\Ms_{0}$ stochastic power-compact and irreducible, the spectral radius $r_{\sigma}(\Ms_{0})=1$ is an algebraically simple and isolated eigenvalue of $\Ms_{0}$ and there exists $0 < r < 1$ such that
$$\mathfrak{S}(\Ms_{0}) \setminus \{1\} \subset \{z \in \C\;;\;|z| < r\}$$
and there is a normalised and positive eigenfunction $\varphi_{0}$ such that
\begin{equation}\label{eq:Ms0varphi0}
\Ms_{0}\,\varphi_{0}=\varphi_{0}, \qquad \int_{\T^{d}\times V}\varphi_{0}\,\d x\,\bm{m}(\d v)=1.\end{equation}

Because $\Ms_{0}$ is stochastic, the dual operator $\Ms_{0}^{\star}$ (in $L^{\infty}(\T^{d}\times V,\d x\,\bm{m}(\d v))$) admits the eigenfunction 
$$\varphi_{0}^{\star}=\mathbf{1}_{\T^{d}\times V}$$ 
associated to the algebraically simple eigenvalue $1$ and the second part of \eqref{eq:Ms0varphi0} reads
$$\langle \varphi_{0},\varphi_{0}^{\star}\rangle=1$$
where $\langle\cdot,\cdot\rangle$ denotes the duality production between $\X_{0}$ and its dual $\X^{\star}_{0}$. 
Notice that then, for any $n \in \N$,
$$\mathfrak{S}(\Ms_{0}^{n}) \setminus \{1\} \subset \{z \in \C\;;\;|z|<r\}$$
with 
$$\Ms_{0}^{n}\varphi_{0}=\varphi_{0}, \qquad \forall n\in \N.$$
The spectral projection of $\Ms_{0}$ associated to the eigenvalue $1$ coincide then with that associated to $\Ms_{0}^{\mathsf{n}}$ (see Theorem \ref{theo:powerB}), i.e
$$\mathsf{P}(0)=\frac{1}{2i\pi}\oint_{\{|z-1|=r_{0}\}}\Rs(z,\Ms_{0})\d z=\frac{1}{2i\pi}\oint_{\{|z-1|=r_{0}\}}\Rs(z,\Ms_{0}^{n})\d z$$
where $r_{0} >0$ is chosen so that $\{z \in \C\;;\;|z-1|=r_{0}\} \subset \{z \in \C\;;\;|z| >r\}.$ 
 
We recall that Theorem \ref{theo:collective} established the fact that there  exists $\mathsf{q} \in \N$ such that 
$$\left\{\Ms_{\l}^{\mathsf{q}}\;;\;0 \leq \mathrm{Re}\l \leq 1\right\} \subset \B(\X_{0})$$
is collectively compact. From now on, we set
$$\H_{\l}=\Ms_{\l}^{\mathsf{q}}, \qquad 0 \leq \mathrm{Re}\l \leq 1.$$
A first consequence of this collective compactness  and  the strong convergence of $\Ms_{\e+i\eta}$ towards $\Ms_{i\eta}$ is the following
\begin{lemme}\label{cor:ResMei} For any $\eta_{0} \in \R\setminus\{0\}$, there is $0 < \delta < \tfrac{1}{2}|\eta_{0}|$ such that, for any $f \in \X_{0}$,
$$\lim_{\e\to0^{+}}\sup_{|\eta-\eta_{0}| < \delta}\bigg\|\Rs(1,\Ms_{\e+i\eta})f-\Rs(1,\Ms_{i\eta})f\bigg\|_{\X_{0}}=0.$$
%and, in the same way,
%$$\lim_{\e\to0^{+}}\sup_{|\eta-\eta_{0}| < \delta}\bigg\|\Rs(1,\Ms_{\e+i\eta}^{n})-\Rs(1,\Ms_{i\eta}^{n})\bigg\|_{\mathscr{B}(\X_{0})}=0, \qquad \forall n \geq \mathsf{q+1}.$$
\end{lemme} 
\begin{proof} Notice that the strong convergence 
$$\lim_{\e\to0}\Rs(1,\H_{\e+i\eta})f=\Rs(1,\H_{i\eta})f$$
can be proven for any $\eta \in \R \setminus \{0\}$ thanks to the collective compactness assumption and \cite[Theorem 5.3 (d)]{anselone}. The fact that  the convergence is uniform with respect to $\eta$ and that the uniform convergence transfers from $\Rs(1,\H_{\e+i\eta})$ to $\Rs(1,\Ms_{\e+i\eta})$ is deduced as follows.
For the uniform convergence, we closely follow the proof of  \cite[Theorem 5.3 (d)]{anselone}. First, recall that, for any $\eta \in \R \setminus \{0\}$, $r_{\sigma}(\H_{i\eta})< 1$ thanks to Proposition \ref{prop:Meps}. Then, according to \cite[Theorem 5.3]{anselone} that, for any $\bar{\eta} \in \R \setminus \{0\}$,
\begin{equation}\label{eq:limiRs}
\lim_{\eta\to \bar{\eta}}\left\|\Rs(1,\H_{i\eta})f-\Rs(1,\H_{i\bar{\eta}})f\right\|_{\X_{0}}=0 \qquad \forall f \in \X_{0}\,,\end{equation}
due to the collective compactness of $\{\H_{i\eta},\eta \in \R\}.$ 
Let us consider $\eta_{0} \in \R\setminus \{0\}$ and observe that, if $0 <\delta<\tfrac{|\eta_{0}|}{2}$ then, $\eta \neq 0$ whenever  $|\eta-\eta_{0}| < \delta$.  
According to Proposition \ref{prop:Meps}, there is $\varrho \in (0,1)$ such that 
\begin{equation}\label{eq:rsigvarrho}
r_{\sigma}(\Ms_{i\eta_{0}}) < \varrho <1.\end{equation}  
% Observe that 
%$$r_{\sigma}(\Ms_{\l})=r_{\sigma}(\H_{\l})=r_{\sigma}(\H_{\l}) \qquad \forall \l \in \C_{+}.$$
For $\l=\e+i\eta$, 
$$\mathbf{I}-\H_{\l}=\left[\mathbf{I}-\left(\H_{\l}-\H_{i\eta_{0}}\right)\Rs(1,\H_{i\eta_{0}})\right]\left(\mathbf{I}-\H_{i\eta_{0}}\right)$$
and, due to the collective compactness of $\left\{\left(\H_{\l}-\H_{i\eta_{0}}\right)\Rs(1,\H_{i\eta_{0}})\;;\;0 < \e < 1\right\}$ and the strong convergence of $\H_{\l}$ to $\H_{i\eta_{0}}$ as $\l \to i\eta_{0}$, we deduce from \cite[Lemma 5.2]{anselone} that
$$\lim_{\l\to i\eta_{0}}\left\|\bigg[\left(\H_{\l}-\H_{i\eta_{0}}\right)\Rs(1,\H_{i\eta_{0}})\bigg]^{2}\right\|_{\B(\X_{0})}=0$$
and there exist $\e_{0} >0,\delta >0$ such that
$$\left\|\left[\left(\H_{\l}-\H_{i\eta_{0}}\right)\Rs(1,\H_{i\eta_{0}})\right]^{2}\right\|_{\B(\X_{0})} \leq \frac{1}{2}, \qquad \forall \l=\e+i\eta, \quad 0 < \e< \e_{0}, \quad |\eta-\eta_{0}| < \delta.$$
In particular, for $0 < \e < \e_{0},$ and $|\eta-\eta_{0}| < \delta$, $\mathbf{I}-\left(\H_{\l}-\H_{i\eta_{0}}\right)\Rs(1,\H_{i\eta_{0}})$ is invertible and there exists $C >0$  such that
$$\underset{|\eta-\eta_{0}| < \delta}{\sup_{0\leq \e <\e_{0}}}\left\|\Rs\big(1,\left(\H_{\l}-\H_{i\eta_{0}}\right)\Rs(1,\H_{i\eta_{0}})\big)\right\|_{\B(\X_{0})}<\infty.$$
This gives that, for any $\l=\e+i\eta$ with $\e \in (0,\e_{0})$, $|\eta-\eta_{0}| < \delta$, $\mathbf{I}-\H_{\l}$ is invertible with
$$\Rs(1,\H_{\l})=\Rs(1,\H_{i\eta_{0}})\Rs\bigg(1,\left(\H_{\l}-\H_{i\eta_{0}}\right)\Rs(1,\H_{i\eta_{0}})\bigg)$$
with
$$\underset{|\eta-\eta_{0}| < \delta}{\sup_{0\leq \e <\e_{0}}}\left\|\Rs(1,\H_{\l})\right\|_{\B(\X_{0})} =M < \infty.$$
This shows that there exists $R \in (0,1)$ such that
$$r_{\sigma}(\H_{\l}) < R \qquad \forall \l=\e+i\eta, \qquad \e \in (0,\e_{0}), \quad |\eta-\eta_{0}| < \delta.$$
Since $r_{\sigma}(\Ms_{\l})=r_{\sigma}(\H_{\l})$, we deduce that
$$r_{\sigma}(\Ms_{\e+i\eta}) < R < 1 \qquad \forall \e \in (0,\e_{0}), \quad \eta \in (\eta_{0}-\delta,\eta_{0}+\delta).$$ 
Then, from the identity
$$\Rs(1,\H_{\l})-\Rs(1,\H_{i\eta})=\Rs(1,\H_{\l})\left(\H_{\l}-\H_{i\eta}\right)\Rs(1,\H_{i\eta})$$
we deduce that, for any $f \in \X_{0}$
\begin{equation}\label{eq:boundRsH}\begin{split}
\|\Rs(1,\H_{i\eta})f-\Rs(1,\H_{\l})f\|_{\X_{0}} &\leq \|\Rs(1,\H_{\l})\|_{\B(\X_{0})}\left\|\left(\H_{\l}-\H_{i\eta}\right)\Rs(1,\H_{i\eta})f\right\|_{\X_{0}}\\
&\leq M\left\|\left(\H_{\e+i\eta}-\H_{i\eta}\right)\Rs(1,\H_{i\eta})f\right\|_{\X_{0}}.
\end{split}\end{equation}
Now, according to Lemma \ref{lem:MeMi} (see Eq. \eqref{cor:uniformpower}), one has
\begin{equation}\label{eq:continu}
\lim_{\e\to0}\sup_{\eta\in \R}\|\H_{\e+i\eta}g-\H_{i\eta}g\|_{\X_{0}}=0 \qquad \forall g \in \X_{0}.\end{equation}
Let now $f \in \X_{0}$. According to \eqref{eq:limiRs}, the family
$$\{g(\eta)=\Rs(1,\H_{i\eta})f\;;\;\eta \in [\eta_{0}-\delta,\eta_{0}+\delta]\} \quad \text{ is a relatively compact subset of $\X_{0}$}$$
and we deduce then easily from \eqref{eq:continu} that
$$\lim_{\e\to0}\sup_{\eta \in [\eta_{0}-\delta,\eta_{0}+\delta]}\left\|\H_{\e+i\eta}g(\eta)-\H_{i\eta}g(\eta)\right\|_{\X_{0}}=0.$$
Thanks to \eqref{eq:boundRsH}, we deduce that
$$\lim_{\e\to0}\sup_{|\eta-\eta_{0}| \leq \delta}\left\|\Rs(1,\H_{\e+i\eta})f-\Rs(1,\H_{i\eta})f\right\|_{\X_{0}}=0, \qquad \forall f \in \X_{0}.$$
Recalling that
$$\Rs(1,\Ms_{\e+i\eta})f=\sum_{j=0}^{\mathsf{q}-1}\Rs(1,\H_{\e+i\eta})\Ms_{\e+i\eta}^{j}f \qquad \forall f \in \X_{0}$$
we deduce from the above convergence of $\Rs(1,\H_{\e+i\eta})$  and \eqref{cor:uniformpower} pointing out that all convergence holds uniformly with respect to $\eta \in [\eta_{0}-\delta,\eta_{0}+\delta]$.
\end{proof}

Thanks to Theorem \ref{theo:collective}, the spectral structure of $\Ms_{0}^{\mathsf{q}}$ is inherited by both $\H_{\l}$ and $\Ms_{\l}$ for $\l$ small enough. In the sequel, for any $z_{0} \in \C$, $r >0$, the disc of center $z_{0}$ and radius $r >0$ is denoted:
$$\mathds{D}(z_{0},r):=\{z \in \C\;;\;|z-z_{0}| < r\}.$$
One has 
\begin{theo}\label{prop:eigenMLH}
For any $\l \in \overline{\C}_{+}$  the spectrum of $\Ms_{\l}$ is given by
$$\mathfrak{S}(\Ms_{\l})=\{0\} \cup \{\mu_{j}(\l)\;;\;j \in \N_{\l} \subset \N\}$$
where, $\N_{\l}$ is a (possibly finite) subset of $\N$ and, for each $j \in \N_{\l}$, $\mu_{j}(\l)$ is an isolated eigenvalue of $\Ms_{\l}$ of finite algebraic multiplicities and $0$ being the only possible accumulation point of the sequence $\{\mu_{j}(\l)\}_{j\in \N_{\l}}$. Moreover,
$$|\mu_{j}(\l)| < 1 \qquad \text{ for any } j \in \N_{\l}, \qquad \l \neq 0.$$
Finally, there exist $\delta_{0}>0$ and $r_{0} \in (0,1)$ such that, for any $|\lambda| \leq \delta_{0}$, $\l \in \overline{\C}_{+}$,
\begin{equation}\label{eq:SML}
\mathfrak{S}(\Ms_{\l}) \cap \mathds{D}(1,r_{0})=\{\mu(\l)\} \end{equation}
where $\mu(\l)$ is an algebraically simple eigenvalue of $\Ms_{\l}$ such that 
$$\lim_{\l \to 0}\mu(\l)=1$$
and there exist an eigenfunction $\varphi_{\l}$ of $\Ms_{\l}$ and an  eigenfunction $\varphi^{\star}_{\l}$ of $\left(\Ms_{\l}\right)^{\star}$ associated to $\mu(\l)$ such that
$$\langle \varphi_{\l},\varphi_{\l}^{\star}\rangle=1, \qquad \lim_{\l\to0}\|\varphi_{\l}-\varphi_{0}\|_{\X_{0}}=0$$
and the spectral projection $\mathsf{P}(\l)$ associated to $\mu(\l)$  $(\l \in \overline{\C}_{+})$ is given by
\begin{equation}\label{eq:varphi*} 
\mathsf{P}(\l)\psi=\langle \psi,\varphi_{\l}^{\star}\rangle\,\varphi_{\l}, \qquad \psi \in \X_{0}\end{equation}
and is such that
\begin{equation}\label{eq:convPL}
\lim_{\l \to 0}\left\|\mathsf{P}(\l)\psi-\mathsf{P}(0)\psi\right\|_{\X_{0}}=0, \qquad \forall \psi \in \X_{0}.\end{equation}
\end{theo}
\begin{proof} Since  $\H_{\l}$ is compact, the structure of $\mathfrak{S}(\Ms_{\l})$ follows. The fact that all eigenvalues have modulus less than one comes from Proposition \ref{prop:Meps}. This gives the first part of the Proposition. For the second part,   the \emph{collective compactness} of $\{\H_{\l}\;,\;0\leq \mathrm{Re}\l \leq 1\}$ is used in a crucial way. Indeed, recall that
$$\lim_{\l \to 0}\|\H_{\l}\varphi-\H_{0}\varphi\|_{\X}=0, \qquad \forall \varphi \in \X.$$
This, together with the collective compactness of the family $\{\H_{\l}\,;\,0\leq \mathrm{Re}\l\leq 1\}$ give the separation of the spectrum thanks to \cite[Theorem 5.3 and Proposition 6.3]{anselone}. Namely,  there exist $\delta_{0},r_{0}$ small enough such that, for $|\l| < \delta_{0}$ small enough, the curve $\{z \in \C\;;\;|z-1|=r_{0}\}$ is separating the spectrum $\mathfrak{S}(\H_{\l})$ into two disjoint parts, say
$$\mathfrak{S}(\H_{\l})=\mathfrak{S}_{\text{in}}(\H_{\l}) \cup \mathfrak{S}_{\text{ext}}(\H_{\l})$$
where $\mathfrak{S}_{\text{in}}(\H_{\l}) \subset \{z\in \C\;;|z-1| < r_{0}\}=\mathds{D}(1,r_{0})$ and $\mathfrak{S}_{\text{ext}}(\H_{\l}) \subset \{z\in \C\;;|z-1|>r_{0}\}.$ Moreover, the spectral projection of $\H_{\l}$ associated to $\mathfrak{S}_{\text{in}}(\H_{\l})$, defined as,
\begin{equation}\label{eq:Pl}
\mathsf{P}_{\mathsf{q}}(\l)=\frac{1}{2i\pi}\oint_{\{|z-1|=r_{0}\}}\Rs(z,\H_{\l})\d z,\end{equation}
is such that 
$$\lim_{\l \to 0}\left\|\mathsf{P}_{\mathsf{q}}(\l)\varphi-\mathsf{P}_{\mathsf{q}}(0)\varphi\right\|_{\X}=0 \qquad \forall \varphi \in\X$$
and
$$\mathrm{dim}(\mathrm{Range}(\mathsf{P}_{\mathsf{q}}(\l)))=\mathrm{dim}(\mathrm{Range}(\mathsf{P}_{\mathsf{q}}(0))), \qquad |\l| < \delta_{0}, \mathrm{Re}\l \geq 0.$$
Now, from the same consideration as those in the beginning of Section \ref{sec:Ml0}, 
$$\mathrm{dim}(\mathrm{Range}(\mathsf{P}_{\mathsf{q}}(0)))=1.$$
Therefore, $\mathrm{dim}(\mathrm{Range}(\mathsf{P}_{\mathsf{q}}(\l)))=1$ for $|\l| < \delta_{0},$ $\mathrm{Re}\l\geq0,$ i.e.
$$\mathfrak{S}_{\text{in}}(\H_{\l})=\mathfrak{S}(\H_{\l}) \cap \mathds{D}(1,r_{0})=\{\nu(\l)\}$$
reduces to some  \emph{algebraically simple} eigenvalue of $\H_{\l}$. From Theorem \ref{theo:powerB} (and Remark \ref{nb:powerB}) in Appendix \ref{sec:functional}, one deduces that there exists a unique eigenvalue $\mu(\l)$ of $\Ms_{\l}$ such that 
$$\nu(\l)=\mu(\l)^{\mathsf{q}}\,,$$ that eigenvalue $\mu(\l)$ is simple and%, for $r=\sqrt[\mathsf{q}]{r_{0}}$,   
$$\mathfrak{S}_{\text{in}}(\Ms_{\l})=\mathfrak{S}(\Ms_{\l}) \cap \mathds{D}(1,r_{0})=\{\mu (\l)\}, \qquad |\l| < \delta_{0}, \mathrm{Re}\l \geq0.$$
Notice that, clearly
$$\lim_{\l\to 0}\mu(\l)=1 \qquad (\mathrm{Re}\l \geq0)$$
while, defining
$$\mathsf{P}(\l)=\frac{1}{2i\pi}\oint_{\{|z-1|=r_{0}\}}\Rs\left(z, \Ms_{\l} \right) \d z, \qquad |\l| \leq \delta_{0}$$
one deduces again from Theorem \ref{theo:powerB} that
$$\mathsf{P}(\l)=\mathsf{P}_{\mathsf{q}}(\l)\,, \qquad |\l| \leq \delta_{0}.$$
In particular, \eqref{eq:convPL} holds true. 
Setting then
$$\varphi_{\l}:=\mathsf{P}(\l)\varphi_{0}, \qquad \l \in \C_{+}$$
where $\varphi_{0}$ is the eigenfunction of $\Ms_{0}$ in \eqref{eq:Ms0varphi0}, one sees that $\varphi_{\l}$ converges to $\mathsf{P}(0)\varphi_{0}=\varphi_{0} \neq 0$ as $\l \to 0$. Thus,  $\varphi_{\l} \neq 0$ for $\l$ small enough and, since $\mu(\l)$ is algebraically simple, $\varphi_{\l}$ is an eigenfunction of $\Ms_{\l}$ for $|\l|$ small enough. Setting 
$$\beta_{\l}:=\langle\varphi_{\l},\varphi_{0}^{\star}\rangle \neq 0$$
one sees that $\lim_{\l\to0}\beta_{\l}=\beta_{0}=\langle\varphi_{0},\varphi^{\star}_{0}\rangle=1$
where we recall $\varphi_{0}^{\star}=\ind_{\T^{d}\times V}$. In particular, for $|\l|$ small enough, $\beta_{\l}\neq0$ and, setting 
$$\varphi^{\star}_{\l}=\frac{1}{\beta_{\l}}\mathsf{P}^{\star}(\l)\varphi_{0}^{\star}$$
one checks easily that 
$$\langle \varphi_{\l},\varphi^{\star}_{\l}\rangle=\beta_{\l}^{-1}\langle \mathsf{P}(\l)\varphi_{\l},\varphi_{0}^{\star}\rangle=\beta_{\l}^{-1}\langle \varphi_{\l},\varphi_{0}^{\star}\rangle=1$$
and, since $\mathsf{P}(\l)$ is of rank one, the expression \eqref{eq:varphi*} follows.
\end{proof}
\begin{nb} Notice that, in full generality, it is not true that $\Ms_{\l}^{\ast}$ converges strongly to $\Ms_{0}^{\ast}$ as $\l \to 0$ in $\X_{0}^{\ast}.$ In particular, one cannot conclude that $\varphi_{\l}^{\star}$ converges to $\varphi_{0}^{\star}$ in $\X_{0}^{\ast}.$ However, from the strong convergence of $\mathsf{P}(\l)$ to $\mathsf{P}(0)$ and the convergence of $\varphi_{\l}$ to $\varphi_{0}$, one deduces that
\begin{equation*}\begin{split}
\lim_{\l\to0}\langle \psi,\varphi_{\l}^{\star}\rangle &=\lim_{\l\to0}\beta_{\l}^{-1}\langle \psi,\mathsf{P}^{\star}(\l)\varphi^{\star}_{0}\rangle\\
&=\lim_{\l\to0}\beta_{\l}^{-1}\langle \mathsf{P}(\l)\psi,\varphi^{\star}_{0}\rangle=\beta_{0}^{-1}\langle \mathsf{P}(0)\psi,\varphi_{0}^{\ast}\rangle=\langle \psi,\mathsf{P}(0)^{\star}\varphi^{\star}_{0}\rangle\end{split}\end{equation*}
so that
\begin{equation}\label{eq:conast}
\lim_{\l\to0}\langle \psi,\varphi_{\l}^{\star}\rangle=\int_{\O\times V}\psi(x,v)\d x\otimes\bm{m}(\d v) \qquad \forall \psi \in \X_{0}
\end{equation}
i.e. $\varphi_{\l}^{\star}$ converges to $\varphi_{0}^{\star}=\ind_{\T^{d}\times V}$ in the weak-$\ast$ topology of $\X_{0}^{\star}$ as $\l \to 0.$
\end{nb}
We can upgrade the convergence in \eqref{eq:convPL} as follows
\begin{cor}\label{cor:PlP0FXs} Let $f \in \X_{s}$, $s \leq N_{0}$. Then
$$\lim_{\l \to 0}\|\mathsf{P}(\l)f-\mathsf{P}(0)f\|_{\X_{s}}=0 .$$
\end{cor}
\begin{proof} The proof is based upon the following observation: if $f \in \X_{s}$ and $g(\l)=\Rs(z,\Ms_{\l})f \in \X_{0}$ with $|z-1|=r_{0}$, then
$$zg(\l)=f + \Ms_{\l}g(\l).$$
Since $\lim_{\l\to0}\|g(\l) -g(0)\|_{\X_{0}}=0$ one has, from the regularising effect of $\Ms_{\l}$ that
$$\lim_{\l\to0}\|\Ms_{\l}g(\l)-\Ms_{\l}g(0)\|_{\X_{s}}=0$$
as soon as $s \leq N_{0}.$ This proves (because $f \in \X_{s})$ that 
$$g(\l)=z^{-1}f+z^{-1}\Ms_{\l}g(\l) \in \X_{s}$$
and 
$$\lim_{\l\to0}\left\|g(\l)-g(0)\right\|_{\X_{s}}=0$$
where the convergence is actually uniform with respect to $z \in \Gamma=\{z \in \C\;;\;|z-1|=r_{0}\}$ since $\sup_{z \in \Gamma}|z|^{-1} < \infty.$ Equivalently
\begin{equation}\label{eq:RsMlz}
\lim_{\l\to0}\sup_{z \in \Gamma}\left\|\Rs(z,\Ms_{\l})f-\Rs(z,\Ms_{0})f\right\|_{\X_{s}}=0.\end{equation}
Since 
$$\mathsf{P}(\l)=\frac{1}{2i\pi}\oint_{\Gamma}\Rs(z,\Ms_{\l})f\d z, \qquad \l \in \overline{\C}_{+}, \quad |\l| \leq \delta_{0}$$
this gives the conclusion.\end{proof}
From now, we define $\delta >0$ small enough, so that the rectangle
$$\mathscr{C}_{\delta}:=\{\l \in \C\;;\;0 \leq \mathrm{Re}\l \leq \delta\,,\,|\mathrm{Im}\l|\leq\delta\} \subset \{\l \in \C\;;\;|\l| < \delta_{0}\}\,,$$
where $\delta_{0}$ is introduced in the previous Theorem \ref{prop:eigenMLH}. In the sequel, the notion of differentiability of functions $h\::\:\l \in \mathscr{C}_{\delta} \mapsto h(\l) \in Y$ (where $Y$ is a given Banach space) is the usual one but, if $\mathrm{Re}\l=0$, we have to emphasize the fact that  limits are always meant in $\overline{\C}_{+}$ 
\footnote{This means for instance that, if $\l_{0} \in \mathscr{C}_{\delta}$ with $\mathrm{Re}\l_{0} >0,$ $h$ is differentiable means that it is holomorphic in a neighborhood of $\l_{0}$ whereas, for $\l_{0}=i\eta_{0}$, $\eta_{0} \in \R$, the differentiability at $\l_{0}$ of $h$ at means that there exists $h'(\l_{0}) \in Y$ such that 
$$\underset{\l \in \overline{\C}_{+}}{\lim_{\l \to \l_{0}}}\left\|\frac{h(\l)-h(\l_{0})}{\l-\l_{0}}-h'(\l_{0})\right\|_{Y}=0$$
where $\|\cdot\|_{Y}$ is the norm on $Y$.}

\begin{lemme}\label{lem:p'0}
For any $\l \in \mathscr{C}_{\delta}$, we denote with $\mathsf{P}(\l)$ the spectral projection associated to the simple eigenvalue $\mu(\l)$ of $\Ms_{\l}$ as defined in Theorem \ref{prop:eigenMLH}. For any $f \in \X_{1}$, the mapping 
$$\l \in \mathscr{C}_{\delta} \longmapsto \mathsf{P}(\l)f \in \X_{0}$$
is differentiable and
$$\mathsf{P}'(0)f=-\frac{1}{2i\pi}\oint_{\{|z-1|=r_{0}\}}\Rs(z,\H_{0})\H_{0}'\Rs(z,\H_{0})f\d z$$
where $\H_{0}'=\mathsf{L}_{\mathsf{q}}^{(1)}(0)$ is the derivative of $\H_{\l}$ at $\l=0$. 
%More generally, for any $\eta \in (-\delta,\delta)$ and any $f \in \X_{1}$,
%$$\frac{\d}{\d\eta}\mathsf{P}(i\eta)f=-\frac{1}{2i\pi}\oint_{\{|z-1|=r_{0}\}}\Rs(z,\H_{i\eta})\left(\dfrac{\d}{\d\eta}\H_{i\eta}\right)\Rs(z,\H_{i\eta})f\d z.$$
\end{lemme}
\begin{proof}  Recall that, being $\mu(\l)$ a simple eigenvalue of $\Ms_{\l}$, we deduce from Theorem \ref{theo:powerB}  that
$$\mathsf{P}(\l)=\mathsf{P}_{\mathsf{q}}(\l)$$
where $\mathsf{P}_{\mathsf{q}}(\l)$ is the spectral projection associated to $\H_{\l}=\Ms_{\l}^{\mathsf{q}}$ and its simple eigenvalue $\mu^{\mathsf{q}}(\l)$.  One has then, for any $f \in \X_{1}$
$$\mathsf{P}(\l)f=\frac{1}{2i\pi}\oint_{\{|z-1|=r_{0}\}}\Rs(z,\H_{\l})f\d z.$$  As soon as $z \notin \mathfrak{S}\left(\Ms_{\l}\right)$ for any $\l \in \mathscr{C}_{\delta}$, one has
$$\dfrac{\d}{\d\l}\Rs(z,\H_{\l})f=-\Rs(z,\H_{\l})\left(\frac{\d}{\d\l}\H_{\l}\right)\Rs(z,\H_{\l})f,$$
so that
$$\dfrac{\d}{\d\l}\mathsf{P}(\l)f=-\frac{1}{2i\pi}\oint_{\{|z-1|=r_{0}\}}\Rs(z,\H_{\l})\left(\frac{\d}{\d\l}\H_{\l}\right)\Rs(z,\H_{\l})f\d z\qquad \forall \l \in \mathscr{C}_{\delta}.$$
Now, since $\left\{\H_{\l}\;;\;0 \leq \mathrm{Re}\l \leq 1\right\}$ is collectively compact and $\H_{\l}$ converges strongly to $\H_{0}$ on $\X_{0}$ as $\l \to 0$, we deduce from 
\cite[Theorem 5.3]{anselone} that
\begin{equation}\label{eq:unfirRs}
\lim_{\l\to0}\sup_{|z-1|=r_{0}}\left\|\Rs(z,\H_{\l})g-\Rs(z,\H_{0})g\right\|_{\X_{0}}=0 \qquad \forall g \in \X_{0}.
\end{equation}
Now, recall from Lemma \ref{lem:diffHl} that
$$\lim_{\l \to 0}\left\|\frac{\d}{\d\l} \H_{\l}\psi-\H_{0}'\psi\right\|_{\X_{0}}=0$$
for any $\psi \in \X_{1}$ where $\H_{0}'=\mathsf{L}_{\mathsf{q}}^{(1)}(0)$. Introducing 
$$\psi(\l)=\Rs(z,\H_{\l})f$$ 
which depends on $z \in \mathds{D}(1,r_{0})$, one can argue as in the proof of Corollary \ref{cor:PlP0FXs} to deduce that
$$\lim_{\l\to0}\|\psi(\l)-\psi(0)\|_{\X_{s}}=0 \qquad s \leq N_{0}$$
where the convergence is uniform with respect to $z \in \mathbb{D}(1,r_{0})$. One deduces then that
$$\lim_{\l\to0}\sup_{|z-1|=r_{0}}\left\|\frac{\d}{\d\l}\H_{\l}\Rs(z,\H_{\l})f-\H_{0}'\Rs(z,\H_{0})f\right\|_{\X_{0}}=0$$
and then, thanks to \eqref{eq:unfirRs} again, we deduce that
$$\lim_{\l\to0}\sup_{|z-1|=r_{0}}\left\|\Rs(z,\H_{\l})\frac{\d}{\d\l}\H_{\l}\Rs(z,\H_{\l})f-\Rs(z,\H_{0})\H_{0}'\Rs(z,\H_{0})f\right\|_{\X_{0}}=0$$
which proves the differentiability in $0$ of $\mathsf{P}(\l)f$. The same computations also give
$$\dfrac{\d}{\d\eta}\mathsf{P}(\e+i\eta)=-\frac{1}{2i\pi}\oint_{\{|z-1|=r_{0}\}}\Rs(z,\H_{\e+i\eta})\left(\frac{\d}{\d\eta}\H_{\e+i\eta}\right)\Rs(z,\H_{\e+i\eta})\d z, \qquad \forall \eta \in \R\setminus\{0\}.$$
Using now Lemma \ref{prop:derMeis}  which asserts that $\tfrac{\d}{\d\eta}\Ms_{\e+i\eta}$ converges to $\tfrac{\d}{\d\eta}\Ms_{i\eta}$ as $\e\to0^{+}$ uniformly with respect to $\eta$, we can prove as in Lemma \ref{lem:diffHl} that the same holds for $\tfrac{\d}{\d\eta}\H_{\e+i\eta}$ and one deduces the second part of the Lemma. 
\end{proof}
We deduce from this the differentiability of the simple eigenvalue $\mu(\l)$ 
\begin{lemme}\label{lem:deriv} With the notations of Theorem \ref{prop:eigenMLH}, the function $\l \in \mathscr{C}_{\delta} \mapsto \mu(\l) \in \C$ is differentiable with derivative $\mu'(\l)$ such that the limit
$$\lim_{\l\to0}\mu'(\l)=\mu'(0)$$
exists with $\mu'(0)<0.$
\end{lemme}
\begin{proof} Recall that we introduced in the proof of Theorem \ref{prop:eigenMLH} the function $\varphi_{\l}=\mathsf{P}(\l)\varphi_{0}$ as well as the unique eigenfunction $\varphi_{\l}^{\star} \in \X_{0}^{\star}$ of $\Ms_{\l}^{\ast}$ associated to $\mu(\l)$ and such that $\langle\varphi_{\l},\varphi_{\l}^{\star}\rangle=1$ where $\langle\cdot,\cdot\rangle$ is the duality bracket between $\X_{0}$ and its dual $\X_{0}^{\star}.$ Since $\varphi_{0} \in \X_{1}$, the mapping  $\l \in \mathscr{C}_{\delta}\mapsto \varphi_{\l} \in \X_{0}$ is differentiable with 
$$\dfrac{\d}{\d\l}\varphi_{\l}=\dfrac{\d}{\d\l}\mathsf{P}(\l)\varphi_{0}, \qquad \l \in \overline{\C}_{+}$$
according to Lemma \ref{lem:p'0}. Since 
$$\Ms_{\l}\varphi_{\l}=\mu(\l)\varphi_{\l},$$
and the mapping $\l \in \overline{\C}_{+} \mapsto \Ms_{\l} \in \B(\X_{0})$ is analytic while $\mu(\l)$ is a \emph{simple} eigenvalue, we deduce from \cite[Chapter II-1]{kato} that,  for any $\l \in \overline{\C}_{+}$, the derivative $\mu'(\l)$ exists with 
$$\left(\dfrac{\d}{\d\l}\Ms_{\l}\right)\varphi_{\l}+\Ms_{\l}\dfrac{\d}{\d\l}\varphi_{\l}=\mu'(\l)\varphi_{\l}+\mu(\l)\frac{\d}{\d\l}\varphi_{\l}.$$
Taking now the duality bracket of this identity with $\varphi_{\l}^{\star}$ and noticing that
$$\left\langle \Ms_{\l}\dfrac{\d}{\d\l}\varphi_{\l},\varphi_{\l}^{\star}\right\rangle=\left\langle \frac{\d}{\d\l}\varphi_{\l},\Ms^{\star}_{\l}\varphi_{\l}^{\star}\right\rangle=\mu(\l)\left\langle \dfrac{\d}{\d\l}\varphi_{\l},\varphi_{\l}^{\star}\right\rangle$$
we get that
$$\left\langle \left(\dfrac{\d}{\d\l}\Ms_{\l}\right)\varphi_{\l},\varphi^{\star}_{\l}\right\rangle=\mu'(\l)\langle \varphi_{\l},\varphi_{\l}^{\star}\rangle=\mu'(\l).$$
One sees that
\begin{multline*}
\left|\left\langle \left(\dfrac{\d}{\d\l}\Ms_{\l}\right)\varphi_{\l},\varphi^{\star}_{\l}\right\rangle-\left\langle \Ms_{0}'\varphi_{0},\varphi^{\star}_{0}\right\rangle\right|\\
\leq \left|\left\langle \left(\dfrac{\d}{\d\l}\Ms_{\l}\right)\left(\varphi_{\l}-\varphi_{0}\right),\varphi^{\star}_{\l}\right\rangle\right|
+\left|\left\langle \left(\dfrac{\d}{\d\l}\Ms_{\l}\right)\varphi_{0}-\Ms_{0}'\varphi_{0},\varphi^{\star}_{\l}\right\rangle\right|\\
+\left|\left\langle \Ms_{0}'\varphi_{0},\varphi^{\star}_{\l}-\varphi^{\star}_{0}\right\rangle\right|
\end{multline*}
where the first and second terms on the right-hand side converges to zero as $\l \to 0$ since 
$$\lim_{\l\to0}\|\varphi_{\l}-\varphi_{0}\|_{\X_{1}}=0 \qquad \text{ and} \quad \lim_{\l\to0}\left\|\left(\dfrac{\d}{\d\l}\Ms_{\l}\right)\varphi_{0}-\Ms_{0}'\varphi_{0}\right\|_{\X_{0}}=0$$ (recall that $\varphi_{0}\in \X_{1}$) while the third term goes to zero as $\l \to0$ since $\varphi_{\l}^{\star}$ converges to $\varphi_{0}^{\star}$ in the weak-$\star$ topology of $\X_{0}^{\star}.$ We deduce from this that 
$$\mu'(0)=\lim_{\l\to0}\mu'(\l)=\lim_{\l\to0}\left\langle \left(\dfrac{\d}{\d\l}\Ms_{\l}\right)\varphi_{\l},\varphi^{\star}_{\l}\right\rangle=\langle \Ms_{0}'\varphi_{0},\varphi_{0}^{\ast}\rangle.$$
Therefore, using Lemma \ref{lem:diffHl} (see in particular \eqref{eq:Ms0n}) and since $\varphi_{0}^{\ast}=1$, we deduce that
\begin{equation*}\begin{split}
\mu'(0)&=\int_{\Omega \times V}\Ms_{0}'\varphi_{0}(x,v)\d x\,\bm{m}(\d v)\\
&=-\int_{\Omega \times V}\d x\,\bm{m}(\d v)\int_{V}\bm{k}(v,w)\bm{m}(\d w)
\int_{0}^{\infty}\,t\exp\left(-\sigma(w)t\right)\,\varphi_{0}\left(x-tw,w\right)\d t\\
&=-\int_{\O\times V}\varphi_{0}(y,w)\d y\,\bm{m}(\d w)\int_{V}\bm{k}(v,w)\bm{m}(\d v)\int_{0}^{\infty}t\exp\left(-t\sigma(w)\right)\d t\\
&=-\int_{\O\times V}\varphi_{0}(y,w)\d y\,\bm{m}(\d w)\int_{V}\bm{k}(v,w)\sigma^{-2}(w)\bm{m}(\d v)
\end{split}\end{equation*}
i.e.
$$
\mu'(0)=-\int_{\Omega \times V}\varphi_{0}(y,w)\sigma^{-1}(w)\d y\,\bm{m}(\d w)$$
 and the negativity of $\mu'(0)$ follows since $\varphi_{0} >0$. 
\end{proof}
We deduce from this the following fundamental result for our construction of the strong limit of $\Rs(\l,\A+\K)$ at $0$,
\begin{propo}\label{propo:P0conv}
Let $f \in \X_{1}$ be such that $\mathsf{P}(0)f=0$. Then,
$${\lim_{\l\to0}}\,\Rs(1,\Ms_{\l})\mathsf{P}(\l)f,$$
exists in $\X_{0}$ and is denoted by $\mathsf{\Phi}_{0}f$. Moreover, 
$$\mathsf{\Phi}_{0}f=-\frac{1}{\mu'(0)}\mathsf{P}'(0)f.$$
\end{propo}
\begin{proof} Recall that, since $f \in \X_{1}$, the derivative $\mathsf{P}'(0)f$ is well-defined whereas, as observed, $\mu'(0) \neq0.$ We simply observes then that, given $\l \in \overline{\C}_{+}\setminus\{0\}$, 
$$\left(\mathbf{I}-\Ms_{\l}\right)\mathsf{P}(\l)\varphi=\left(1-\mu(\l)\right)\mathsf{P}(\l)\varphi \qquad \forall \varphi \in\X_{0}$$
which implies obviously that
$$\Rs(1,\Ms_{\l})\mathsf{P}(\l)\varphi=\frac{1}{1-\mu(\l)}\mathsf{P}(\l)\varphi, \qquad \forall \varphi \in \X_{0}$$
provided $\l\neq0$ so that $\mu(\l) \neq 1$ (and in particular $1 \notin \mathfrak{S}(\Ms_{\l})$). Then, for $f \in \X_{1}$ with $\mathsf{P}(0)f=0$ one can write, for $|\l|$ small enough,
$$\Rs(1,\Ms_{\l})\mathsf{P}(\l)f=\frac{1}{1-\mu(\l)}\left(\mathsf{P}(\l)f-\mathsf{P}(0)f\right)=\frac{\l}{1-\mu(\l)}\left(\frac{\mathsf{P}(\l)f-\mathsf{P}(0)}{\l}\right)\,.$$
Now, because $f \in \X_{1}$, 
$$\lim_{\l \to 0}\frac{\mathsf{P}(\l)f-\mathsf{P}(0)}{\l}=\mathsf{P}'(0)f$$
where the convergence holds in $\X_{0}$ while, since $\mu'(0) \neq 0$, 
$\lim_{\l\to0}\frac{\l}{1-\mu(\l)}=-\frac{1}{\mu'(0)}$
which gives the result.\end{proof}

 \section{The extension of $\Rs(\l,\A+\K)$ to the imaginary axis}\label{sec:trace}

 \subsection{Existence of the boundary function for $\Rs(\l,\A+\K)$}
We begin with the extension of $\Rs(\l,\A+\K)$ far away from zero:
\begin{lemme}\label{propo:convRsAKeta}
Let $\eta_{0} \neq 0$ be given. Then, there is $\delta >0$ (small enough) such that, for any $f \in\X_{1}$, 
$$\lim_{\e\to0}\sup_{|\eta-\eta_{0}| \leq \delta}\bigg\|\Rs(\e+i\eta,\A+\K)f-\bigg[\Rs(i\eta,\A)f+\Rs(i\eta,\A)\Ms_{i\eta}\Rs(1,\Ms_{i\eta})f\bigg]\bigg\|_{\X_{0}}=0.$$
\end{lemme}
\begin{proof} We recall that, for any $f \in \X_{0}$, Lemma \ref{cor:ResMei}  established that
$$\lim_{\e\to0^{+}}\sup_{|\eta-\eta_{0}| < \delta}\bigg\|\Rs(1,\Ms_{\e+i\eta})g-\Rs(1,\Ms_{i\eta})g\bigg\|_{\X_{0}}=0$$
holds for any $g \in \X_{0}$. Recalling that \eqref{eq:strongMiekUn}
 holds in particular for $s=1$, i.e.  
$$\lim_{\e\to0^{+}}\sup_{|\eta-\eta_{0}| < \delta}\bigg\|\Ms_{\e+i\eta}\Rs(1,\Ms_{\e+i\eta})g-\Ms_{i\eta}\Rs(1,\Ms_{i\eta})g\bigg\|_{\X_{1}}=0, \qquad \forall g \in \X_{0}$$
we deduce from Corollary \ref{cor:double} that
$$\lim_{\e\to0^{+}}\sup_{|\eta-\eta_{0}| < \delta}\bigg\|\Rs(\e+i\eta,\A)\Ms_{\e+i\eta}\Rs(1,\Ms_{\e+i\eta})f-\Rs(i\eta,\A)\Ms_{i\eta}\Rs(1,\Ms_{i\eta})f\bigg\|_{\X_{0}}=0 \qquad \forall f \in \X_{0}.$$
Combining this with \eqref{eq:RetaXs}
 and the fact that 
$$\Rs(\e+i\eta,\A+\K)f=\Rs(\e+i\eta,\A)f + \Rs(\e+i\eta,\A)\Ms_{\e+i\eta}\Rs(1,\Ms_{\e+i\eta})f$$
we deduce the result.
\end{proof}
This shows that, away from zero, it is possible to extend the resolvent of $\Rs(\l,\A+\K)f$ along the imaginary axis, i.e. for $\l=i\eta$. The extension for $\l=0$ is much more involved and follows the approach of \cite[Section 4]{MKL-JFA}. For such a case, recalling that, with the definition of $\mathsf{P}_{0}$ and $\mathbf{P}$, given $f \in \X_{0}$, 
$$\mathsf{P}(0)f=0 \quad \iff \mathbf{P}f=0 \quad \iff \varrho_{f}=\int_{\O\times V}f(x,v)\d x\otimes\bm{m}(\d v)=0.$$
We introduce then
$$\X_{k}^{0}:=\{f \in \X_{k}\;;\;\varrho_{f}=0\}, \qquad k \in \N.$$
which is a closed subspace of $\X_{k}$. Notice that, endowed with the $\X_{k}$-norm, $\X_{k}^{0}$ is a Banach space.  
Since the semigroup $(\mathcal{V}(t))_{t\geq0}$ generated by $\A+\K$ is conservative: 
$$\int_{\O \times V}\mathcal{V}(t)f\d x\otimes \bm{m}(\d v)=\int_{\O\times V}f\d x\otimes \bm{m}(\d v), \qquad \forall t \geq0, \quad f \in \X_{0}$$
one has
\begin{equation*}\begin{split}
\int_{\Omega\times V}\Rs(\l,\A+\K)f\d x\otimes \bm{m}(\d v)&=\int_{\Omega\times V}\left(\int_{0}^{\infty}e^{-\l\,t}V(t)f\d t\right)\d x\otimes \bm{m}(\d v)\\
&=\frac{1}{\l}\int_{\Omega\times V}f\d x\otimes \bm{m}(\d v), \qquad \forall \l \in \C_{+}\end{split}\end{equation*}
and therefore the resolvent and all its iterates $\Rs(\l,\A+\K)^{k}$ leave $\X_{0}^{0}$ invariant ($k \geq 0$). 
We  deduce from Proposition \ref{propo:P0conv} the following
\begin{lemme}\label{propo:P1conv}
Let $f \in \X_{1}$ with $\varrho_{f}=0$, i.e. $f \in \X_{1}^{0},$ then
$${\lim_{\l\to0}}\,\Rs(\l,\A+\K)f,$$
exists in $\X_{0}$. We denote it by $\Rs(0)f$ and has
$$\Rs(0)f=\Rs(0,\A)f+\Rs(0,\A)\Ms_{0}\left[\Rs\left(1,\Ms_{0}\left(\mathbf{I}-\mathsf{P}(0)\right)\right)f + \mathsf{\Phi}_{0}f\right]\,$$
\end{lemme}
\begin{proof} With the previous notations, if $f \in \X_{1}$, using the fact  that $\Ms_{\l}$ and $\mathsf{P}(\l)$ commute, one has
\begin{equation}\label{eq:RsA+K3term}\begin{split}
\Rs&(\l,\A+\K)f=\Rs(\l,\A)f + \Rs(\l,\A)\Ms_{\l}\Rs(1,\Ms_{\l})f\\
&=\Rs(\l,\A)f+\Rs(\l,\A)\Ms_{\l}\Rs(1,\Ms_{\l})\left(\mathbf{I}-\mathsf{P}(\l)\right)f + \Rs(\l,\A)\Ms_{\l}\Rs(1,\Ms_{\l})\mathsf{P}(\l)f\\
&=\Rs(\l,\A)f+\Rs(\l,\A)\Ms_{\l}\Rs\left(1,\Ms_{\l}\left(\mathbf{I}-\mathsf{P}(\l)\right)\right)f + \Rs(\l,\A)\Ms_{\l}\Rs(1,\Ms_{\l})\mathsf{P}(\l)f\,.
\end{split}\end{equation}
We investigate separately the convergence of each of the last three terms. First, since $f \in \X_{1}$,
$$\lim_{\l\to0}\left\|\Rs(\l,\A)f-\Rs(0,\A)f\right\|_{\X_{0}}=0.$$
Now, for the second term,  recall that $\H_{\l}=\Ms_{\l}^{\mathsf{q}}$. Since $\left\{\H_{\l}\left(\mathbf{I}-\mathsf{P}(\l)\right)\right\}_{\mathrm{Re}\l \in [0,1]}$ is collectively compact and strongly converges to $\H_{0}(\mathbf{I}-\mathsf{P}(0))$ as $\l \to 0$ with $r_{\sigma}\left(\H_{0}(\mathbf{I}-\mathsf{P}(0))\right) < 1$, we deduce from the convergence of the spectrum established in \cite[Theorem 5.3(a)]{anselone} that, for $\delta >0$ small enough, there is $c\in (0,1)$ such that
$$r_{\sigma}\left(\H_{\l}\left(\mathbf{I}-\mathsf{P}(\l)\right)\right) \leq c < 1, \qquad \forall \l \in \mathscr{C}_{\delta}.$$
Moreover, %since 
%$$\lim_{\l\to0}\left\|\H_{\l}\left(\mathbf{I}-\mathsf{P}(\l)\right)\varphi-\H_{0}\left(\mathbf{I}-\mathsf{P}(0)\right)\varphi\right\|_{\X_{0}}=0\,, \qquad \forall \varphi \in \X_{0}$$
we deduce from \cite[Theorem 5.3 (d)]{anselone} that
$$\lim_{\l\to0}\left\|\Rs\left(1,\H_{\l}\left(\mathbf{I}-\mathsf{P}(\l)\right)\right)g-\Rs\left(1,\H_{0}\left(\mathbf{I}-\mathsf{P}(0)\right)\right)g\right\|_{\X_{0}}=0, \qquad \forall g \in \X_{0}.$$
Arguing exactly as in the proof of Lemma \ref{propo:convRsAKeta}, we deduce first that
$$\lim_{\l\to0}\left\|\Ms_{\l}\Rs\left(1,\H_{\l}\left(\mathbf{I}-\mathsf{P}(\l)\right)\right)g-\Ms_{0}\Rs\left(1,\H_{0}\left(\mathbf{I}-\mathsf{P}(0)\right)\right)g\right\|_{\X_{1}}, \qquad \forall g \in \X_{0}$$
and then
$$\lim_{\l\to0}\left\|\Rs(\l,\A)\Ms_{\l}\Rs\left(1,\H_{\l}\left(\mathbf{I}-\mathsf{P}(\l)\right)\right)f-\Rs(0,\A)\Ms_{0}\Rs\left(1,\H_{0}\left(\mathbf{I}-\mathsf{P}(0)\right)\right)f\right\|_{\X_{0}}=0.$$
The convergence of the third term is dealt with in the same way. Indeed, since $f \in\X_{1}^{0}$, we deduce from Proposition \ref{propo:P0conv} that
$$\lim_{\l\to0}\left\|\Rs(1,\Ms_{\l})\mathsf{P}(\l)f-\mathsf{\Phi}_{0}f\right\|_{\X_{0}}=0.$$
Then, as before,
$$\lim_{\l\to0}\left\|\Ms_{\l}\Rs(1,\Ms_{\l})\mathsf{P}(\l)f-\Ms_{0}\mathsf{\Phi}_{0}f\right\|_{\X_{1}}=0$$
and subsequently
$$\lim_{\l\to0}\left\|\Rs(\l,\A)\Ms_{\l}\Rs(1,\Ms_{\l})\mathsf{P}(\l)f-\Rs(0,\A)\Ms_{0}\mathsf{\Phi}_{0}f\right\|_{\X_{0}}=0.$$
We proved the convergence of the three terms in the right-hand-side of \eqref{eq:RsA+K3term} and this shows the result.\end{proof}
We deduce from this the following which holds under Assumption \ref{hypH} but \emph{does not resort on} \eqref{eq:K2}--\eqref{eq:logM} and the fact that $\bm{m}(\d v)$ is absolutely continuous with respect to the Lebesgue measure (see also Remark \ref{nb:ingham}):
 \begin{theo}\label{theo:P1conv}
Let $f \in \X_{1}$ with $\varrho_{f}=0$, i.e. $f \in \X_{1}^{0},$ then the limit
$${\lim_{\e\to0}}\,\Rs(\e+i\eta,\A+\K)f,$$
exists in $\X_{0}$. We denote it by $\Rs(\eta)f$ and has
$$\Rs(\eta)f=\begin{cases} \Rs(i\eta,A)f+ \Rs(i\eta,\A)\Ms_{i\eta}\Rs(1,\Ms_{i\eta})f, \qquad &\text{ if } \quad \eta \neq 0\\
\Rs(0,\A)f+\Rs(0,\A)\Ms_{0}\left[\Rs\left(1,\Ms_{0}\left(\mathbf{I}-\mathsf{P}(0)\right)\right)f + \mathsf{\Phi}_{0}f\right]\,\qquad &\text{ if } \quad \eta=0\,.\end{cases}$$
Moreover, the convergence of $\Rs(\e+i\eta,\A+\K)f$ to $\Rs(\eta)f$ is uniform with respect to $\eta \in\R$.
\end{theo}
\begin{proof} The convergence of $\Rs(\e+i\eta,\A+\K)f$ towards the desired limit is just the combination of the two previous Lemmas. We only need to check the uniform convergence in the vicinity of $\eta=0$ and near infinity. Let us first prove that the convergence is uniform with respect to $|\eta| \leq \delta.$  According to \eqref{eq:strongMiekUn}
 and Corollary \ref{cor:double}, we only need to prove that the convergence in $\X_{0}$
$$\lim_{\e\to0^{+}}\Rs(1,\Ms_{\e+i\eta})f=\mathsf{\Psi}(\eta)$$
is uniform with respect to $|\eta| < \delta$ where 
$$\mathsf{\Psi}(\eta)=\begin{cases}
\Rs(1,\mathsf{M}_{i\eta})f \qquad &\text{ if } \eta \neq 0\\
\Rs\left(1,\Ms_{0}\left(\mathbf{I}-\mathsf{P}(0)\right)\right)f+\mathsf{\Phi}_{0}f \qquad &\text{ if } \eta =0.\end{cases}
$$
 We argue by contradiction, assuming that there exist $c >0$, a sequence $(\e_{n})_{n} \subset (0,\infty)$ converging to $0$ and a sequence $(\eta_{n})_{n} \subset (-\delta,\delta)$ such that
\begin{equation}\label{eq:absurd}
\left\|\Rs(1,\Ms_{\e_{n}+i\eta_{n}})-\mathsf{\Psi}(\eta_{n})\right\|_{\X_{0}} \geq c >0.\end{equation}
Up to considering a subsequence, if necessary, we can assume without loss of generality that $\lim_{n}\eta_{n}=\eta_{0}$ with $|\eta_{0}| \leq \delta.$ First, one sees that then $\eta_{0}=0$ since the convergence of $\Rs(1,\mathsf{M}_{\e+i\eta})f$ to $\mathsf{\Psi}(\eta)$ is actually uniform in any neighbourhood around $\eta_{0} \neq 0$ (see  Lemma \ref{cor:ResMei}). Because $\eta_{0}=0$, defining $\lambda_{n}:=\e_{n}+i\eta_{n}$, $n \in \N$, the sequence $(\l_{n})_{n}\subset \mathscr{C}_{\delta}$ is converging to $0$. Now, from Lemma \ref{propo:P1conv} one has
$$\lim_{n\to\infty}\Rs(1,\Ms_{\l_{n}})f=\Rs\left(1,\Ms_{0}\left(\mathbf{I}-\mathsf{P}(0)\right)\right)f+\mathsf{\Phi}_{0}f$$
as $n \to \infty$.  Moreover, one also has
$$\mathsf{\Psi}(\eta_{n})=\Rs(1,\Ms_{i\eta_{n}})f$$
also converges as $n \to \infty$ towards $\Rs\left(1,\Ms_{0}\left(\mathbf{I}-\mathsf{P}(0)\right)\right)f+\mathsf{\Phi}_{0}f$ still thanks to Lemma \ref{propo:P1conv} (recall the convergence holds for $\l \to 0$, $\l \in \overline{\C}_{+}$ including the case of purely imaginary $\l$). This contradicts \eqref{eq:absurd}. The uniform convergence in the vicinity of $\eta=0$ together with Lemma \ref{propo:convRsAKeta}, we deduce that, for any $\eta_{0} \in \R$, there is  $\delta >0$ such that, for any $f \in\X_{1}^{0}$, 
$$\lim_{\e\to0}\sup_{|\eta-\eta_{0}| \leq \delta}\bigg\|\Rs(\e+i\eta,\A+\K)f-\Rs(\eta)f\bigg\|_{\X_{0}}=0.$$
To prove that
$$\lim_{\e\to0}\sup_{\eta \in \R}\bigg\|\Rs(\e+i\eta,\A+\K)f-\Rs(\eta)f\bigg\|_{\X_{0}}=0$$
we only to check that there is $R >0$ (large enough) such that
$$\lim_{\e\to0}\sup_{|\eta|>R}\bigg\|\Rs(\e+i\eta,\A+\K)f-\Rs(\eta)f\bigg\|_{\X_{0}}=0.$$
One observes that, for $|\eta| >R >0$, 
\begin{multline*}
\Rs(\e+i\eta,\A+\K)f-\Rs(\eta)f
=\left[\Rs(\e+i\eta,\A)f-\Rs(i\eta,\A)f\right]\\
+\left[\Rs(\e+i\eta,\A)-\Rs(i\eta,\A)\right]\Ms_{i\eta}\Rs(1,\Ms_{i\eta})f\\
+\Rs(\e+i\eta,\A)\left[\Ms_{\e+i\eta}\Rs(1,\Ms_{\e+i\eta})f-\Ms_{i\eta}\Rs(1,\Ms_{i\eta})f\right]\end{multline*}
where the two first terms converge to $0$ as $\e\to0$ uniformly with respect to $|\eta|>R$ according to Eq. \eqref{eq:RetaXs} in Proposition \ref{prop:convRsT0}. To prove the result, we only need to prove that
$$\lim_{\e\to0}\sup_{|\eta|>R}\left\|\Rs(\e+i\eta,\A)\left[\Ms_{\e+i\eta}\Rs(1,\Ms_{\e+i\eta})f-\Ms_{i\eta}\Rs(1,\Ms_{i\eta})f\right]\right\|_{\X_{0}}=0$$
which is deduced from Remark \ref{nb:I=R} since one checks easily that
$$\lim_{|\eta|\to\infty}\left\|\Ms_{\e+i\eta}\Rs(1,\Ms_{\e+i\eta})f-\Ms_{i\eta}\Rs(1,\Ms_{i\eta})f\right\|_{\X_{1}}=0.$$
This proves the convergence is uniform with respect to $\eta \in \R$.
%Therefore, one has
%$$\lim_{\e\to0^{+}}\sup_{|\eta|<\delta}\left\|\Rs(1,\mathsf{M}_{\e+i\eta}\H)\mathsf{G}_{\e+i\eta}f-\mathsf{\Phi}(\eta)f \right\|_{\X_{0}}=0.$$
\end{proof} 
\begin{nb}\label{nb:ingham} For any $f \in \X_{1}^{0}$, the above Theorem asserts that the resolvent of $\Rs(\l,\A+\K)f$ admits a trace $\Rs(\eta)f$ on the boundary $\l=i\eta$, $\eta \in \R$ with $\eta \mapsto \Rs(\eta)f$ belonging to $\mathscr{C}_{0}(\R,\X_{0})$ and, in particular, to $L^{1}_{\mathrm{loc}}(\R,\X_{0})$. Therefore, according to Ingham's theorem (see e.g. \cite[Theorem 1.1]{chill}), we directly deduce that $\mathcal{V}(t)f \to 0$ as $t \to \infty$ for any $f \in \X_{1}^{0}$ which, in turns, implies that
$$\lim_{t \to \infty}\|\mathcal{V}(t)f-\varrho_{f}\Psi\|_{\X_{0}}=0 \qquad \forall f \in \X_{1}.$$
By a density argument, the above still holds for any $f \in \X_{0}$, providing a \emph{new non quantitative convergence} to equilibrium under the sole assumption of collective compactness in Theorem \ref{theo:main-assum}. Notice here indeed that, up to now, we did not use the quantitative estimate  \eqref{eq:decayP}  which is the only one in our analysis which requires assumptions \eqref{eq:K2}--\eqref{eq:logM} and the fact that $\bm{m}(\d v)$ is absolutely continuous with respect to the Lebesgue measure. 
\end{nb}

Again, under additional integrability of $f$, one can upgrade the convergence as follows
\begin{cor}\label{cor:convRetRk} Let $N \in \{1,\ldots,N_{0}\}$ and $f \in \X_{N}^{0},$ then
$$\lim_{\e\to0}\sup_{\eta\in\R}\left\|\Rs(\e+i\eta,\A+\K)f-\Rs(\eta)f\right\|_{\X_{N-1}}=0.$$
\end{cor}
\begin{proof} For $N=1,$ the result is nothing but Theorem \ref{theo:P1conv}. For $N \geq 1$, the proof is similar and based upon the representation
$$\Rs(\e+i\eta,\A+\K)f=\Rs(\e+i\eta,\A)f+\Rs(\e+i\eta,\A)\Ms_{\e+i\eta}\Rs(1,\Ms_{\e+i\eta})f, \qquad \e >0,\eta \in \R.$$
First, for $f \in \X_{k}$,
$$\lim_{\e\to0}\sup_{\eta\in\R}\left\|\Rs(\e+i\eta,\A)f-\Rs(i\eta,\A)f\right\|_{\X_{k-1}}=0.$$
We saw in the previous proof that
$$\lim_{\e\to0}\sup_{\eta\in\R}\left\|\Rs(1,\Ms_{\e+i\eta})f-\mathsf{\Psi}(\eta)\right\|_{\X_{0}}=0$$
as soon as $f \in \X_{1}^{0}.$ According to Remark \ref{nb:doubleMs} (see Eq. \eqref{eq:strongMiekUn}), this implies that 
$$\lim_{\e\to0}\sup_{\eta\in \R}\left\|\Ms_{\e+i\eta}\Rs(1,\Ms_{\e+i\eta})f-\Ms_{i\eta}\mathsf{\Psi}(\eta)\right\|_{\X_{N}}=0$$
for any $N \leq N_{0}$. Then, according to Remark \ref{nb:I=R}
\begin{equation}\label{eq:Sn0trac}
\lim_{\e\to0}\sup_{\eta\in \R}\left\|\Rs(\e+i\eta,\A)\Ms_{\e+i\eta}\Rs(1,\Ms_{\e+i\eta})f-\Rs(i\eta,\A)\Ms_{i\eta}\mathsf{\Psi}(\eta)\right\|_{\X_{N-1}}=0\end{equation}
and the result follows since $\Rs(i\eta,\A)\Ms_{i\eta}\mathsf{\Psi}(\eta)=\Rs(\eta)f-\Rs(i\eta,\A)f.$
\end{proof}

\subsection{Regularity of the boundary function} The previous results ensure that the extension $\Rs(\eta)$ of $\Rs(\l,\A+\K)$ to the imaginary axis $\l=i\eta$ is continuous and tends to zero at infinity, namely,
$$\Rs(\cdot)f \in \mathscr{C}_{0}(\R,\X_{k-1}), \qquad \forall f \in \X_{k}^{0}, \qquad k=1,\ldots,N_{0}.$$
We need to extend this regularity to capture some differentiability property. The key point is the well-known formula for the resolvent: for any $k \in \N$ and any $\e >0,\eta\in\R$,
$$\dfrac{\d^{k}}{\d\eta^{k}}\Rs(\e+i\eta,\A+\K)f=(-i)^{k} k!\left[\Rs(\e+i\eta,\A+\K)\right]^{k+1}f.$$
In particular, from Lemma \ref{lem:convDerRsT0}, we also recall that, for any $k \in \N$ and any $f \in \X_{k+1}$, the mapping
$$\eta \in \R \longmapsto \Rs(i\eta,\A)f \in \X_{0}$$
belongs to $\mathscr{C}_{0}^{k}(\R,\X_{0})$. To prove then that the same holds for the boundary function of $\Rs(\e+i\eta,\A+\K)f$ we need to investigate iterates of the resolvent. To do so, we recall the notation introduced in \eqref{eq:Snl}, for $\e >0,\eta \in \R$, $n \in \N$,
$$\mathscr{S}_{n}(\e+i\eta)=\sum_{k=n}^{\infty}\Rs(\e+i\eta,\A)\Ms_{\e+i\eta}^{k}=\Rs(\e+i\eta,\A)\Ms_{\e+i\eta}^{n}\Rs(1,\Ms_{\e+i\eta})$$
so that
$$\Rs(\e+i\eta,\A+\K)=\Rs(\e+i\eta,\A) + \mathscr{S}_{0}(\e+i\eta).$$
According to \eqref{eq:Sn0trac}, for any $j \leq N_{0}$,
$$\lim_{\e\to0}\sup_{\eta \in \R}\left\|\mathscr{S}_{0}(\e+i\eta)g-\mathscr{S}_{0}(i\eta)g\right\|_{\X_{j-1}}=0, \qquad \forall g \in \X_{1}^{0}$$
where
$$\mathscr{S}_{0}(i\eta)=\Rs(i\eta,\A)\Ms_{i\eta}\mathsf{\Psi}(\eta)=\begin{cases}\Rs(i\eta,\A)\Ms_{i\eta}\Rs(1,\Ms_{i\eta})g \quad \text{ if } \eta \neq 0\\
\Rs(0,\A)\Ms_{0}\Rs\left(1,\Ms_{0}\left(\mathbf{I}-\mathsf{P}(0)\right)\right)g+\mathsf{\Phi}_{0}g 	\quad \text{ if } \eta=0\end{cases}$$
One has then the following
\begin{lemme}\label{lem:UPS}
Let $N \in \{1,\ldots,N_{0}\}$, for any $f \in \X_{N}^{0}$ and any $k \in \{1,\ldots,N\}$, one has
$$\lim_{\e\to0}\sup_{\eta\in\R}\left\|\left[\Rs(\e+i\eta,\A+\K)\right]^{k}f-\Rs(\eta)^{k}f\right\|_{\X_{N-k}}=0$$
where $\Rs(\eta)^{k}f \in \X_{N-k}^{0}$ for any $\eta \in \R$.

In particular, if $f \in \X_{N_{0}}^{0}$ then
$$\lim_{\e\to0}\sup_{\eta\in\R}\left\|\left[\Rs(\e+i\eta,\A+\K)\right]^{k}f-\Rs(\eta)^{k}f\right\|_{\X_{N_{0}-k}}=0, \qquad \forall k \in \{1,\ldots,N_{0}\}$$
\end{lemme}
\begin{proof} Let $N \in \{1,\ldots,N_{0}\}$ be given. The proof is made by induction over $k \in \{1,\ldots,N\}$. For $k=1$, the result is exactly Corollary \ref{cor:convRetRk}. Let us assume that the result is true for $k\in \{1,\ldots,N-1\}$  and let us prove it for $k+1$. As part of the induction assumption, one has
$$\int_{\O\times V}\Rs(\eta)^{k}f\d x\,\bm{m}(\d v)=0.$$
Observe that
\begin{multline*}
\left[\Rs(\e+i\eta,\A+\K)\right]^{k+1}f=\Rs(\e+i\eta,\A)\left[\Rs(\e+i\eta,\A+\K)\right]^{k}f\\
+\mathscr{S}_{0}(\e+i\eta)\left[\Rs(\e+i\eta,\A+\K)\right]^{k}f\,.\end{multline*}
According to the induction hypothesis, 
$$\lim_{\e\to0}\left[\Rs(\e+i\eta,\A+\K)\right]^{k}f=\Rs(\eta)^{k}f$$
holds in $\X_{N-k}$ uniformly with respect to $\eta \in \R$. Thanks to Corollary \ref{cor:double} (and since $N-k\geq 1$), 
$$\lim_{\e\to0}\Rs(\e+i\eta,\A)\left[\Rs(\e+i\eta,\A+\K)\right]^{k}f=\Rs(i\eta,\A)\Rs(\eta)^{k}f$$
holds in $\X_{N-k-1}=\X_{N-(k+1)}$ uniformly with respect to $\eta\in \R.$ Now, since $\Rs(\eta)^{k}f \in \X_{N-k}^{0}$ by induction assumption, we can resume the proof of Corollary \ref{cor:double} together with \eqref{eq:Sn0trac} to deduce also that
$$\lim_{\e\to0}\mathscr{S}_{0}(\e+i\eta,\A)\left[\Rs(\e+i\eta,\A+\K)\right]^{k}f=\mathscr{S}_{0}(i\eta)\Rs(i\eta,\A)\Rs(\eta)^{k}f$$
holds true in $\X_{N-k-1}$ uniformly with respect to $\eta \in \R$. We deduce that
$$\lim_{\e\to0}\sup_{\eta\in\R}\left\|\left[\Rs(\e+i\eta,\A+\K)\right]^{k+1}f-\left[\Rs(i\eta,\A)\right]^{k+1}f-\mathscr{S}_{0}(i\eta)\Rs(\eta)^{k}f\right\|_{\X_{N-k-1}}=0$$
which proves the result with 
$$\Rs(\eta)^{k+1}f=\left[\Rs(i\eta,\A)\right]^{k+1}f-\mathscr{S}_{0}(i\eta)\Rs(\eta)^{k}f.$$
This achieves the induction and proves the Lemma.
\end{proof}
A fundamental consequence  is 
\begin{cor}\label{cor:reguRf}  For any $f \in \X_{N_{0}}^{0},$ the mapping 
$$\eta \in \R \longmapsto \Rs(\eta)f $$
defined in Theorem \ref{theo:P1conv} belongs to $\mathscr{C}_{0}^{N_{0}-1}(\R,\X_{0})$ and the convergence
$$\lim_{\e\to0^{+}}\Rs(\e+i\eta,\A+\K)f=\Rs(\eta)f$$
holds in $\mathscr{C}_{0}^{N_{0}-1}(\R,\X_{0}).$\end{cor}
\begin{proof} Let $f \in \X_{N_{0}}^{0}$ be fixed. Since, for any $k \in \{0,\ldots,N_{0}-1\}$ and any $\e >0,\eta\in\R$,
$$\dfrac{\d^{k}}{\d\eta^{k}}\Rs(\e+i\eta,\A+\K)f=(-i)^{k} k!\left[\Rs(\e+i\eta,\A+\K)\right]^{k+1}f$$
the result follows directly  from Lemma \ref{lem:UPS} where the derivatives of $\Rs(\eta)f$ are defined by
$$\dfrac{\d^{k}}{\d\eta^{k}}\Rs(i\eta)f=(-i)^{k} k!\left[\Rs(i\eta)\right]^{k+1}f$$
for any $k \in \{0,\ldots,N_{0}-1\}.$
\end{proof}

\section{The boundary function of $\mathscr{S}_{n}(\l)$} \label{sec:near0}
 
This section is devoted to the construction of the trace along the imaginary axis, that is when $\lambda=i\eta$, $\eta \in \R$,  of  
$$\mathscr{S}_{n}(\l)f=\Rs(\l,\A)\Ms_{\l}^{n}\Rs(1,\Ms_{\l})f, \qquad \l \in \mathbb{C}, \mathrm{Re}\l \geq0,\: n \in \N$$
for a suitable class of function $f$. Recall that $\mathscr{S}_{n}(\l)$ has been introduced in \eqref{eq:Snl}. Of course, the crucial observation is  the alternative representation of $\mathscr{S}_{n}(\l)f$ as
\begin{equation}\label{eq:dif}
\mathscr{S}_{n}(\l)=\Rs(\l,\A+\K)-\Rs(\l,\A)-\sum_{k=0}^{n-1}\Rs(\l,\A)\Ms_{\l}^{k}\, \qquad \l \in \C_{+}\end{equation}
We already investigated the existence and regularity of the traces on the imaginary axis of the first two terms in the right-hand-side of \eqref{eq:dif} so we just need to focus on the properties of the \emph{finite sum} 
\begin{equation}\label{def:snl}
{s}_{n}(\l):=\sum_{k=0}^{n}\Rs(\l,\A)\Ms_{\l}^{k}\ \qquad \l \in \overline{\C}_{+}.\end{equation}
 All the above results allow to prove the regularity the \emph{finite} sum $s_{n}(\l)$ defined by \eqref{def:snl}, the proof of which is deferred to Appendix \ref{appen:Ml}:
\begin{lemme}\label{prop:snl}  Let $f \in  \X_{N_{0}}$ be fixed. For any $\e >0$, the mapping 
$$\eta \in \R \longmapsto s_{n}(\e+i\eta)f \in \X_{0} \quad \text{ belongs to }  \mathscr{C}^{N_{0}-1}_{0}(\R,\X_{0})$$ 
for any $n \in \N$. Moreover, 
$$\lim_{\e\to0}\dfrac{\d^{k}}{\d\eta^{k}}s_{n}(\e+i\eta)f  \qquad k \in \{0,\ldots,N_{0}-1\}$$
exist uniformly with respect to $\eta \in \R$. In particular, the mapping
$$\eta \in \R \longmapsto  s_{n}(i\eta)f:=\lim_{\e\to0}s_{n}(\e+i\eta)f \in \X_{0}$$
belongs to $\mathscr{C}^{N_{0}-1}_{0}(\R,\X_{0}).$
\end{lemme}
 We have all the tools to prove the first point in Theorem \ref{theo:MainLaplace} 
 \begin{propo}\label{theo:existtrace} Let $f \in \X_{N_{0}}$ be such that
\begin{equation}\label{eq:0mean}
\varrho_{f}=\int_{\Omega\times V}f(x,v)\d x \otimes \bm{m}(\d v)=0.\end{equation}
Then, for any $n\geq0$ the limit
$$\lim_{\e\to0^{+}}\mathscr{S}_{n}(\e+i\eta)f,$$
exists in $\mathscr{C}_{0}^{N_{0}-1}(\R,\X_{0})$. Its limit is denoted ${\Upsilon}_{n}(\eta)f$.\medskip
\end{propo}
\begin{proof}
 We know from Corollary \ref{cor:reguRf} that
$$\lim_{\e\to0^{+}}\Rs(\e+i\eta,\A+\K)f=\Rs(i\eta,\A)f$$
holds in $\mathscr{C}_{0}^{N_{0}-1}(\R,\X_{0}).$ In the same way, Lemma \ref{lem:convDerRsT0} shows that 
$$\lim_{\e\to0^{+}}\Rs(\e+i\eta,\A)f=\Rs(i\eta,\A)f$$
holds in $\mathscr{C}_{0}^{N_{0}-1}(\R,\X_{0})$. Since one sees easily from Lemma \ref{prop:snl} that
$$\lim_{\e\to0^{+}}s_{n}(\e+i\eta)f=s_{n}(i\eta)f \qquad \text{ in } \quad \mathscr{C}_{0}^{N_{0}-1}(\R,\X_{0})$$
we get the result from the representation \eqref{eq:dif} which asserts that $\mathscr{S}_{n}(\e+i\eta)f=\Rs(\l,\A+\K)f-\Rs(\l,\A)f-s_{n-1}(\e+i\eta)f$.
\end{proof}
In the following, we show also that, if $n$ is large enough, the boundary function $\Upsilon_{n}(\cdot)f$ and its derivatives are  integrable which proves the second point of Theorem \ref{theo:MainLaplace}
\begin{propo}\label{prop:reguPsif}  Assume that $n \geq 5\cdot 2^{N_{0}-1}$ and $f \in \X_{N_{0}}^{0}$. Then, the derivatives of the trace function
$$\eta \in \R \longmapsto \Upsilon_{n}(\eta)f \in \X_{0}$$ are  integrable, i.e.
$$\int_{\R}\left\|\frac{\d^{k}}{\d\eta^{k}}\Upsilon_{n}(\eta)f\right\|_{\X_{0}}\d\eta < \infty \qquad \forall k \in \{0,\ldots,N_{0}-1\}.$$
\end{propo}
\begin{proof} Let $k \in \{0,\ldots,N_{0}-1\}$ be given as well as $f \in \X_{N_{0}}^{0}.$ Using the continuity of $\dfrac{\d^{k}}{\d\eta^{k}}\Upsilon_{n}(\eta)f$, it is clear that the mapping
$$\eta \in \R \mapsto \dfrac{\d^{k}}{\d\eta^{k}}\Upsilon_{n}(\eta)f$$
is locally integrable. It is enough then to prove that there is $R >0$ such that
\begin{equation}\label{eq:INtRR}
\int_{|\eta| >R}\left\|\frac{\d^{k}}{\d\eta^{k}}\Upsilon_{n}(\eta)f\right\|_{\X_{0}}\d\eta < \infty \qquad \forall k \in \{0,\ldots,N_{0}-1\}.
\end{equation}
We recall that, for any $\mathsf{p} >4$,
$$\int_{|\eta| >1}\left\|\Ms_{i\eta}^{\mathsf{p}}\right\|_{\B(\X_{0})} \d \eta < \infty.$$
We \emph{fix} $\mathsf{p} >4$ in all the rest of the proof. We recall that 
$$\mathsf{G}_{m}(\l)=\left[\Rs(\l,\A)\K\right]^{m}, \qquad \forall \l \in \overline{\C}_{+}, \qquad m \in \N$$
and we choose $R >0$ large enough so that
$$\left\|\mathsf{G}_{\mathsf{p}}(i\eta)\right\|_{\B(\X_{0})} \leq \frac{1}{2} \qquad \text{ for } |\eta| >R.$$
For $|\eta| >R$, we use the following representation of $\Upsilon_{N}(\eta)f$:
$$\Upsilon_{n}(\eta)f=\Rs(i\eta,\A)\Ms_{i\eta}^{n}\Rs(1,\Ms_{i\eta})f, \qquad |\eta| >R.$$
Writing,
$$\Rs(1,\Ms_{i\eta})f=\sum_{j=0}^{\infty}\Ms_{i\eta}^{j}f=\sum_{m=0}^{\infty}\sum_{r=0}^{\mathsf{p}-1}\Ms_{i\eta}^{m\mathsf{p}+r}f$$
we have
\begin{equation*}\begin{split}
\Upsilon_{n}(\eta)f&=\Rs(i\eta,\A)\Ms_{i\eta}^{n}\Rs(1,\Ms_{i\eta})f=\sum_{m=0}^{\infty}\sum_{r=0}^{\mathsf{p}-1}\Rs(i\eta,\A)\left[\K\Rs(i\eta,\A)\right]^{m\mathsf{p}+r+n}f\\
&=\sum_{m=0}^{\infty}\sum_{r=0}^{\mathsf{p}-1}\left[\Rs(i\eta,\A)\K\right]^{m\mathsf{p}+r+n}\Rs(i\eta,\A)f\\
&=\sum_{m=0}^{\infty}\sum_{r=0}^{\mathsf{p}-1}\mathsf{G}_{m\mathsf{p}+r+n}(i\eta)\Rs(i\eta,\A)f\,.
\end{split}\end{equation*}
Let us fix then $|\eta| >R$. Thanks to Leibniz rule
$$\dfrac{\d^{k}}{\d\eta^{k}}\Upsilon_{n}(\eta)f=(-i)^{k}\sum_{m=0}^{\infty}\sum_{r=0}^{\mathsf{p}-1}\sum_{j=0}^{k}{k \choose j}\mathsf{G}^{(j)}_{m\mathsf{p}+r+n}(i\eta)\left[\dfrac{\d^{k-j}}{\d\eta^{k-j}}\Rs(i\eta,\A)f\right] $$
so that
\begin{equation*}\begin{split}
\left\|\dfrac{\d^{k}}{\d\eta^{k}}\Upsilon_{n}(\eta)f\right\|_{\X_{0}}&\leq \sum_{m=0}^{\infty}\sum_{r=0}^{\mathsf{p}-1}\sum_{j=0}^{k}{k \choose j}\left\|\mathsf{G}^{(j)}_{m\mathsf{p}+r+n}(i\eta)\right\|_{\B(\X_{0})}\left\|\dfrac{\d^{k-j}}{\d\eta^{k-j}}\Rs(i\eta,\A)f\right\|_{\X_{0}}\\
&\leq \sum_{m=0}^{\infty}\sum_{r=0}^{\mathsf{p}-1}\sum_{j=0}^{k}{k \choose j}(k-j)!\left\|\mathsf{G}^{(j)}_{m\mathsf{p}+r+n}(i\eta)\right\|_{\B(\X_{0})}\|f\|_{\X_{k-j+1}}
\end{split}\end{equation*}
where we used \eqref{eq:dlkRslA}. Since $\|f\|_{\X_{k-j+1}} \leq \|f\|_{\X_{N_{0}}}$, we get
\begin{equation}\label{eq:derivPsi}\left\|\dfrac{\d^{k}}{\d\eta^{k}}\Upsilon_{n}(\eta)f\right\|_{\X_{0}} \leq  \|f\|_{\X_{N_{0}}}\sum_{m=0}^{\infty}\sum_{r=0}^{\mathsf{p}-1}\sum_{j=0}^{k} \left\|\mathsf{G}^{(j)}_{m\mathsf{p}+r+n}(i\eta)\right\|_{\B(\X_{0})}.\end{equation}
We use now Lemma \ref{lem:estJ} (see Eq. \eqref{eq:estJ}) to estimate, given $m,r,n$
$$ \left\|\mathsf{G}^{(j)}_{m\mathsf{p}+r+n}(i\eta)\right\|_{\B(\X_{j0})} \leq \bar{C}_{j} \left(m\mathsf{p}+r+n\right)^{j}\left\|\mathsf{G}_{\floor{\frac{m\mathsf{p}+r+n-j}{2^{j}}}}(i\eta)\right\|_{\B(\X_{0})}$$
Since $j \leq k$  and $r \leq \mathsf{p-1}$, setting $\bm{C}_{k}=\max_{j\leq k}\bar{C}_{j}$, one has
$$ \left\|\mathsf{G}^{(j)}_{m\mathsf{p}+r+n}(i\eta)\right\|_{\B(\X_{0})} \leq \bm{C}_{k} \left((m+1)\mathsf{p}+n\right)^{k}\left\|\mathsf{G}_{\floor{\frac{m\mathsf{p}+r+n-j}{2^{j}}}}(i\eta)\right\|_{\B(\X_{0})}$$
Then, since $n \geq 2^{k}\mathsf{p}+k \geq 2^{j}\mathsf{p}+j$, one has $\floor{\frac{m\mathsf{p}+r+n-j}{2^{j}}}\geq \mathsf{p}+\floor{\frac{m}{2^{j}}}\mathsf{p}$, i.e.
$$\mathsf{G}_{\floor{\frac{m\mathsf{p}+r+n-j}{2^{j}}}}(i\eta)=\mathsf{G}_{b}(i\eta)\mathsf{G}_{\mathsf{p}}(i\eta)\mathsf{G}_{\floor{\frac{m}{2^{j}}}\mathsf{p}}(i\eta)$$
for some $b \geq0$, and we deduce from  Lemma \ref{lem:normGn} that
%$$ \left\|\mathsf{G}^{(j)}_{m\mathsf{p}+r+N}(i\eta)\right\|_{\B(\X_{0})} \leq \bm{C}_{k} \|\sigma\|_{\infty}\|\vartheta_{{1}}\|_{\infty}\left((m+1)\mathsf{p}+N\right)^{k}\left\|\mathsf{G}_{\mathsf{p}}(i\eta)\right\|_{\B(\X_{0})}$$
%thanks to \eqref{eq:Gn+k}.
\begin{equation*}\begin{split}
\left\|\mathsf{G}_{m\mathsf{p}+r+n}^{(j)}(i\eta)\right\|_{\mathscr{B}(\X_{0})} &\leq \bm{C}_{k}\|\sigma\|_{\infty}\|\vartheta_{{1}}\|_{\infty}\|\mathsf{G}_{\mathsf{p}}(i\eta)\|_{\mathscr{B}(\X_{0})}\left((m+1)\mathsf{p}+n\right)^{k}\left\|\mathsf{G}_{\floor{\frac{m}{2^{j}}}\mathsf{p}}(i\eta)\right\|_{\mathscr{B}(\X_{0})}\\
&\leq \bm{C}_{k}\|\sigma\|_{\infty}\|\vartheta_{{1}}\|_{\infty}\|\mathsf{G}_{\mathsf{p}}(i\eta)\|_{\mathscr{B}(\X_{0})}\left((m+1)\mathsf{p}+n\right)^{k}2^{-\floor{\frac{m}{2^{j}}}}
\end{split}\end{equation*}
since $\mathsf{G}_{\floor{\frac{m}{2^{j}}}\mathsf{p}}(i\eta)=\mathsf{G}_{\mathsf{p}}(i\eta)^{\floor{\frac{m}{2^{j}}}}$ and we choose $R >0$ such that $\|\mathsf{G}_{p}(i\eta)\|_{\B(\X_{0})} \leq \frac{1}{2}$ for $|\eta| >R$.  
Noticing that 
$$\sum_{m=0}^{\infty} \sum_{r=0}^{\mathsf{p}-1}\sum_{j=0}^{k}\left((m+1)\mathsf{p}+n\right)^{k}2^{-\floor{\frac{m}{2^{j}}}} \leq \mathsf{p}(k+1)\sum_{m=0}^{\infty}\left((m+1)\mathsf{p}+n\right)^{k}2^{-\floor{\frac{m}{2^{k}}}} < \infty$$
we deduce from \eqref{eq:derivPsi} that there is a positive constant $\alpha_{k}(n)>0$ such that
\begin{equation}\label{eq:estimLpPsi}
\left\|\dfrac{\d^{k}}{\d\eta^{k}}\Upsilon_{n}(\eta)f\right\|_{\X_{0}} \leq \alpha_{k}(n)\|f\|_{\X_{N_{0}}}\left\|\mathsf{G}_{\mathsf{p}}(i\eta)\right\|_{\mathscr{B}(\X_{0})} \qquad \forall |\eta| >R.\end{equation}
We deduce then \eqref{eq:INtRR} thanks to \eqref{eq:power}. 
\end{proof}

We have all in hands to give the full proof of Theorem \ref{theo:MainLaplace} from which, as pointed out in Section \ref{sec:Main}, our main convergence Theorem \ref{theo:maindec} is deduced.
\begin{proof}[Proof of Theorem \ref{theo:MainLaplace}] The previous two propositions give a complete proof of points \textit{(1)} and \textit{(2)} of Theorem \ref{theo:MainLaplace}. In turns, recalling that (see Proposition \ref{propo:Sn+1} in Appendix \ref{appen:DYSON})
\begin{equation}\label{eq:maindecay}
\bm{S}_{n+1}(t)f=\frac{\exp(\e t)}{2\pi}\lim_{\ell\to\infty} \int_{-\ell}^{\ell}\exp\left(i\eta t\right)\mathscr{S}_{n+1}(\e+i\eta)f\d\eta, \qquad \forall f \in \X_{0}\end{equation}
for any $t >0$, $\e >0$, we  deduce from the uniform convergence obtained in Proposition \ref{theo:existtrace} together with the integrability condition in Proposition \ref{prop:reguPsif} that,  for any  $f \in \X_{N_{0}}$ and any $t \geq0$
\begin{equation*}\label{conv:integ}
\lim_{\e\to0^{+}}\frac{1}{2\pi}\int_{-\infty}^{\infty}\exp\left((\e+i\eta)t\right)\mathscr{S}_{n+1}(\e+i\eta)f\d \eta=
\frac{1}{2\pi}\int_{-\infty}^{\infty}\exp\left(i\eta t\right)\Upsilon_{n+1}(\eta)f\d \eta\end{equation*}
where the convergence occurs in $\X_{0}$ as soon as $n+1 \geq  5\cdot 2^{N_{0}-1}$ thanks to the dominated convergence theorem. This, together with \eqref{eq:maindecay} shows \eqref{eq:SntInt-main}.  The proof of \eqref{eq:SnDeri-main}
 is then deduced easily after $N_{0}-1$ integration by parts, using again Proposition \ref{prop:reguPsif}. This achieves the proof.\end{proof}

%Theorem \ref{theo:existtrace} allows to establish the following new representation of $\bm{S}_{n+1}(t)$ where we recall that $\bm{S}_{n+1}(t)$ is a reminder of the Dyson-Phillips representation series defining the semigroup generated by $\A+\K$ and  has been defined in \eqref{eq:Sn+1}. Namely, we are in position to give an expression of $\bm{S}_{n+1}(t)$ which is alternative to that provided in Proposition \ref{propo:Sn+1}:

%We have all in hands to prove our main result about $\bm{S}_{n+1}.$

 \appendix 
 
 \section{Properties of the Dyson-Phillips iterated} \label{appen:DYSON}
 
\subsection{Continuous dependence with respect to $\K$}
We begin with recalling that the Dyson-Phillips iterated are depending continuously of $\K \in \mathscr{B}(\X_{0})$. We can be more precise here. Namely, let us consider a sequence $\left(\K_{n}\right)_{n\in\N}$ of boundary operators 
$$\left(\K_{n}\right)_{n} \subset \mathscr{B}(\X_{-1},\X_{0}), \qquad \|\K_{n}\|_{\mathscr{B}(\X_{-1},\X_{0})} \leq 1, \qquad \forall n \in \N$$
and introduce $V_{0}(t)=U_{0}(t)$ and
\begin{equation}\label{eq:VnKn}
V_{n+1}(t)=\int_{0}^{t}V_{n}(t-s)\K_{n+1}U_{0}(s)\d s, \qquad t \geq0, \qquad n \in \N.\end{equation}
Then
\begin{propo}\label{prop:Vnt}
For any $n \geq 1$ and any $t \geq 0$, ${V}_{n}(t) \in \mathscr{B}(\X_{0})$ with
\begin{equation}\label{eq:Vnt}
\left\|{V}_{n}(t)\right\|_{\mathscr{B}(\X_{0})} \leq \prod_{j=1}^{n}\left\|\K_{j}\right\|_{\mathscr{B}(\X_{-1},\X_{0})}, \qquad n \geq 1,\qquad t\geq0.\end{equation}
\end{propo}
 \begin{proof} The proof is made by induction.  Let $f \in \X_{-1}$ and $t\geq0$ be fixed. For $n=1$, one has 
$$\|{V}_{1}(t)f\|_{\X_{0}}=\int_{0}^{t}\|U_{0}(t-s)\K_{1}U_{0}(s)f\|_{\X_{0}}\d s \leq \int_{0}^{t}\|\K_{1}U_{0}(s)f\|_{\X_{0}}\d s$$
since $U_{0}(t)$ is a contraction in $\X_{0}$. Then, one deduces
$$\|V_{1}(t)f\|_{\X_{0}} \leq \|\K_{1}\|_{\mathscr{B}(\X_{-1},\X_{0})}\int_{0}^{t}\|U_{0}(s)f\|_{\X_{-1}}\d s.$$
Since $U_{0}(s)$ commutes with $\varpi_{-1}(v)=\min\left(1,\sigma(v)\right)$, one has
$$\|U_{0}(s)f\|_{\X_{-1}}=\int_{\T^{d}\times V}\varpi_{-1}(v)e^{-s\sigma(v)}|f(x,v)|\d x\,\bm{m}(\d v)$$
and, since $\varpi_{-1} \leq \sigma$,
\begin{multline*}
\int_{0}^{t}\|U_{0}(s)f\|_{\X_{-1}}\d s \leq \int_{\T^{d}\times V}|f(x,v)|\d x\,\bm{m}(\d v)\int_{0}^{t}\sigma(v)e^{-s\sigma(v)}\d s\\
=\int_{\T^{d}\times V}|f(x,v)|\left(1-e^{-t\sigma(v)}\right)\d x\,\bm{m}(\d v)
\leq \|f\|_{\X_{0}}\end{multline*}
This proves \eqref{eq:Vnt} for $n=1$. Assume then the result to be true for $n \geq 1$ and let us prove for $n+1$. One has, as before,
\begin{equation*}\begin{split}
\left\|{V}_{n+1}(t)f\right\|_{\X_{0}}&\leq \sup_{s\in [0,t]}\|V_{n}(t-s)\|_{\mathscr{B}(\X_{0})}\int_{0}^{t}\|\K_{n+1}U_{0}(s)f\|_{\X_{0}}\d s\\
&\leq \sup_{s\in [0,t]}\|V_{n}(t-s)\|_{\mathscr{B}(\X_{0})}\|\K_{n+1}\|_{\mathscr{B}(\X_{-1},\X_{0})}\int_{0}^{t}\|U_{0}(s)f\|_{\X_{-1}}\d s
\end{split}\end{equation*}
We saw how to estimate this last integral and, with the induction hypothesis $\sup_{s\in [0,t]}\|V_{n}(t-s)\|_{\mathscr{B}(\X_{0})} \leq \prod_{j=1}^{n}\|\K_{j}\|_{\mathscr{B}(\X_{-1},\X_{0})}$, we deduce that
$$\|V_{n+1}(t)f\|_{\X_{0}}\leq \prod_{j=1}^{n+1}\|\K_{j}\|_{\mathscr{B}(\X_{-1},\X_{0})}\|f\|_{\X_{0}}$$
which proves the result.
\end{proof}

 \subsection{Decay of the iterates}\label{sec:decay} We extend the decay of the semigroup $\left(U_{0}(t))\right)_{t\geq0}$ obtained in Lemma \ref{lem:decayU0} to the iterates 
$\left(U_{k}(t)\right)_{t\geq0}$. We recall that, for any $\delta >0$, we introduce
$$\Lambda_{\delta}:=\{v \in V\;;\;\sigma(v) \geq \delta\}, \qquad \Sigma_{\delta}=V \setminus \Lambda_{\delta}$$
and 
%so that
%$$\|U_{0}(t)\ind_{\Lambda_{\delta}}f\|_{\X_{0}} \leq e^{-t\delta}\|\ind_{\Lambda_{\delta}}f\|_{\X_{0}}\leq e^{-t\delta}\|f\|_{\X_{0}} \qquad \forall f \in \X_{0},\qquad t\geq 0.$$ 
%Let us now introducing for any $\delta >0$, the operator 
$\K^{(\delta)} \in \mathscr{B}(\X_{0})$ given by \eqref{defi:Hepsi}, i.e.
\begin{equation*}
\K^{(\delta)}f(x,v)=\ind_{\Lambda_{\delta}}\K f(x,v) \qquad \forall f \in \X_{0}, \quad (x,v) \in \T^{d}\times V\end{equation*}
We prove here Lemma \ref{lem:sizeDelta} which investigate the size of $\|\K-\K^{(\delta)}\|_{\B(\X_{-1},\X_{0})}.$
\begin{proof}[Proof of Lemma \ref{lem:sizeDelta}]
Introduce 
\begin{equation}\label{def:Kdel}
\overline{\K}^{(\delta)}=\K-\K^{(\delta)}\end{equation}
so that
$$\overline{\K}^{(\delta)}f(x,v)=\ind_{\Sigma_{\delta}}(v)\K f(x,v).$$
Then, for any $f \in \X_{-1}$, one has, for any $n \geq0$,
\begin{equation*}\begin{split}
\|\overline{\K}^{(\delta)}f\|_{\X_{0}}&=\int_{\T^{d}}\d x\int_{\Sigma_{\delta}}|\K f(x,v)|\bm{m}(\d v)\\
&\leq \int_{\T^{d}}\d x\int_{V}\ind_{\Sigma_{\delta}}(v)\bm{m}(\d v)\int_{V}\bm{k}(v,w)|f(x,w)|\bm{m}(\d w)\\
&\leq \int_{\T^{d} \times V}|f(x,w)|\d x\,\bm{m}(\d w)\int_{V}\ind_{\Sigma_{\delta}}(v)\bm{k}(v,w)\frac{\sigma^{n}(v)}{\sigma^{n}(v)}\bm{m}(\d v)\\
&\leq \delta^{n}\int_{\T^{d}\times V}|f(x,w)|\d x\,\bm{m}(\d w)\int_{V}\sigma^{-n}(v)\bm{k}(v,w)\bm{m}(\d v).
\end{split}\end{equation*}
By definition of $\vartheta_{n}$, we deduce that
$$\|\overline{\K}^{(\delta)}f\|_{\X_{0}} \leq \delta^{n}\int_{\T^{d}\times V}\sigma(w)\vartheta_{n}(w)\,|f(x,w)|\d x\,\bm{m}(\d w)$$

i.e.
$$\|\overline{\K}^{(\delta)}f\|_{\X_{0}} \leq \delta^{n}\|\vartheta_{n}\|_{\infty}\int_{\T^{d}\times V}\sigma(w)\,|f(x,w)|\d x\,\bm{m}(\d w) \leq \delta^{n}\|\vartheta_{n}\|_{\infty}\,\|f\|_{\X_{-1}}$$
as soon as $\vartheta_{n} \in L^{\infty}(V)$. This gives \eqref{eq:overKdelta}.\end{proof}
Before proving the precise decay of $U_{k}(t)$ for any $k \in \N$, for the clarity of exposition, we give full details for the decay of $U_{1}(t)$. 
\begin{lemme}\label{lem:decayU1}
Let $f \in \X_{N_{0}}$. Then, there exists some universal constant $C_{1} >0$ (depending only on $\K$, $N_{0}$  but not  on $f$) such that
$$\left\|U_{1}(t)f\right\|_{\X_{0}} \leq  C_{1}\left(\frac{\log t}{t}\right)^{N_{0}}\,\|f\|_{\X_{N_{0}}}, \qquad \forall t >0.$$
\end{lemme}
\begin{proof} The proof is based upon the decomposition of $\K$ for small and large collision frequency. Namely, for some $\delta >0$ to be determined, we use the above splitting 
$$\K=\K^{(\delta)}+\overline{\K}^{(\delta)}$$
where $\K^{(\delta)}$ is defined in \eqref{defi:Hepsi} and 
$${U}_{1}(t)=U_{1}^{(\delta)}(t)+\overline{U}_{1}^{(\delta)}(t), \qquad t >0, \qquad \delta >0$$
where $U_{1}^{(\delta)}(t),\overline{U}_{1}^{(\delta)}(t)$ are  given by
$$U_{1}^{(\delta)}(t)=\int_{0}^{t}U_{0}(t-s)\K^{(\delta)}U_{0}(s)\d s, \qquad \overline{U}_{1}^{(\delta)}(t)=\int_{0}^{t}U_{0}(t-s)\overline{\K}^{(\delta)}U_{0}(s)\d s$$
Let now fix $k\geq 1,$ $f \in \X_{k}, t >0.$ 
\begin{equation*}
\|U_{1}(t)f\|_{\X_{0}} \leq \|U_{1}^{(\delta)}(t)f\|_{\X_{0}}+\|\overline{U}_{1}^{(\delta)}(t)f\|_{\X_{0}} 
\leq \|{U}_{1}^{(\delta)}(t)f\|_{\X_{0}} + \|\overline{\K}^{(\delta)}\|_{\mathscr{B}(\X_{-1},\X_{0})}\|f\|_{\X_{0}} \end{equation*}
where we used \eqref{eq:Vnt}. Using now \eqref{eq:overKdelta},
\begin{equation}\label{eq:U1tep}
\|U_{1}(t)f\|_{\X_{0}} \leq \delta^{N_{0}}\|\vartheta_{N_{0}}\|_{\infty}\|f\|_{\X_{0}}+\|{U}_{1}^{(\delta)}(t)f\|_{\X_{0}} \qquad \forall \delta >0, \qquad t \geq0.\end{equation}
Let us focus then on the estimate for $\ {U_{1}^{(\delta)}}(t)f$. Given $s \in [0,t]$, one computes easily
\begin{multline*}
U_{0}(t-s)\K^{(\delta)}U_{0}(s)f(x,v)=\\
\ind_{\Lambda_{\delta}}(v)\int_{V}\bm{k}(v,w)\exp\left(-(t-s)\sigma(v)-s\sigma(w)\right)f(x-tv+s(v-w),w)\,
\bm{m}(\d w)\end{multline*}
so that
\begin{multline*}
\|U_{1}^{(\delta)}(t)f\|_{\X_{0}}\leq \int_{0}^{t}\d s\int_{\T^{d}\times \Lambda_{\delta}}\d x \bm{m}(\d v)\\
\int_{V}\bm{k}(v,w)\exp\left(-(t-s)\sigma(v)-s\sigma(w)\right)\left|f(x-tv+s(v-w),w)\right|\,\bm{m}(\d w)\\
\leq \int_{\T^{d}\times V}|f(y,w)|\bm{m}(\d w)\int_{0}^{t}\d s\int_{\Lambda_{\delta}}\bm{k}(v,w)\exp\left(-(t-s)\sigma(v)-s\sigma(w)\right)\bm{m}(\d v).
\end{multline*}
Introducing again
$$g(x,v)=\sigma(v)^{-k}f(x,v)$$
and using Fubini's Theorem, we get
\begin{equation}\label{eq:normU1delta}
\|U_{1}^{(\delta)}(t)f\|_{\X_{0}} \leq \int_{\T^{d}\times V}\sigma^{k}(w)\Theta_{1}^{(\delta)}(t,w)|g(y,w)|\d y\,\bm{m}(\d w)\end{equation}
where
\begin{equation*}\begin{split}
\Theta_{1}^{(\delta)}(t,w)&=\int_{\Lambda_{\delta}}\bm{k}(v,w)\bm{m}(\d v)\int_{0}^{t}\exp\left(-(t-s)\sigma(v)-s\sigma(w)\right)\d s\\
&=\int_{\Lambda_{\delta}}\exp\left(-t\sigma(v)\right)\bm{k}(v,w)\frac{\exp\left(t\left[\sigma(v)-\sigma(w)\right]\right)-1}{\sigma(v)-\sigma(w)}\bm{m}(\d v)
\end{split}\end{equation*}
(where we notice that, if $\sigma(v)=\sigma(w)$ the last quotient is equal to $1$). Notice that, since $\exp\left(-(t-s)\sigma(v)-s\sigma(w)\right) \leq \exp(-s\sigma(w))$ for any $v,w$ and any $s \in [0,t]$ we have, 
$$\Theta_{1}^{(\delta)}(t,w) \leq \int_{\Lambda_{\delta}}\bm{k}(v,w)\bm{m}(\d v)\int_{0}^{t}\exp\left(-s\sigma(w)\right)\d s \leq \frac{1}{\sigma(w)}\int_{V}\bm{k}(v,w)\bm{m}(\d v) =1$$
for any $t\geq0$, $w \in V.$

To estimate the integral in the right-hand-side of \eqref{eq:normU1delta}, we estimate $\Theta_{1}^{\delta}(t,w)$ distinguishing between the two cases 
$$w \in \Lambda_{\frac{\delta}{2}} \quad \text{ or  } \quad w \notin \Lambda_{\frac{\delta}{2}}.$$ Assume first that $w \in \Lambda_{\frac{\delta}{2}}$, i.e. $\sigma(w) \geq \frac{\delta}{2}$. Then, for any $v \in V$, $\sigma(v)-\sigma(w) \leq \sigma(v)-\frac{\delta}{2}$ and, since the mapping
$$u \in \R \longmapsto \frac{e^{u}-1}{u}$$
is nondecreasing, we get that
$$\frac{\exp\left(t\left[\sigma(v)-\sigma(w)\right]\right)-1}{\sigma(v)-\sigma(w)} \leq \frac{\exp\left(t\left[\sigma(v)-\frac{\delta}{2}\right] \right)-1}{\sigma(v)-\frac{\delta}{2}}.$$
i.e.
$$\Theta_{1}^{(\delta)}(t,w) \leq \int_{\Lambda_{\delta}}\bm{k}(v,w)\frac{\exp\left(-t\frac{\delta}{2}\right)-\exp\left(-t\sigma(v)\right)}{\sigma(v)-\frac{\delta}{2}}\bm{m}(\d v), \qquad w \in \Lambda_{\frac{\delta}{2}}.$$
Now, for $v \in \Lambda_{\delta}$, 
$$\sigma(v)-\frac{\delta}{2} \geq \frac{\delta}{2} \qquad \text{ and } \qquad \frac{\exp\left(-t\frac{\delta}{2}\right)-\exp\left(-t\sigma(v)\right)}{\sigma(v)-\frac{\delta}{2}} \leq \frac{2}{\delta}\exp\left(-t\frac{\delta}{2}\right)$$
from which we get
\begin{equation}\label{eq:away}
\Theta_{1}^{(\delta)}(t,w) \leq \frac{2\sigma(w)}{\delta}\exp\left(-\frac{\delta}{2}\right), \qquad w \in \Lambda_{\frac{\delta}{2}}\end{equation}
where we used that $\int_{\Lambda_{\delta}} \bm{k}(v,w)\bm{m}(\d v) \leq \int_{V}\bm{k}(v,w)\bm{m}(\d v)=\sigma(w)$. Inserting this into \eqref{eq:normU1delta} we deduce
\begin{multline*}
\|U_{1}^{(\delta)}(t)f\|_{\X_{0}} \leq \frac{2}{\delta}\exp\left(-t\frac{\delta}{2}\right)\int_{\T^{d}\times V}\sigma^{k+1}(w)|g(y,w)|\d y\,\bm{m}(\d w) \\
+ \int_{\T^{d}\times \Lambda_{\frac{\delta}{2}}^{c}}\sigma^{k}(w)\Theta_{1}^{(\delta)}(t,w)\left|g(y,w)\right|\d y\,\bm{m}(\d w).\end{multline*}
For $w \notin \Lambda_{\frac{\delta}{2}}$, i.e. $\sigma(w) \leq \frac{\delta}{2}$, we simply recall that $\Theta_{1}^{(\delta)}(t,w) \leq 1$ and of course $\sigma^{k}(w) \leq \left(\frac{\delta}{2}\right)^{k}$ so that 
$$\int_{\T^{d}\times \Lambda_{\frac{\delta}{2}}^{c}}\sigma^{k}(w)\Theta_{1}^{(\delta)}(t,w)\left|g(y,w)\right|\d y\,\bm{m}(\d w) \leq \left(\frac{\delta}{2}\right)^{k}\|g\|_{\X_{0}}$$
which, combined with the previous estimate gives
$$\|U_{1}^{(\delta)}(t)f\|_{\X_{0}} \leq \frac{2}{\delta}\exp\left(-t\frac{\delta}{2}\right)\|\sigma^{k+1}g\|_{\X_{0}} + \left(\frac{\delta}{2}\right)^{k}\|g\|_{\X_{0}} \qquad \forall \delta >0.$$
Adding this to \eqref{eq:U1tep} and with $k=N_{0}$, we get 
$$\|U_{1}(t)f\|_{\X_{0}} \leq \frac{2}{\delta}\exp\left(-t\frac{\delta}{2}\right)\|\sigma^{k+1}g\|_{\X_{0}} + \left(\frac{\delta}{2}\right)^{N_{0}}\|g\|_{\X_{0}}+ \delta^{N_{0}}\|\vartheta_{N_{0}}\|_{\infty}\|f\|_{\X_{0}}, \qquad \delta >0.$$
Observing that  
$$\|\sigma^{k+1}g\|_{\X_{0}}+\|g\|_{\X_{0}} + \|f\|_{\X_{0}}=\|\sigma\,f\|_{\X_{0}} + \|\sigma^{-k}f\|_{\X_{0}} +\|f\|_{\X_{0}} \leq \left(1+\|\sigma\|_{\infty}\right)\|f\|_{\X_{k}}$$
and there is a positive constant $C >0$ such that
$$\|U_{1}(t)f\|_{\X_{0}} \leq C \left(\frac{2}{\delta}\exp\left(-t\frac{\delta}{2}\right) + \left(\frac{\delta}{2}\right)^{N_{0}}\right)\|f\|_{\X_{N_{0}}}, \qquad \forall \delta >0.$$
Choosing, for $t > e$, 
$$\delta=2(N_{0}+1)\frac{\log t}{t}$$
we deduce that there is $C_{0} >0$ such that
\begin{equation*}\begin{split}
\|U_{1}(t)f\|_{\X_{0}} &\leq C\left(\frac{1}{(N_{0}+1)t^{N_{0}}\log t} + (N_{0}+1)^{N_{0}}\left(\frac{\log t}{t}\right)^{N_{0}}\right)\|f\|_{\X_{N_{0}}}	\\
&\leq C_0 \left(\frac{\log t}{t}\right)^{N_{0}}\|f\|_{\X_{N_{0}}}, \qquad t > e\end{split}\end{equation*}
which gives the result.
\end{proof}
We generalise this approach to the other iterates

\begin{lemme}\label{lem:decayUn}
Let $f \in \X_{N_{0}}$. For any $n \geq 1$ there exists some universal constant $C_{n} >0$ (depending only on $\K$, $n,N_{0}$  but not  on $f$) such that
$$\left\|U_{n}(t)f\right\|_{\X_{0}} \leq  C_{n}\left(\frac{\log t}{t}\right)^{N_{0}}\,\|f\|_{\X_{N_{0}}}, \qquad \forall t >0.$$
\end{lemme}  
\begin{proof} The proof uses the same ideas introduced in the proof for $n=1$. Given $\delta >0$, we still introduce the splitting $\K=\K^{(\delta)}
+\overline{\K}^{(\delta)}$ which gives, by a simple combinatorial argument that, for $n \geq 1$, one can write
$$U_{n}(t)=U_{n}^{(\delta)}(t)+\overline{U}_{n}^{(\delta)}(t)$$
where $U_{n}^{(\delta)}(t)$ is constructed as a Dyson-Phillips iterated involving \emph{only} the operator $\K^{(\delta)}$ whereas the reminder term $\overline{U}_{n}^{(\delta)}(t)$ is the some of $2^{n}-1$ operators 
$$\overline{U}_{n}^{(\delta)}(t)=\sum_{j=1}^{2^{n}-1}\bm{V}_{n}^{(j)}(t)$$
where, for any $j \in \{1,\ldots,2^{n}-1\}$, $\bm{V}_{n}^{(j)}(t)$ is defined by \eqref{eq:VnKn}  for a (finite) family of operators $(\K_{1},\ldots,\K_{n})$ where there is at least one $i \in \{1,\ldots,n\}$ such that $\K_{i}=\overline{\K}^{(\delta)}$ (the other ones being indifferently $\K^{(\delta)}$ or $\overline{\K}^{(\delta)}$).  Using Proposition \ref{prop:Vnt}, and recalling that $\|\K\|_{\mathscr{B}(\X_{-1},\X_{0})} \leq 1$ one has then 
$$\|\bm{V}_{n}^{(j)}(t)\|_{\mathscr{B}(\X_{0})} \leq \|\overline{\K}^{(\delta)}\|_{\mathscr{B}(\X_{-1},\X_{0})}.$$
Therefore
$$\left\|\overline{U}_{n}^{(\delta)}(t)f\right\|_{\X_{0}} \leq \left(2^{n}-1\right)\|\overline{\K}^{(\delta)}\|_{\mathscr{B}(\X_{-1},\X_{0})}\|f\|_{\X_{0}}$$
and
\begin{equation}\label{eq:Untep}
\left\|U_{n}(t)f\right\|_{\X_{0}}\leq \|U_{n}^{(\delta)}(t)f\|_{\X_{0}}+ C(2^{n}-1)\,\delta^{N_{0}}\|f\|_{\X_{0}}, \qquad t >0, \qquad \delta >0\end{equation}
where we used \eqref{eq:overKdelta}
. We focus now on the expression of $U_{n}^{(\delta)}(t)f$. From the subsequent Lemma, we have
\begin{equation}\label{eq:UnDe}
\|U_{n}^{(\delta)}(t)f\|_{\X_{0}} \leq \int_{\T^{d}\times V}\Theta_{n}^{(\delta)}(t,w)\,|f(y,w)|\d y\,\bm{m}(\d w)\end{equation}
where $\Theta_{n}^{(\delta)}(t,w)$ is defined by 
\begin{multline}\label{eq:thetadN}
\Theta_{n}^{(\delta)}(t,w)=\int_{\Lambda_{\delta}}\prod_{j=1}^{n}\bm{k}(v_{j},w)\bm{m}(\d v_{1})\ldots\bm{m}(\d v_{n})\\
\int_{\Delta_t}\prod_{j=1}^{n}\exp\left(-(s_{j-1}-s_{j})\sigma(v_{j})\right)\exp\left(-s_{n}\sigma(w)\right)\d s_{1}\ldots \d s_{n}.\end{multline}
where we used the convention $s_0=t$ and $\Delta_t$ denotes the simplex 
$$\Delta_t=\left\{(s_1,\ldots,s_n) \in \R^{n}_{+}\;;\;0 \leq s_{n} \leq s_{n-1} \leq \ldots s_{1} \leq t\right\}.$$
We saw in Lemma \ref{lem:decayU1} that
$$\Theta_{1}^{(\delta)}(s,w) \leq \dfrac{2\sigma(w)}{\delta}\exp\left(-s\frac{\delta}{2}\right) \qquad \text{ for } w \in \Lambda_{\frac{\delta}{2}}$$
and $\Theta_{1}^{(\delta)}(s,w) \leq 1$ for any $s \geq0,w\in V.$ Let us prove a similar estimate holds true for any $n\in \N.$ The fact that
\begin{equation}\label{eq:The11}
\Theta_{n}^{(\delta)}(t,w) \leq  1 \qquad \text{ for any $t\geq0$ and any $w \in V$}\end{equation} is easily seen by induction since
$$\Theta_{n}^{(\delta)}(t,w)=\int_{\Lambda_{\delta}}\bm{k}(v,w)\bm{m}(\d v)\int_{0}^{t}\exp\left(-(t-s)\sigma(v)\right)\Theta_{n-1}^{(\delta)}(s,w)\d s.$$
Let now $w \in \Lambda_{\frac{\delta}{2}}$ be fixed. We estimate, for $(v_{1},\ldots,v_{n}) \in \Lambda_{\delta}^{n}$ the integral on the simplex $\Delta_t$ as in the previous Lemma starting from the integral with respect to $s_n$. For fixed $(s_{1},\ldots,s_{n-1})$ we have, as in Lemma \ref{lem:decayU1}
\begin{equation*}\begin{split}
G_1(s_{n-1},v_n,w)&:=\int_{0}^{s_{n-1}}\exp\left(-(s_{n-1}-s_{n})\sigma(v_{n})-s_{n}\sigma(w)\right)\d s_{n} \\
%=\exp\left(-s_{n-1}\sigma(v_{n})\right)\int_{0}^{s_{n-1}}\exp\left(s_{n}\left(\sigma(v_{n})-\sigma(w)\right)\right)\d s_{n}\\
&\leq \exp\left(-s_{n-1}\sigma(v_{n})\right)\int_{0}^{s_{n-1}}\exp\left(s_{n}\left(\sigma(v_{n})-\frac{\delta}{2}\right)\right)\d s_{n}\\
&\leq \frac{1}{\sigma(v_{n})-\frac{\delta}{2}}\exp\left(-s_{n-1}\frac{\delta}{2}\right).\end{split}\end{equation*}
Since $\sigma(v_{n}) \geq \delta$ we get $\sigma(v_{n}) -\frac{\delta}{2} \geq \frac{1}{2}\sigma(v_{n})$ and
$$\int_{0}^{s_{n-1}}\exp\left(-(s_{n-1}-s_{n})\sigma(v_{n})-s_{n}\sigma(w)\right)\d s_{n} \leq \frac{2}{\sigma(v_{n})} \exp\left(-s_{n-1}\frac{\delta}{2}\right).$$
Now, for $(s_1,\ldots,s_{n-2})$ given, we multiply this by $\exp\left(-(s_{n-2}-s_{n-1})\sigma(v_{n-1})\right)$ and integrate with respect to $s_{n-1}$ to get
\begin{equation*}\begin{split}
G_{2}(s_{n-2},v_{n-1},v_{n},w)&:=\int_{0}^{s_{n-2}}\exp\left(-(s_{n-2}-s_{n-1})\sigma(v_{n-1})\right)G(s_{n-1},v_{n},w)\d s_{n-1}\\
&\leq \frac{2}{\sigma(v_{n})}\exp\left(-s_{n-2}\sigma(v_{n-1})\right)\int_{0}^{s_{n-2}}\exp\left(s_{n-1}\left(\sigma(v_{n-1})-\frac{\delta}{2}\right)\right)\d s_{n-1}\\
&\leq \frac{2}{\sigma(v_{n})}\exp\left(-s_{n-2}\sigma(v_{n-1})\right)\frac{\exp\left(s_{n-2}\left(\sigma(v_{n-1})-\frac{\delta}{2}\right)\right)-1}{\sigma(v_{n-1})-\frac{\delta}{2}}
\end{split}\end{equation*}
and, as before,
$$G_{2}(s_{n-2},v_{n-1},v_{n},w) \leq \frac{2}{\sigma(v_{n})}\frac{2}{\sigma(v_{n-1})}\exp\left(-s_{n-2}\frac{\delta}{2}\right)$$
Iterating this process (recalling that $s_{0}=t$, we end up with the following estimate for the integral over the simplex:
\begin{multline*}
\int_{\Delta_{t}}\prod_{j=1}^{n}\exp\left(-(s_{j-1}-s_{j})\sigma(v_{j})\right)\exp\left(-s_{n}\sigma(w)\right)\d s_{1}\ldots \d s_{n}\\
\leq \frac{2^{n}}{\prod_{j=1}^{n}\sigma(v_{j})}\exp\left(-t\frac{\delta}{2}\right) \quad \forall (v_1,\ldots,v_{n}) \in \Lambda_{\delta}^{n}, \quad w \in \Lambda_{\frac{\delta}{2}}.
\end{multline*}
Inserting this into \eqref{eq:thetadN} yields
$$
\Theta_{n}^{(\delta)}(t,w) \leq 2^n\exp\left(-t\frac{\delta}{2}\right)\int_{\Lambda_{\delta}}\prod_{j=1}^{n}\frac{\bm{k}(v_{j},w)}{\sigma(v_{j})}\bm{m}(\d v_{1})\ldots\bm{m}(\d v_{n})$$
and, using that 
$$\sigma(v_{j}) \geq \delta \qquad \text{ whereas } \int_{\Lambda_{\delta}}\bm{k}(v_{j},w)\bm{m}(\d v_{j}) \leq \int_{V}\bm{k}(v_{j},w)\bm{m}(\d v_{j})=\sigma(w)$$
for any $j \in \{1,\ldots,n\}$ we obtain
\begin{equation}\label{eq:EstmThe}
\Theta_{n}^{(\delta)}(t,w) \leq \left(\frac{2\sigma(w)}{\delta}\right)^{n}\exp\left(-t\frac{\delta}{2}\right)\,,\qquad \qquad \forall t \geq0, \delta >0, \qquad w \in \Lambda_{\frac{\delta}{2}}.\end{equation}
Inserting this, together with \eqref{eq:The11}, into \eqref{eq:UnDe} gives
\begin{multline*}
\|U_{n}^{(\delta)}(t)f\|_{\X_{0}} \leq \int_{V}\ind_{\Sigma_{\frac{\delta}{2}}}|f(y,w)|\d y\,\bm{m}(\d w)\\
 + \left(\frac{2}{\delta}\right)^{n}\exp\left(-t\frac{\delta}{2}\right)\int_{V}\sigma(w)^{n}|f(y,w)|\d y\,\bm{m}(\d w)\end{multline*}
Introducing as before $g(x,v)=\sigma^{-k}(v)f(x,v)$, we get
$$\|U_{n}^{(\delta)}(t)f\|_{\X_{0}} \leq \left(\frac{\delta}{2}\right)^{k}\|g\|_{\X_{0}} +  \left(\frac{2\|\sigma\|_{\infty}}{\delta}\right)^{n}\exp\left(-t\frac{\delta}{2}\right)\|f\|_{\X_{0}}$$
and, for $k=N_{0}$, using \eqref{eq:Untep} we end up with
\begin{multline*}
\|U_{n}(t)f\|_{\X_{0}} \leq C(2^{n}-1)\,\delta^{N_{0}}\|f\|_{\X_{0}} +  \left(\frac{\delta}{2}\right)^{N_{0}}\|g\|_{\X_{0}}  + \left(\frac{2 \|\sigma\|_{\infty}}{\delta}\right)^{n}\exp\left(-t\frac{\delta}{2}\right)\|f\|_{\X_{0}}\\
\leq C_{1}\|f\|_{\X_{N_{0}}}\left(\left(\frac{\delta}{2}\right)^{N_{0}} +  \left(\frac{2}{\delta}\right)^{n}\exp\left(-t\frac{\delta}{2}\right)\right), \qquad \forall \delta >0.
\end{multline*}
where $C_{1}$ depends on $n$, $N_{0}$ and $\|\sigma\|_{\infty}$ but not on $\delta$. Picking now, for $t >e$, 
$$\delta=2(n+N_{0})\frac{\log t}{t}$$
one gets easily the result as in Lemma \ref{lem:decayU1}.
\end{proof}
We establish here the general estimate allowing the definition of $\Theta_{n}^{(\delta)}(t,w)$ that has been used in the previous Lemma. We give a more general result which applies to general integral operator $\K$ and not only to $\K^{(\delta)}$. The version for $\K^{(\delta)}$ being deduced obviously by observing that the kernel of $\K^{(\delta)}$ is $\ind_{\Lambda_{\delta}}\bm{k}(v,w)$.
\begin{lemme}
For 
$$\K f(x,v)=\int_{V}\bm{k}(v,w)f(x,w)\d w$$
the norm of the Dyson-Phillips iterated $\left(U_{n}(t)\right)_{n\geq 0}$ is such that
$$\left\|U_{n}(t)f\right\|_{\X_{0}} \leq \int_{\T^{d} \times V}\Theta_{n}(t,w)\,|f(y,w)|\d y \bm{m}(\d w), \qquad \forall t \geq 0, \qquad n \geq 0$$
for any $f \in \X_{0}$ where
$$\Theta_{n}(t,w)=\int_{V} \bm{k}(v,w)\bm{m}(\d v)\int_{0}^{t} \exp\left(-(t-s)\sigma(v)\right)\Theta_{n-1}(s,w)\d s, \qquad n \geq 1$$
and $\Theta_{0}(s,w)=\exp\left(-s\sigma(w)\right).$ In particular,
\begin{multline}\label{eq.thetaN}
\Theta_{n}(t,w)=\int_{V}\prod_{j=1}^{n}\bm{k}(v_{j},w)\bm{m}(\d v_{1})\ldots\bm{m}(\d v_{n})\\
\int_{\Delta_t}\prod_{j=1}^{n}\exp\left(-(s_{j-1}-s_{j})\sigma(v_{j})\right)\exp\left(-s_{n}\sigma(w)\right)\d s_{1}\ldots \d s_{n}.\end{multline}
where we used the convention $s_0=t$ and $\Delta_t$ denotes the simplex 
$$\Delta_t=\left\{(s_1,\ldots,s_n) \in \R^{n}_{+}\;;\;0 \leq s_{n} \leq s_{n-1} \leq \ldots s_{1} \leq t\right\}.$$
\end{lemme}
\begin{proof} We prove the result by induction. For $n=0$, the result is obvious. Assume the result to be true for some $n \geq 0$. Let us fix $t\geq0,$ $f\in \X_{0}$. Then, from
$$U_{n+1}(t)f=\int_{0}^t U_{0}(t-s)\K U_{n}(s)f\d s$$
one has, introducing $h_n(s,x,v)=U_{n}(s)f(x,v)$, 
$$U_{n+1}(t)f(x,v)=\int_{0}^t e^{-(t-s)\sigma(v)}\d s\int_V \bm{k}(v,w)h_{n}(s,x-(t-s)v,w)\bm{m}(\d w)$$
and, introducing the change of variable $y=x-(t-s)v$, we get easily
\begin{multline*}
\|U_{n+1}f\|_{\X_{0}} \leq \int_{\T^{d}\times V}\d y\,\bm{m}(\d w)\int_0^t e^{-(t-s)\sigma(v)}\d s\int_{V} \bm{k}(v,w)|h_{n}(s,y,w)|\bm{m}(d \v)\\
\leq \int_{V} \bm{k}(v,w)\bm{m}(\d v)\int_{0}^{t}\exp\left(-(t-s)\sigma(v)\right)\d s\int_{\T^{d}\times V}|h_{n}(s,y,w)|\d y\,\bm{m}(\d w)\\
\leq \int_{V} \bm{k}(v,w)\bm{m}(\d v)\int_{0}^{t}\exp\left(-(t-s)\sigma(v)\right)\|U_n(s)f\|_{\X_{0}}\d s\end{multline*}
which, thanks to the induction hypothesis, gives
$$\|U_{n+1}(t)f\|_{\X_{0}}\leq \int_{\T^{d}\times V}|f(y,w)|\d y \bm{m}(\d w)\int_{V}\bm{k}(v,w)\bm{m}(\d v)\int_{0}^{t}\exp\left(-(t-s)\sigma(v)\right)\Theta_{n}(s,w)\d s$$
from which the desired estimate easily follows. This achieves the proof by induction. One shows then by direct inspection that
\begin{multline*}
\Theta_{n}(t,w)=\int_{V}\bm{k}(v_{1},w)\bm{m}(\d v_{1})\int_{V}\bm{k}(v_{2},w)\bm{m}(\d v_{2})\ldots \int_{V}\bm{k}(v_{n},w)\bm{m}(\d v_{n})\\
\int_{0}^{t}\exp\left(-(t-s_{1})\sigma(v)\right)\d s_{1}\int_{0}^{s_{1}}\exp\left(-(s_{1}-s_{2})\sigma(v_{1})\right)\d s_{2}\\
\times \ldots \int_{0}^{s_{n}}\exp\left(-(s_{n-1}-s_{n})\sigma(v_{n})-s_{n}\sigma(w)\right)\d s_{n}\end{multline*}
which gives \eqref{eq.thetaN}.
\end{proof}
Combining Lemmas \ref{lem:decayU0} and \ref{lem:decayUn} one deduces easily Proposition \ref{prop:Snt}

\subsection{Inverse Laplace transform}

We establish here a somehow classical result regarding the inverse Laplace transform where we recall that $\bm{S}_{n+1}(t)$ has been defined in \eqref{eq:Sn+1}.

\begin{propo} \label{propo:Sn+1}
Under Assumption \eqref{eq:power}, for any $n\geq 5$ and any $f \in \X_{0}$,
one has
\begin{equation*}\label{eq:maindecay}
\bm{S}_{n+1}(t)f=\frac{\exp(\e t)}{2\pi}\lim_{\ell\to\infty} \int_{-\ell}^{\ell}\exp\left(i\eta t\right)\mathscr{S}_{n+1}(\e+i\eta)f\d\eta, \qquad \forall f \in \X_{0}\end{equation*}
for any $t >0$, $\e >0$ where
\begin{equation}\label{eq:Snl}
\mathscr{S}_{n+1}(\l)f:=\sum_{k=n}^{\infty}\Rs(\l,\A)\left[\K\Rs(\l,\A)\right]^{k+1}  f, \qquad \forall \mathrm{Re}\l >0, f \in \X_{0}.\end{equation}
%
%
%
%\begin{equation}\label{eq:maindecay}\begin{split}
%\bm{S}_{n+1}(t)f&=\lim_{\ell\to\infty}\frac{\exp(\e t)}{2\pi}\int_{-\ell}^{\ell}\exp\left(i\eta t\right)\Xi_{\e+i\eta}\H\left(\Ms_{\e+i\eta}\right)^{n}\Rs\left(1,\Ms%_{\e+i\eta}\H\right)\mathsf{G}_{\e+i\eta}f\d\eta\\
%&=\frac{\exp(\e t)}{2\pi}\int_{-\infty}^{\infty}\exp\left(i\eta t\right)\Xi_{\e+i\eta}\H\left(\Ms_{\e+i\eta}\right)^{n}\Rs\left(1,\Ms_{\e+i\eta}\right)\mathsf{G}_{\e+i\eta}f\d\eta\,,\end{split}
%\end{equation}
%for any $t >0,\e >0.$
%\begin{multline}\label{eq:maindecay}
%\left\|U_{\H}(t)f - \frac{\exp(\e t)}{2\pi}\int_{-\infty}^{\infty}\exp\left(i\eta t\right)\Xi_{\e+i\eta}\H\left(\Ms_{\e+i\eta}\right)^{n}\Rs\left(1,\Ms_{\e+i\eta}\right)\mathsf{G}_{\e+i\eta}f\d\eta\,\right\|_{\X_{0}} \\
%\leq \bm{C}_{n}\,t^{-\left(N_{0}+1\right)}\,\left\|f\right\|_{\X_{N_{0}+1}}, \qquad t >0, \qquad \e >0.\end{multline}
\end{propo}
 
\begin{proof} The fact that the Laplace transform of $\bm{S}_{n+1}(t)f$ is exactly $\mathscr{S}_{n}(\e+i\eta)f$ is a well-known fact and the complex Laplace inversion formula  \cite[Theorem 4.2.21]{arendt} allows to directly deduce that  $\bm{S}_{n+1}(t)f$ is the Cesar\`o limit of the family 
$$\left(\int_{-\ell}^{\ell}\exp\left(i\eta t\right)\mathscr{S}_{n+1}(\e+i\eta)f\d\eta\right)_{\ell}.$$
One has to work a little bit more to deduce that it is a \emph{classical limit}.
 We recall that
$$\int_{0}^{\infty}\exp\left(-\l t\right)U_{n}(t)\d t=\left[\Rs(\l,\A)\K\right]^{n}\Rs(\l,\A)=\Rs(\l,\A)\left[\K\Rs(\l,\A)\right]^{n}.$$
With this in mind, for any $n \geq0$ and any $f \in \X_{0}$, it holds
\begin{equation*}\begin{split}
\int_{0}^{\infty}\exp\left(-\l t\right)\bm{S}_{n+1}(t)f\d t&=\sum_{k=n}^{\infty}\int_{0}^{\infty}\exp\left(-\l t\right)U_{k+1}(t)f\d t\\
&=\Rs(\l,\A)\sum_{k=n}^{\infty}\left[\K\Rs(\l,\A)\right]^{k+1} f=:\mathscr{S}_{n+1}(\l)f, \qquad \mathrm{Re}\l >0\end{split}\end{equation*}
where we recall  that, for $\mathrm{Re}\l >0$, $\|\K\Rs(\l,\A)\|_{\mathscr{B}(\X_{0})} < 1.$ Since moreover, for any $f \in \X_{0}$, the mapping $t \geq 0 \mapsto \bm{S}_{n+1}(t)f$ is continuous and bounded, with $\bm{S}_{n+1}(0)f=0$, one applies the complex Laplace inversion formula \cite[Theorem 4.2.21]{arendt} to deduce 
\begin{equation}\label{eq:maindecay}
\bm{S}_{n+1}(t)f=\frac{\exp(\e t)}{2\pi}\lim_{L\to\infty}\frac{1}{2L}\int_{-L}^{L}\d \ell \int_{-\ell}^{\ell}\exp\left(i\eta t\right)\mathscr{S}_{n+1}(\e+i\eta)f\d\eta, \qquad \forall f \in \X_{0}\end{equation}
for any $t >0$, $\e >0$, i.e. $\bm{S}_{n+1}(t)f$ is the Cesar\`o limit of the family 
$$\left(\int_{-\ell}^{\ell}\exp\left(i\eta t\right)\mathscr{S}_{n+1}(\e+i\eta)f\d\eta\right)_{\ell}.$$ Let us prove it is actually a classical limit. 
Fix $\e >0$ and $f \in\X_{0}$. Recalling the notation,
$$\Ms_{\l}:=\K\Rs(\l,\A), \qquad \mathrm{Re}\l >0,$$
one has, for any $\mathrm{Re}\l >0$
$$\mathscr{S}_{n+1}(\l)= \Rs(\l,\A)\sum_{k=n}^{\infty}\Ms_{\l}^{k+1}=\Rs(\l,\A)\Ms_{\l}^{n+1}\Rs(1,\Ms_{\l})$$
where we notice that, for $\l=\varepsilon+i\eta$,
$$\left\|\Rs(1,\Ms_{\e+i\eta})\right\|_{\mathscr{B}(\X_{0})} \leq \|\Rs(1,\Ms_{\e})\|_{\mathscr{B}(\X_{0})} < \infty\,.$$
Since $\sup_{\eta}\|\Rs(\e+i\eta,\A)\|_{\mathscr{B}(\X_{0})} \leq\frac{1}{\e}$, we deduce that there exists $C_{\e} >0$ such that
$$\left\|\mathscr{S}_{n+1}(\e+i\eta)\right\|_{\mathscr{B}(\X_{0})} \leq C_{\e}\left\|\Ms_{\e+i\eta}^{n+1}\right\|_{\mathscr{B}(\X_{0})}, \qquad \forall \eta \in \R.$$
For $n+1 \geq \mathsf{p}$, one has $\left\| \Ms_{\e+i\eta}^{n+1}\right\|_{\mathscr{B}(\X_{ \X0})} \leq \left\|\Ms_{\e+i\eta}^{\mathsf{p}}\right\|_{\mathscr{B}(\X_{0})}$, we deduce from \eqref{eq:power} that there is $M_{\e} >0$ such that
$$\int_{-\infty}^{\infty}\left\|\mathscr{S}_{n+1}(\e+i\eta)\right\|_{\mathscr{B}(\X_{0})} \d \eta \leq M_{\e}, \qquad \forall \e >0.$$
This of course implies that
$$\int_{-\infty}^{\infty}\left\|\exp\left((\e+i\eta)t\right)\mathscr{S}_{n+1}(\e+i\eta)\right\|_{\mathscr{B}(\X_{0})} \d \eta \leq M_{\e}\exp(\e t), \qquad \forall \e >0.$$
In particular, for any $f \in \X_{0}$, the limit 
$$\lim_{\ell\to\infty}\frac{1}{2\pi}\int_{-\ell}^{\ell}\exp\left((\e+i\eta)t\right)\mathscr{S}_{n+1}(\e+i\eta)f\d\eta$$
exists in $\X_{0}$. Since its Cesar\`o limit is $\bm{S}_{n+1}(t)f$, we deduce the result.
\end{proof}

\section{Properties of the operator $\Ms_{\l}$}\label{appen:Ml}

We collect in this Appendix some technical results regarding the properties of $\Ms_{\l}$. 

\subsection{Proof of Proposition \ref{lem:diffHl}}

We given in this section the full proof of Proposition \ref{lem:diffHl} which regards the norm of derivatives of 
$$\mathsf{L}_{n}(\l)=\Ms_{\l}^{n}=\left[\K\Rs(\l,\A)\right]^{n}, \qquad \forall \l \in \C_{+}, \qquad n \in \N.$$

%\begin{proof}[Proof of Proposition \ref{lem:diffHl}] 
We focus first one the convergence in $\X_{0}$. One checks easily by induction that, for any $f \in \X_{0}$ and any $n \geq 1$,
\begin{multline*}
\mathsf{L}_{n}(\l)f(x,v)=\int_{V^{n}}\bm{k}(v,w_{1})\bm{k}(w,w_{2})\ldots\bm{k}(w_{n-1},w_{n})\bm{m}(\d w_{1})\ldots \bm{m}(\d w_{n})\\
\int_{0}^{\infty}\exp\left(-\sigma(w_{1})t\right)\d t_{1}\ldots \int_{0}^{\infty}\exp\left(-\sigma(w_{n})t_{n}\right)\times\\
\times\exp\left(-\l\sum_{j=1}^{n}t_{j}\right)f\left(x-\sum_{j=1}^{n}t_{j}w_{j},w_{n}\right)\d t_{n},
\end{multline*}
for any $ (x,v) \in \T^{d}\times V, \l \in \C_{+}$
It is clear then that, for any $f \in \X_{p}$
\begin{multline*}
\dfrac{\d^{p}}{\d\l^{p}}\mathsf{L}_{n}(\l)f(x,v)=\left(-1\right)^{p}\int_{V^{n-1}}\bm{k}(v,w_{1})\bm{k}(w,w_{2})\ldots\bm{k}(w_{n-1},w_{n})\bm{m}(\d w_{1})\ldots \bm{m}(\d w_{n})\\
\int_{0}^{\infty}\exp\left(-\sigma(w_{1})t\right)\d t_{1}\ldots \int_{0}^{\infty}\exp\left(-\sigma(w_{n})t_{n}\right)\times\\
\times\left(\sum_{j=1}^{n}t_{j}\right)^{p}\,\exp\left(-\l\sum_{j=1}^{n}t_{j}\right)f\left(x-\sum_{j=1}^{n}t_{j}w_{j},w_{n}\right)\d t_{n}.
\end{multline*}
Set then, for any $f \in \X_{p}$ and any $(x,v) \in \T^{d}\times V$,
\begin{multline}\label{eq:Ms0n}
\mathsf{L}_{n}^{(p)}(0)f(x,v)=(-1)^{p}\int_{V^{n-1}}\bm{k}(v,w_{1})\bm{k}(w,w_{2})\ldots\bm{k}(w_{n-1},w_{n})\bm{m}(\d w_{1})\ldots \bm{m}(\d w_{n})\\
\int_{0}^{\infty}\exp\left(-\sigma(w_{1})t\right)\d t_{1}\ldots \int_{0}^{\infty}\exp\left(-\sigma(w_{n})t_{n}\right)\times\\
\times\left(\sum_{j=1}^{n}t_{j}\right)^{p}\,f\left(x-\sum_{j=1}^{n}t_{j}w_{j},w_{n}\right)\d t_{n}.
\end{multline}
As before, one observes thanks to the change of variable $x \mapsto y=x-\sum_{j=1}^{n}t_{j}w_{j}$ that
\begin{equation}\label{eq:Mn0''}
\left\|\mathsf{L}_{n}^{(p)}(0)f\right\|_{\X_{0}} \leq \int_{\O\times V}\left|f(y,w_{n})\right| \Theta_{n,p}(w_{n})\d y\,\bm{m}(\d w_{n})\end{equation}
where
\begin{multline*}
\Theta_{n,p}(w_{n})=\int_{V^{n}}\bm{k}(v,w_{1})\ldots\bm{k}(w_{n-1},w_{n})\bm{m}(\d w_{n-1})\ldots\bm{m}(\d w_{1})\bm{m}(\d v)\times\\
\times \int_{0}^{\infty}\exp\left(-\sigma(w_{1})t_{1}\right)\d t_{1}\ldots \int_{0}^{\infty}\exp\left(-\sigma(w_{n})t_{n}\right)
\left(\sum_{j=1}^{n}t_{j}\right)^{p}\,\d t_{n}\,.
\end{multline*}
We recall the multinomial formula
$$\left(\sum_{j=1}^{n}t_{j}\right)^{p}=\sum _{|\bm{r}|=p}{p \choose \bm{r}}\prod _{j=1}^{n}t_{j}^{r_{j}}\,,$$
where $\bm{r}=(r_{1},\ldots,r_{n}) \in \N^{n}$ is a multi-index with $|\bm{r}|=\sum_{i=1}^{n}r_{i}=p$ and ${p \choose \bm{r}}=\frac{p!}{r_{1}!\ldots\,r_{n}!}.$ With this, one easily sees that
\begin{equation*}\begin{split}
\int_{0}^{\infty}\exp\left(-\sigma(w_{1})t_{1}\right)&\d t_{1}\ldots \int_{0}^{\infty}\exp\left(-\sigma(w_{n})t_{n}\right)
\left(\sum_{j=1}^{n}t_{j}\right)^{p}\,\d t_{n}\\
&=\sum _{|\bm{r}|=p}{p \choose \bm{r}}\int_{[0,\infty)^{n}}\prod _{j=1}^{n} t_{j}^{r_{j}}\exp\left(-\sigma(w_{j})t_{j}\right)\d t_{1}\ldots \d t_{n}\\
&=\sum _{|\bm{r}|=p}{p \choose \bm{r}}\prod _{j=1}^{n}\left(\frac{1}{\sigma(w_{j})}\right)^{r_{j}+1}.
\end{split}\end{equation*}
Therefore
\begin{multline*}
\Theta_{n,p}(w_{n})
=\sum _{|\bm{r}|=p}{p \choose \bm{r}}\int_{V^{n}}\prod _{j=1}^{n}\left(\frac{1}{\sigma(w_{j})}\right)^{r_{j}+1}\bm{k}(w_{j-1},w_{j})\bm{m}(\d w_{n-1})\ldots\bm{m}(\d w_{1})\bm{m}(\d w_{0})\,.
\end{multline*}
Let $\bm{r} \in \N^{n}$ with $|\bm{r}|=p$ be given. One sees that
\begin{multline*}
\int_{V^{n}}\prod _{j=1}^{n}\left(\frac{1}{\sigma(w_{j})}\right)^{r_{j}+1}\bm{k}(w_{j-1},w_{j})\bm{m}(\d w_{n-1})\ldots\bm{m}(\d w_{1})\bm{m}(\d w_{0})\,\\
=\sigma(w_{n})^{-r_{n}-1}\int_{V}\bm{k}(w_{0},w_{1})\bm{m}(\d w_{1})\int_{V}\bm{k}(w_{1},w_{2})\sigma(w_{1})^{-r_{1}-1}\bm{m}(\d w_{1})\\
\int_{V}\bm{k}(w_{1},w_{2})\sigma(w_{2})^{-r_{2}-1}\bm{m}(\d w_{2})\ldots\int_{V}\bm{k}(w_{n-1},w_{n})\sigma(w_{n-1})^{-r_{n-1}-1}\bm{m}(\d w_{n-1}).
\end{multline*}
Using the definition of $\vartheta_{s}$ in \eqref{eq:varthetas}, one has, since $r_{j} \leq p \leq N_{0}$,
\begin{multline*}
\int_{V}\bm{k}(w_{0},w_{1})\bm{m}(\d w_{1})\int_{V}\bm{k}(w_{1},w_{2})\sigma(w_{1})^{-r_{1}-1}\bm{m}(\d w_{1})\\
=\int_{V}\bm{k}(w_{1},w_{2})\sigma(w_{1})^{-r_{1}}\bm{m}(\d w_{1})=\sigma(w_{2})\vartheta_{r_{1}}(w_{2})\leq \|\vartheta_{r_{1}}\|_{\infty}\sigma(w_{2}).
\end{multline*}
Computing then the integral with respect to $\bm{m}(\d w_{2})$ and then to $\bm{m}(\d w_{j})$ for increasing $j \in \{2,\ldots,n-1\}$, one easily checks that
\begin{multline*}
\int_{V}\bm{k}(w_{0},w_{1})\bm{m}(\d w_{1})\int_{V}\bm{k}(w_{1},w_{2})\sigma(w_{1})^{-r_{1}-1}\bm{m}(\d w_{1})\\
\int_{V}\bm{k}(w_{1},w_{2})\sigma(w_{2})^{-r_{2}-1}\bm{m}(\d w_{2})\ldots\int_{V}\bm{k}(w_{n-1},w_{n})\sigma(w_{n-1})^{-r_{n-1}-1}\bm{m}(\d w_{n-1})\\
\leq \|\vartheta_{r_{1}}\|_{\infty}\|\vartheta_{r_{2}}\|_{\infty}\ldots\|\vartheta_{r_{n-2}}\|_{\infty}\int_{V}\bm{k}(w_{n-1},w_{n})\sigma(w_{n-1})^{-r_{n-1}}\bm{m}(\d w_{n-1})\\
\leq  \|\vartheta_{r_{1}}\|_{\infty}\|\vartheta_{r_{2}}\|_{\infty}\ldots\|\vartheta_{r_{n-1}}\|_{\infty}\sigma(w_{n}).
\end{multline*}
Therefore,
%\begin{multline*}
%\int_{V^{n}}\prod _{j=1}^{n}\left(\frac{1}{\sigma(w_{j})}\right)^{r_{j}+1}\bm{k}(w_{j-1},w_{j})\bm{m}(\d w_{n-1})\ldots\bm{m}(\d w_{1})\bm{m}(\d w_{0})\,\\
%\leq \sigma(w_{n})^{-r_{n}}\prod_{j=1}^{n-1}\|\vartheta_{r_{j}}\|_{\infty}\end{multline*}
%and
\begin{equation*}
\Theta_{n,p}(w_{n})
\leq \sum _{|\bm{r}|=p}{p \choose \bm{r}} \sigma(w_{n})^{-r_{n}}\prod_{j=1}^{n-1}\|\vartheta_{r_{j}}\|_{\infty}\end{equation*}
and
$$\left\|\mathsf{L}_{n}^{(p)}(0)f\right\|_{\X_{0}} \leq  \sum _{|\bm{r}|=p}{p \choose \bm{r}}\prod_{j=1}^{n-1}\|\vartheta_{r_{j}}\|_{\infty}\int_{\O\times V}|f(y,w_{n})|\,\sigma(w_{n})^{-r_{n}}\d y\,\bm{m}(\d w_{n}).$$
This proves that
$$\left\|\mathsf{L}_{n}^{(p)}(0)f\right\|_{\X_{0}} \leq  C_{n,p}\|f\|_{\X_{p}}$$
with
$$C_{n,p}=\sum _{|\bm{r}|=p}{p \choose \bm{r}}\prod_{j=1}^{n-1}\|\vartheta_{r_{j}}\|_{\infty}.$$
As in the proof of the previous Lemma, this allows to prove by a simple use of the dominant convergence theorem that
$$\underset{\l \in \C_{+}}{\lim_{\l\to0}}\left\|\dfrac{\d^{p}}{\d \l^{p}}\mathsf{L}_{n}(\l)f-\mathsf{L}_{n}^{(p)}(0)f\right\|_{\X_{0}}=0$$
for any $f \in \X_{p}$. This proves the convergence in $\X_{0}.$ To deal with the convergence in $\X_{s}$, one simply notices from the above proof that
$$\left\|\mathsf{L}_{n}^{(p)}(0)f\right\|_{\X_{s}} \leq \int_{\O\times V}|f(y,w_{n})|\widetilde{\Theta}_{n,p,s}(w_{n})\d y\,\bm{m}(\d w_{n})$$
with now
\begin{multline*}
\widetilde{\Theta}_{n,p,s}(w_{n})=\int_{V^{n}}\sigma(w_{0})^{-s}\bm{k}(w_{0},w_{1})\ldots\bm{k}(w_{n-1},w_{n})\bm{m}(\d w_{n-1})\ldots\bm{m}(\d w_{1})\bm{m}(\d w_{0})\times\\
\times \int_{0}^{\infty}\exp\left(-\sigma(w_{1})t_{1}\right)\d t_{1}\ldots \int_{0}^{\infty}\exp\left(-\sigma(w_{n})t_{n}\right)
\left(\sum_{j=1}^{n}t_{j}\right)^{p}\,\d t_{n}\,.
\end{multline*}
The only difference between $\widetilde{\Theta}_{n,p,s}(w_{n})$ and $\Theta_{n,p}(w_{n})$ lies in the additional weight $\sigma(w_{0})^{-s}$ in the first integral. But, since
$$\int_{V}\bm{k}(w_{0},w_{1})\sigma(w_{0})^{-s}\bm{m}(\d w_{0}) \leq \sigma(w_{1})\|\vartheta_{s}\|_{\infty}=\|\vartheta_{s}\|_{\infty}\int_{V}\bm{k}(w_{0},w_{1})\bm{m}(\d w_{0})$$
one sees from the above proof that $\left|\widetilde{\Theta}_{n,p,s}(w_{n})\right| \leq \|\vartheta_{s}\|_{\infty}\Theta_{n,p}(w_{n}).$ This gives \eqref{eq:Lnps} and achieves the proof of Proposition \ref{lem:diffHl}.
%\end{proof}

\begin{nb} Notice that the above computations actually apply to the case \emph{without derivatives} (corresponding of course to $p=0$) and, from \eqref{eq:Mn0''}, one sees then that, for any $N \geq 1$, 
$$\sup_{\l\in \overline{\C}_{+}}\left\|\mathsf{L}_{n}(\l)f\right\|_{\X_{s}} \leq \|\mathsf{L}_{n}(0)f\|_{\X_{1}}\leq \int_{\O\times V}\left|f(y,w_{n})\right|\Theta_{n,0,s}(w_{n})\d y\,\bm{m}(\d w_{n})$$
where now
\begin{multline*}
\Theta_{n,0,s}(w_{n})=\int_{V^{n}}\bm{k}(v,w_{1})\ldots\bm{k}(w_{n-1},w_{n})\bm{m}(\d w_{n-1})\ldots\bm{m}(\d w_{1})\frac{\bm{m}(\d v)}{\left(\min(1,\sigma(v)\right)^{s}}\times\\
\times \int_{0}^{\infty}\exp\left(-\sigma(w_{1})t_{1}\right)\d t_{1}\ldots \int_{0}^{\infty}\exp\left(-\sigma(w_{n})t_{n}\right)\\
=\frac{1}{\sigma(w_{n})}\int_{V^{n}}\bm{k}(v,w_{1})\ldots\bm{k}(w_{n-1},w_{n})\frac{\bm{m}(\d w_{n-1})}{\sigma(w_{n-1})}\ldots\frac{\bm{m}(\d w_{1})}{\sigma(w_{1})}\frac{\bm{m}(\d v)}{\left(\min(1,\sigma(v)\right)^{s}}.
\end{multline*}
One has then
$$\Theta_{n,0,s}(w_{n}) \leq \|\vartheta_{s}\|_{\infty} \qquad \forall w_{n} \in V, \qquad s \leq N_{0}.$$
Consequently,
\begin{equation}\label{eq:normLN}
\left\|\mathsf{L}_{n}(\l)\right\|_{\B(\X_{0},\X_{s})} \leq \|\vartheta_{s}\|_{\infty} \qquad \forall s \leq N_{0}
\end{equation}\end{nb}

We also establish the following convergence of derivatives of $\Ms_{\e+i\eta}f$
\begin{lemme}\label{prop:derMeis}
For any $k \geq 0$ and any $\varphi \in \X_{k+1}$ one has
 $$\lim_{\e\to0}\sup_{\eta\in \R}\left\|\dfrac{\d^{k}}{\d\eta^{k}}\Ms_{\varepsilon+i\eta}\varphi-\dfrac{\d^{k}}{\d\eta^{k}}\Ms_{i\eta}\varphi\right\|_{\X_{0}}=0.$$ \end{lemme}
 \begin{proof}For $\varphi \in \X_{k+1}$, $\e >0$, $\eta \in\R$ one checks easily from \eqref{eq:KReieta} that
\begin{multline*}
\dfrac{\d^{k}}{\d\eta^{k}}\Ms_{\varepsilon+i\eta}\varphi(x,v)-\dfrac{\d^{k}}{\d\eta^{k}}\Ms_{i\eta}\varphi(x,v)\\
=(-i)^{k}\int_{V}\bm{k}(v,w)\bm{m}(\d w)
\int_{0}^{\infty}t^{k}\left[\exp\left(-\e t\right)-1\right]\exp\left(-\left(i\eta+\sigma(w)\right)t\right)\varphi(x-tw,w)\d t\end{multline*}
from which we deduce easily that
\begin{multline*}
\left\|\dfrac{\d^{k}}{\d\eta^{k}}\Ms_{\varepsilon+i\eta}\varphi-\dfrac{\d^{k}}{\d\eta^{k}}\Ms_{i\eta}\varphi\right\|_{\X_{0}}
\leq \int_{\T^{d}\times V}\left|\varphi(y,w)\right|\d y \bm{m}(\d w)\int_{V}\bm{k}(v,w)\bm{m}(\d v)\\
\int_{0}^{\infty}t^{k}\left[1-\exp\left(-\e t\right)\right]\exp\left(-t\sigma(w)\right)\d t\,.
%\\
%\leq \e\Gamma(k+2)\int_{\T^{d}\times V}\left|\varphi(y,w)\right|\d y \frac{\bm{m}(\d w)}{\sigma^{k+2}(w)}\int_{V}\bm{k}(v,w)\bm{m}(\d v)
\end{multline*}
%where we used the same argument as in the previous Lemma. Using again \eqref{eq:conservative}, one deduces that
One has
$$\int_{0}^{\infty}t^{k}\left[1-\exp\left(-\e t\right)\right]\exp\left(-t\sigma(w)\right)\d t\leq \int_{0}^{\infty}t^{k}\exp\left(-t\sigma(w)\right)\d t=k!\sigma(w)^{-k-1}$$
so that
$$\int_{V}\bm{k}(v,w)\bm{m}(\d v)
\int_{0}^{\infty}t^{k}\left[1-\exp\left(-\e t\right)\right]\exp\left(-t\sigma(w)\right)\d t \leq k!\sigma(w)^{-k}.$$
Using that
$$\lim_{\e\to0}\int_{V}\bm{k}(v,w)\bm{m}(\d v)
\int_{0}^{\infty}t^{k}\left[1-\exp\left(-\e t\right)\right]\exp\left(-t\sigma(w)\right)\d t=0 \qquad \text{ for a.e. } w \in V$$
and, since $\varphi \in \X_{k}$, we deduce the result from the Lebesgue dominated convergence theorem.
\end{proof}
\begin{nb} Notice that, if $\varphi \in \X_{k+1}$, one sees that
$$\sup_{\eta\in \R}\left\|\dfrac{\d^{k}}{\d\eta^{k}}\Ms_{\varepsilon+i\eta}\varphi-\dfrac{\d^{k}}{\d\eta^{k}}\Ms_{i\eta}\varphi\right\|_{\X_{0}} \leq \e\Gamma(k+2)\|\varphi\|_{\X_{k+1}}$$
making the above convergence quantitative.
 \end{nb}
\subsection{Additional properties}

Introduce now the operators 
$$\mathsf{G}_{n}(\l)=\left[\Rs(\l,\A)\K\right]^{n}, \qquad \forall \l \in \overline{\C}_{+}, \qquad n \in \N.$$
Notice that,  $\mathsf{G}_{n}(\l) \in \B(\X_{0})$ for any $\l \in \overline{\C}_{+}$ since
$$\|\mathsf{G}_{1}(\l)\|_{\B(\X_{0})}=\|\Rs(\l,\A)\K\|_{\B(\X_{0})} \leq \|\K\|_{\B(\X_{0},\X_{1})}\|\Rs(\l,\A)\|_{\B(\X_{1},\X_{0})}$$
with
$$\|\K\|_{\B(\X_{0},\X_{1})}\leq \|\sigma\|_{\infty}\|\vartheta_{1}\|_{\infty}, \qquad \text{ and } \quad \|\Rs(\l,\A)\|_{\B(\X_{1},\X_{0})} \leq 1.$$
Therefore,
$$\|\mathsf{G}_{1}(\l)\|_{\B(\X_{0})} \leq \|\vartheta_{1}\|_{\infty}, \qquad \forall \l \in\overline{\C}_{+}.$$
It is striking to observe that $\|\mathsf{G}_{n}(\l)\|_{\B(\X_{0})}$ can actually be estimated \emph{uniformly with respect to} $n\in \N$. Namely,
\begin{lemme}\label{lem:normGn}
For any $n \in \N$,
$$\left\|\mathsf{G}_{n}(\l)\right\|_{\B(\X_{0})} \leq \|\sigma\|_{\infty}\,\|\vartheta_{1}\|_{\infty} \qquad \forall \l \in \overline{\C}_{+}.$$
In particular, for any $n\in \N,k \in \N$
\begin{equation}\label{eq:Gn+k}
\left\|\mathsf{G}_{n+k}(\l)\right\|_{\B(\X_{0})} \leq \|\sigma\|_{\infty}\,\|\vartheta_{1}\|_{\infty}\left\|\mathsf{G}_{n}\right\|_{\B(\X_{0})}\end{equation}
for any $\l \in\overline{\C}_{+}$.
\end{lemme}
\begin{proof} One simply observes that, for any $n \in \N$,
$$\mathsf{G}_{n}(\l)=\Rs(\l,\A)\mathsf{L}_{n-1}(\l)\K$$
so that
$$\|\mathsf{G}_{n}(\l)\|_{\B(\X_{0})} \leq \|\Rs(\l,\A)\|_{\B(\X_{0},\X_{1})}\|\mathsf{L}_{n-1}(\l)\|_{\B(\X_{0},\X_{1})}\|\K\|_{\B(\X_{0})}$$
where, as known (see \eqref{eq:normLN}),
$$ \|\Rs(\l,\A)\|_{\B(\X_{0},\X_{1})}\leq 1, \qquad \|\mathsf{L}_{n-1}(\l)\|_{\B(\X_{0},\X_{1})}\leq \|\vartheta_{1}\|_{\infty}, \qquad \|\K\|_{\B(\X_{0})}=\|\sigma\|_{\infty}$$
which gives the result. To prove \eqref{eq:Gn+k} one simply observes that
$$\|\mathsf{G}_{n+k}(\l)\|_{\B(\X_{0})}=\|\mathsf{G}_{n}(\l)\mathsf{G}_{k}(\l)\|_{\B(\X_{0})}\leq \|\mathsf{G}_{k}(\l)\|_{\B(\X_{0})}\|\mathsf{G}_{n}(\l)\|_{\B(\X_{0})}$$
and the estimate follows from the first part of the proof.
\end{proof}
We notice also that
$$\mathsf{G}_{1}^{(k)}(\l)=\dfrac{\d^{k}}{\d\l^{k}}\mathsf{G}_{1}(\l)=\left[\dfrac{\d^{k}}{\d\l^{k}}\Rs(\l,\A)\right]\K$$
%=(-1)^{k}\left[\Rs(\l,\A)\right]^{k+1}\K$$
and, for $k \leq N_{0}-1$, 
$$\left\|\left[\dfrac{\d^{k}}{\d\l^{k}}\Rs(\l,\A)\right]\K\right\|_{\B(\X_{0})} \leq \left\|\left[\dfrac{\d^{k}}{\d\l^{k}}\Rs(\l,\A)\right]\right\|_{\B(\X_{k+1},\X_{0})}\|\K\|_{\B(\X_{0},\X_{k+1})}
\leq k! \|\sigma\|_{\infty}\|\vartheta_{k}\|_{\infty}$$
where we used \eqref{eq:dlkRslA}. Thus
\begin{equation}\label{eq:normG1k}
\sup_{\l\in\overline{\C}_{+}}\left\|\mathsf{G}_{1}^{(k)}(\l)\right\|_{\B(\X_{0})} \leq k! \|\sigma\|_{\infty}\,\|\vartheta_{k}\|_{\infty}.\end{equation}
This shows that, for any $f \in \X_{0}$, the mapping
$$\l \in \overline{\C}_{+} \mapsto \mathsf{G}_{1}(\l)f \in \X_{0}$$
is of class $\mathscr{C}^{N_{0}-1}$. This easily extends to $\mathsf{G}_{n}(\l)f$ and, besides the mere estimates for $\mathsf{G}_{n}^{(j)}(\l)$, we also need to understand the decay of $\mathsf{G}_{n}^{(j)}(i\eta)f$ as $|\eta| \to \infty$. We resort to do so to the following  
\begin{lemme}\label{lem:estJ} For any $k \in \{1,\ldots,N_{0}-1\}$, there exists $\bar{C}_{k} >0$ such that
\begin{equation}\label{eq:estJ}
\left\|\mathsf{G}^{(k)}_{N}(i\eta)\right\|_{\mathscr{B}(\X_{0})} \leq \bar{C}_{k}N^{k}\,\|\mathsf{G}_{\floor{\frac{N-k}{2^{k}}}}(i\eta)\|_{\mathscr{B}(\X_{0})} \qquad \forall N \geq 2^{k}+k.\end{equation}
\end{lemme}
\begin{proof} The proof is based upon elementary but tedious computations.  We notice first that,  since $\mathsf{G}_{N}(i\eta)=\left(\mathsf{G}_{1}(i\eta)\right)^{N}$, one has for the first derivative:
$$\mathsf{G}_{N}^{(1)}(i\eta)=\sum_{r=0}^{N-1}\mathsf{G}_{r}(i\eta)\mathsf{G}_{1}^{(1)}(i\eta)\mathsf{G}_{N-1-r}(i\eta)$$
According to \eqref{eq:normG1k}
$$\|\mathsf{G}_{1}^{(1)}(i\eta)\|_{\B(\X_{0})}\leq \|\sigma\|_{\infty}\|\vartheta_{1}\|_{\infty}=C_{1}$$
so that
\begin{equation*}\begin{split}
\left\|\mathsf{G}_{N}^{(1)}(i\eta)\right\|_{\B(\X_{0})} &\leq  C_{1}\sum_{r=0}^{N-1}\left\|\mathsf{G}_{r}(i\eta)\right\|_{\B(\X_{0})} \,\|\mathsf{G}_{N-1-r}(i\eta)\|_{\B(\X_{0})}\\
 &\leq 2C_{1}\sum_{r=0}^{\floor{\frac{N-1}{2}}}
\left\|\mathsf{G}_{r}(i\eta)\right\|_{\B(\X_{0})}\,\left\|\mathsf{G}_{N-1-r}(i\eta)\right\|_{\B(\X_{0})}.\end{split}\end{equation*}
%Assuming $N \geq 2\mathsf{mp}$, one has $N-r \geq \mathsf{2mp}-\floor{\frac{N}{2}}\geq \mathsf{mp}$ for any $0 \leq r \leq \floor{\frac{N}{2}}$, i.e.
Since $N-1-r \geq \floor{\frac{N-1}{2}}$ for any $0\leq r \leq \floor{\frac{N-1}{2}}$, we get thanks to \eqref{eq:Gn+k} that
\begin{equation}\label{eq:j=1}
\left\|\mathsf{G}_{N}^{(1)}(i\eta)\right\|_{\B(\X_{0})}\leq 2C_{1}^{2}\,\left\|\mathsf{G}_{\floor{\frac{N-1}{2}}}(i\eta)\right\|_{\B(\X_{0})}\sum_{r=0}^{\floor{\frac{N-1}{2}}}\left\|\mathsf{G}_{r}(i\eta)\right\|_{\B(\X_{0})}, \qquad N-1 \geq 2\end{equation}
which results in 
\begin{equation*}\label{eq:L11}
\left\|\mathsf{G}_{N}^{(1)}(i\eta)\right\|_{\B(\X_{0})} \leq 2{C_{1}}^{3}(N+1)\left\|\mathsf{G}_{\floor{\frac{N-1}{2}}}(i\eta)\right\|_{\B(\X_{0})}% \leq  \|\vartheta_{1}\|_{\infty}^{2}(N+1)\|\mathsf{L}_{\floor{\frac{N}{2}}}(i\eta)\|_{\B(\X_{0})}
 \end{equation*}
 since $\|\mathsf{G}_{r}(i\eta)\|_{\B(\X_{0})}\leq C_{1}$ for any $r$ and $\floor{\frac{N-1}{2}}+1 \leq N$. This proves the result for $k=1$. Now, for $k=2$ and $N \geq 4$, one has
\begin{multline*}
\mathsf{G}^{(2)}_{N}(i\eta)=\sum_{r=0}^{N-1}\mathsf{G}^{(1)}_{r}(i\eta)\mathsf{G}_{1}^{(1)}(i\eta)\mathsf{G}_{N-r}(i\eta)
+\sum_{r=0}^{N-1}\mathsf{G}_{r}(i\eta)\mathsf{G}_{1}^{(2)}(i\eta)\mathsf{G}_{N-1-r}(i\eta)\\
+\sum_{r=0}^{N-1}\mathsf{G}_{r}(i\eta)\mathsf{G}_{1}^{(1)}(i\eta)\mathsf{G}_{N-1-r}^{(1)}(i\eta).\end{multline*}
Using \eqref{eq:normG1k}, 
$$\|\mathsf{G}_{1}^{(2)}(i\eta)\|_{\B(\X_{0})} \leq 2\|\sigma\|_{\infty}\|\vartheta_{2}\|_{\infty}=:C_{2},$$ while we recall that $\|\mathsf{G}^{(1)}_{1}(i\eta)\|_{\B(\X_{0})} \leq  C_{1},$
so that 
\begin{multline*}
\|\mathsf{G}^{(2)}_{N}(i\eta)\|_{\B(\X_{0})} \leq C_{1}\sum_{r=0}^{N-1}\|\mathsf{G}^{(1)}_{r}(i\eta)\|_{\B(\X_{0})}\|\mathsf{G}_{N-1-r}(i\eta)\|_{\B(\X_{0})}\\
+ C_{1}\sum_{r=0}^{N-1}\|\mathsf{G}_{r}(i\eta)\|_{\B(\X_{0})}\,\|\mathsf{G}_{N-1-r}^{(1)}(i\eta)\|_{\B(\X_{0})}\\
+C_{2}\sum_{r=0}^{N-1}\|\mathsf{G}_{r}(i\eta)\|_{\B(\X_{0})}\|\mathsf{G}_{N-1-r}(i\eta)\|_{\B(\X_{0})}\,,\end{multline*}
and, as before,
\begin{multline*}
\|\mathsf{G}^{(2)}_{N}(i\eta)\|_{\B(\X_{0})} \leq 4C_{1} \sum_{r=0}^{\floor{\frac{N-1}{2}}}\|\mathsf{G}_{r}(i\eta)\|_{\B(\X_{0})}\|\mathsf{G}_{N-1-r}^{(1)}(i\eta)\|_{\B(\X_{0})}\\
+2C_{2}\sum_{r=0}^{\floor{\frac{N-1}{2}}}\|\mathsf{G}_{r}(i\eta)\|_{\B(\X_{0})}\|\mathsf{G}_{N-1-r}(i\eta)\|_{\B(\X_{0})}\,.\end{multline*}
Now,  according to  \eqref{eq:j=1}
$$\|\mathsf{G}_{N-1-r}^{(1)}(i\eta)\|_{\B(\X_{0})} \leq 2C_{1}^{2}\|\mathsf{L}_{\floor{\frac{N-2-r}{2}}}(i\eta)\|_{\B(\X_{0})}\sum_{r_{1}=0}^{\floor{\frac{N-2-r}{2}}}\left\|\mathsf{G}_{r_{1}}(i\eta)\right\|_{\B(\X_{0})}$$
and, since $\floor{\frac{N-2-r}{2}} \geq \floor{\frac{N-2}{4}}$ for any $r \leq \floor{\frac{N-1}{2}}$, invoking \eqref{eq:Gn+k} again we deduce
\begin{equation*} 
\|\mathsf{G}_{N-1-r}^{(1)}(i\eta)\|_{\B(\X_{0})} \leq 2C_{1}^{3}\|\mathsf{L}_{\floor{\frac{N-2}{4}}}(i\eta)\|_{\B(\X_{0})}\sum_{r_{1}=0}^{\floor{\frac{N-2-r}{2}}}\left\|\mathsf{G}_{r_{1}}(i\eta)\right\|_{\B(\X_{0})}\\\end{equation*}
Thus,
\begin{multline*}
 4C_{1} \sum_{r=0}^{\floor{\frac{N-1}{2}}}\|\mathsf{G}_{r}(i\eta)\|_{\B(\X_{0})}\|\mathsf{G}_{N-1-r}^{(1)}(i\eta)\|_{\B(\X_{0})}\\
 \leq 8C_{1}^{4}\|\mathsf{G}_{\floor{\frac{N-2}{4}}}(i\eta)\|_{\B(\X_{0})}
 \sum_{r=0}^{\floor{\frac{N-1}{2}}}\sum_{r_{1}=0}^{\floor{\frac{N-1-r}{2}}}\left\|\mathsf{G}_{r_{1}}(i\eta)\right\|_{\B(\X_{0})}\,,
 \end{multline*} 
while, as in the proof of \eqref{eq:j=1}
\begin{multline*}
2C_{2}\sum_{r=0}^{\floor{\frac{N-1}{2}}}\|\mathsf{G}_{r}(i\eta)\|_{\B(\X_{0})}\,\|\mathsf{G}_{N-1-r}(i\eta)\|_{\B)(\X_{0})}\\
\leq 2C_{2}C_{1}\|\mathsf{G}_{\floor{\frac{N-1}{2}}}(i\eta)\|_{\B(\X_{0})}\sum_{r=0}^{\floor{\frac{N-1}{2}}}\|\mathsf{G}_{r}(i\eta)\|_{\B(\X_{0})}\\
\leq 2C_{2}C_{1}^{2}\|\mathsf{G}_{\floor{\frac{N-2}{4}}}(i\eta)\|_{\B(\X_{0})}\sum_{r=0}^{\floor{\frac{N-1}{2}}}\|\mathsf{G}_{r}(i\eta)\|_{\B(\X_{0})}.
\end{multline*}
  Consequently, 
\begin{multline*}\label{eq:j=2}
\left\|\mathsf{G}_{N}^{(2)}(i\eta)\right\|_{\B(\X_{0})} \leq 2C_{1}^{2}\|\mathsf{G}_{\floor{\frac{N-2}{4}}}(i\eta)\|_{\B(\X_{0})}\times\\
\times \left(4 C_{1}^{2}\sum_{r=0}^{\floor{\frac{N-1}{2}}}\sum_{r_{1}=0}^{\floor{\frac{N-2-r}{2}}}\|\mathsf{G}_{r_{1}}(i\eta)\|_{\B(\X_{0})}
+C_{2}\sum_{r=0}^{\floor{\frac{N-1}{2}}}\|\mathsf{G}_{r}(i\eta)\|_{\B(\X_{0})}\right).\end{multline*}
This clearly gives the rough estimate (using $\floor{\frac{N-2-r}{2}}+1 \leq \floor{\frac{N-1}{2}}+1 \leq N$), 
\begin{equation}
\label{eq:j=2}
\left\|\mathsf{G}_{N}^{(2)}(i\eta)\right\|_{\B(\X_{2})} \leq  2C_{1}^{2}\|\mathsf{G}_{\floor{\frac{N-2}{4}}}(i\eta)\|_{\B(\X_{0})}\left(2C_{1}^{2}N^{2}+C_{2}C_{1}N\right)\end{equation}
and proves the result for $k=2$. By a tedious but simple induction argument, we deduce then the result for any $k \in \{0,\ldots,N_{0}-1\}$. \end{proof}

We can prove now Lemma \ref{prop:snl} given in Section \ref{sec:near0}.

\begin{proof}[Proof of Lemma \ref{prop:snl}] Let $n \in \N$ be fixed. Recall that we defined, for any $k \in \N$, 
$$\mathsf{G}_{k}(\l)=\left(\Rs(\l,\A)\K\right)^{k}, \qquad \l \in \C_{+}$$ 
and we denoted its derivatives of order $j$ by $\mathsf{G}_{k}^{(j)}(\l)$. With this operator, we notices that
$${s}_{n}(\l):=\sum_{k=0}^{n}\Gs_{k}(\l)\Rs(\l,\A)\ \qquad \l \in \overline{\C}_{+}$$

 Computing derivatives with Leibniz rule we get, for any $k \in \N$
\begin{equation}\label{eq:derivSn}
\dfrac{\d^{k}}{\d\l^{k}}s_{n}(\l)f=\sum_{m=0}^{n}\sum_{j=0}^{k}\left(\begin{array}{c}k \\ j\end{array}\right)\Gs^{(j)}_{m}(\l)\dfrac{\d^{k-j}}{\d\l^{k-j}}\left[\Rs(\l,\A)f\right]\end{equation}
Now, as  observed, if $f \in \X_{k-j+1}$, then $\dfrac{\d^{k-j}}{\d\l^{k-j}}\left[\Rs(\l,\A)f\right] \in \X_{0}$ (see \eqref{eq:dlkRslA}) so that (see Eq. \eqref{eq:estJ}) %
$$\Gs^{(j)}_{m}(\l)\dfrac{\d^{k-j}}{\d\l^{k-j}}\left[\Rs(\l,\A)f\right] \in \X_{0}.$$ This easily proves that the mapping
$$\l \in \overline{\C}_{+} \longmapsto s_{n}(\l)f \in \X_{0}$$
is of class $\mathscr{C}^{N_{0}-1}$ with 
$$\sup_{\l \in \C_{+}}\left\|\frac{\d^{k}}{\d\l^{k}}s_{n}(\l)f\right\|_{\X_{0}} \leq C_{k}\|f\|_{\X_{N_{0}}}$$
for some positive $C_{k} >0$ depending only on $k \in \{0,\ldots,N_{0}-1\}.$ Let us now prove 
\begin{equation}\label{eq:snEps}
\lim_{|\eta|\to\infty}\sup_{\e \in (0,1]}\left\|\dfrac{\d^{k}}{\d\eta^{k}}\,s_{n}(\e+i\eta)f\right\|_{\X_{0}}=0\end{equation}
for any $k \in \{0,\ldots,N_{0}-1\}$ which will also prove the fact that the mapping $\eta \mapsto s_{n}(\e+i\eta)f$ belongs to $\mathscr{C}_{0}^{N_{0}-1}(\R,\X_{0})$. Starting from \eqref{eq:derivSn}, we notice that, for any $k \leq N_{0}-1$,  $m \in \{0,\ldots,n\}$, $j \in \{0,\ldots,k\}$, one easily see that, for any $R >0$,
\begin{multline*}
\sup_{\e\in [0,1]}\sup_{|\eta| >R}\left\|\mathsf{G}_{m}^{(j)}(\e+i\eta)\right\|_{\mathscr{B}(\X_{0})} \leq
\sup_{|\eta| >R}\|\mathsf{G}_{m}^{(j)}(i\eta)\|_{\mathscr{B}(\X_{0})} \\
\leq \bar{C}_{k}n^{k}\sup_{|\eta| >R}\|\mathsf{G}_{\floor{\frac{n-k}{2^{k}}}}(i\eta)\|_{\mathscr{B}(\X_{0})}\end{multline*}
according to Lemma \ref{lem:estJ}. from which we easily deduce that
$$\sup_{\e\in [0,1]}\sup_{|\eta| >R}\left\|\mathsf{G}_{m}^{(j)}(\e+i\eta)\right\|_{\mathscr{B}(\X_{0})} < \infty.$$
Combining this with the fact that
$$\lim_{|\eta| \to \infty}\sup_{\e\in [0,1]}\left\|\dfrac{\d^{k-j}}{\d\eta^{k-j}}\Rs(\e+i\eta,\A)f\right\|_{\X_{0}}=0$$
(see Proposition \ref{prop:convRsT0} and Lemma \ref{lem:convDerRsT0}), we deduce \eqref{eq:snEps} from the representation \eqref{eq:derivSn}. The fact that $s_{n}(\e+i\eta)f$ converges to $s_{n}(i\eta)f$ in $\mathscr{C}_{0}^{N_{0}-1}(\R,\X_{0})$ is then deduced from \eqref{eq:derivSn} and the limits established in  Proposition \ref{prop:convRsT0} which easily implies that $\mathsf{G}_{k}(\e+i\eta)$ converges to $\mathsf{G}_{k}(i\eta):=\left[\Rs(i\eta,\A)\K\right]^{k}$ in $\mathscr{C}_{0}^{N_{0}-1}(\R,\X_{0})$ as $\e \to 0$, details are left to the reader.\end{proof}
 
  \section{About collectively power compact operators}\label{sec:functional}

We prove the following result which is likely to be well-known but we were not able to find in the literature.
\begin{theo}\label{theo:powerB} Let $X$ be a Banach space and let $B \in \mathscr{B}(X)$ be such that there is $n \in \N$ such that
$B^{n}$ is compact. Let $\nu_{0}$ be a semi-simple  eigenvalue of $B^{n}$ and let $\mu_{0} \in \mathfrak{S}(B)$ be such that
$$\mu_{0}^{n}=\nu_{0}.$$
Then, $\mu_{0}$ is a semi-simple eigenvalue of $B$. Moreover, if $\nu_{0}$ is simple then so is $\mu_{0}$ and the corresponding spectral projection coincides.\end{theo}

\begin{nb}\label{nb:powerB} Notice that, under the assumption of Theorem \ref{theo:powerB}, since $\nu_{0}$ is a simple eigenvalue of $B^{n}$, there exists only \emph{one} eigenvalue $\mu_{0}$ of $B$ such that
$\nu_{0}=\mu_{0}^{n}.$
%Indeed, if $\mu_{1} \neq \mu_{0}$ is another eigenvalue of $B$ with $\mu_{0}^{n}=\mu_{1}^{n}$, then there are associated to linearly independent eigenfunctions which but all these eigenfunctions would be eigenfunctions of $B^{n}$ which is impossible since $\nu_{0}$ is simple.
\end{nb}
\begin{proof} Observe that for any $\lambda \in \C$,
$$\left(\sum_{j=0}^{n-1}\lambda^{n-1-j}B^{j}\right)\left(\l-B\right)=\l^{n}-B^{n}.$$
In particular, if $\l^{n} \notin \mathfrak{S}(B^{n})$ then $\lambda \notin \mathfrak{S}(B)$ and 
\begin{equation}\label{eq:resolvpuis}
\Rs(\l^{n},B^{n})\sum_{j=0}^{n-1}\lambda^{n-1-j}A^{j}=\Rs(\l,B).\end{equation}
Let us assume now $\mu_{0}^{n}=\nu_{0}$ is semi-simple and let $P_{0}$ be the spectral projection associated to $\nu_{0}$ (and $B^{n}$), then (recall that $\nu_{0}$ is a simple pole of the resolvent $\Rs(\cdot,B^{n})$ 
$$P_{0}=\lim_{z \to \nu_{0}}\left(z-\nu_{0}\right)\Rs(z,B^{n})=\lim_{\l \to \mu_{0}}\left(\l^{n}-\mu_{0}^{n}\right)\Rs(\l^{n},B^{n}).$$
Writing, for $\l^{n} \neq \mu_{0}$,
$$\left(\l-\mu_{0}\right)\Rs(\l,B)=\frac{\l-\mu_{0}}{\l^{n}-\mu_{0}^{n}}\left(\l^{n}-\mu_{0}^{n}\right)\Rs(\l^{n},B^{n})\sum_{j=0}^{n-1}\l^{n-1-j}B^{j}$$
and using that 
$$\lim_{\l\to\mu_{0}}\frac{\l-\mu_{0}}{\l^{n}-\mu_{0}^{n}}=\frac{1}{n\mu_{0}^{n-1}},$$
we deduce that the limit
$$\lim_{\l\to\mu_{0}}\left(\l-\mu_{0}\right)\Rs(\l,B)$$
exists and is equal to 
$$\frac{1}{n\mu_{0}^{n-1}}P_{0}\sum_{j=0}^{n-1}\mu_{0}^{n-1-j}B^{j}.$$
Therefore, $\mu_{0}$ is a simple pole of $\Rs(\cdot,B)$ and
$$\Pi_{0}=\frac{1}{n\mu_{0}^{n-1}}P_{0}\sum_{j=0}^{n-1}\mu_{0}^{n-1-j}B^{j}.$$
In particular, $\mathrm{Range}(\Pi_{0}) \subset \mathrm{Range}(P_{0})$ and $\dim\left(\mathrm{Range}(\Pi_{0})\right) \leq \dim\left(\mathrm{Range}(P_{0})\right).$ In particular, if $\nu_{0}$ is simple then so is $\mu_{0}$ with
$$\mathrm{Range}(\Pi_{0})=\mathrm{Range}(P_{0}).$$
Let now $\psi \in X$ be an eigenfunction of $B$ associated to $\mu_{0}$ and let $\psi^{*} \in X^{*}$ be an eigenfunction of $B^{*}$ associated to $\mu_{0}$ with the normalisation
$$\langle \psi^{*},\psi\rangle=1.$$
Then, $\mathrm{Range}(P_{0})=\mathrm{Range}(\Pi_{0})=\mathrm{Span}(\psi)$ and for any $h \in X$, there is $\alpha_{h} \in \R$ such that
$$\Pi_{0}h=\alpha_{h}\psi.$$
%Of course, 
%$$B\Pi_{0}h=\Pi_{0}Bh=\alpha_{h}B\psi=\mu_{0}\alpha_{h}\psi=\mu_{0}\Pi_{0}h.$$
One has then $\langle \psi^{*},\Pi_{0}h\rangle=\alpha_{h}$ i.e.
$$\alpha_{h}=\langle \Pi_{0}^{*}\psi^{*},h\rangle=\langle \psi^{*},h\rangle$$
i.e.
$$\Pi_{0}h=\langle \psi^{*},h\rangle\,\psi.$$
In the same way, $P_{0}h=\langle \psi^{*},h\rangle\,\psi$ and $\Pi_{0}=P_{0}.$
\end{proof}

\end{document}